\documentclass[11pt]{amsart}
\usepackage{fancyhdr}
\usepackage[english]{babel}
\usepackage{amssymb}
\usepackage{mathrsfs}
\usepackage{enumerate}
\usepackage{color}
\usepackage{ifpdf}
\usepackage{tikz}
\usepackage{tikz-cd}
\usetikzlibrary{decorations.pathmorphing,arrows}
\usepackage[initials]{amsrefs}

\usepackage{color}

\newcommand{\bbN}{{\mathbb N}}
\newcommand{\bbR}{{\mathbb R}}

\newcommand{\bbZ}{{\mathbb Z}}
\newcommand{\bbC}{{\mathbb C}}
\newcommand{\bbE}{\mathbb{E}}

\newcommand{\calA}{\mathcal{A}}
\newcommand{\calB}{\mathcal{B}}
\newcommand{\calC}{\mathcal{C}}

\newcommand{\calF}{\mathcal{F}}

\newcommand{\calV}{\mathcal{V}}

\newcommand{\calX}{\mathcal{X}}
\newcommand{\calY}{\mathcal{Y}}
\newcommand{\calZ}{\mathcal{Z}}

\newcommand{\forae}{\forall^{\textrm{ae}}}

\newcommand{\id}{\operatorname{id}}

\newcommand{\pr}{\operatorname{pr}}

\newcommand{\supp}{\operatorname{supp}}

\newcommand{\Isom}{\operatorname{Isom}}

\newcommand{\Ad}{\operatorname{Ad}}

\newcommand{\Prob}{\operatorname{Prob}}

\newcommand{\half}{\frac{1}{2}}

\newcommand{\diam}{\operatorname{diam}}

\newcommand{\bary}{\operatorname{bar}}

\newcommand{\Aff}{\operatorname{Aff}}

\newcommand{\Hom}{\operatorname{Hom}}

\newcommand{\Var}{\operatorname{Var}}

\newcommand{\Homeo}{\operatorname{Homeo}}
\newcommand{\Stab}{\operatorname{Stab}}
\newcommand{\wlong}{\operatorname{w}_{\operatorname{long}}}

\newcommand{\SL}{\operatorname{SL}}

\newcommand{\GL}{\operatorname{GL}}
\newcommand{\PGL}{\operatorname{PGL}}

\newcommand{\dd }{\,{\rm d}}

\newcommand{\normal}{\triangleleft}

\newcommand{\acts}{\curvearrowright}
\newcommand{\overto}[1]{\,{\buildrel{#1}\over\longrightarrow}\,}

\newcommand{\setdef}[2]{ \left\{\left. {#1}\ \right|\ {#2} \right\} }

\newcommand{\wt}[1]{\widetilde{#1}}
\newcommand{\wh}[1]{\widehat{#1}}
\newcommand{\ol}[1]{\overline{#1}}
\newcommand{\ev}{\operatorname{ev}}
\newcommand{\maxev}{\operatorname{max-ev}}

\newcommand{\maxpart}{\operatorname{max-part}}

\newcommand{\ext}{\operatorname{ext}}
\newcommand{\Map}{\operatorname{Map}}
\newcommand{\Stat}{\operatorname{Map}^{\operatorname{stat}}}

\newtheorem{mthm}{Theorem}

\newtheorem{theorem}{Theorem}[section]
\newtheorem{lemma}[theorem]{Lemma}

\newtheorem*{claim}{Claim}

\newtheorem{cor}[theorem]{Corollary}

\newtheorem{prop}[theorem]{Proposition}

\theoremstyle{definition}

\newtheorem{defn}[theorem]{Definition}

\newtheorem{example}[theorem]{Example}

\newtheorem{remark}[theorem]{Remark}

\numberwithin{equation}{section}

\newcommand{\newcomment}[4]{%
\newcounter{#2counter}
\expandafter\newcommand\csname #1\endcsname[1]{%
\refstepcounter{#2counter}%
{\color{#4}(#3\arabic{#2counter})}\marginpar{\scriptsize\raggedright\textbf{\color{#4}(#2 \arabic{#2counter}):} ##1}%
}}
\definecolor{darkgreen}{rgb}{0,0.6,0}
\newcomment{ub}{Uri}{U}{darkgreen}
\newcomment{af}{Alex}{A}{blue}

\usepackage{subfiles} 

\title[Simplicity of the Spectrum]{Lyapunov spectrum via boundary theory I - framework}

\author{Uri Bader}
\address{Weizmann Institute}
\email{bader@weizmann.ac.il}

\author{Alex Furman}
\address{University of Illinois at Chicago}
\email{furman@uic.edu}


\begin{document}
\date{\today}

\begin{abstract}
    This paper is concerned with the Lyapunov spectrum for measurable cocycles
    over an ergodic pmp system taking values in semi-simple real Lie groups. 
    We prove simplicity of the Lyapunov spectrum and its continuity
    under certain perturbations for a class systems that includes many 
    familiar examples. 
    
    Our framework uses some soft qualitative assumptions, and does not
    rely on symbolic dynamics.
    We use ideas from boundary theory
    that appear in the study of super-rigidity to deduce our results.
    This gives a new perspective even on the most studied case of random 
    matrix products.

    The current paper introduces the general framework and contains the proofs
    of the main results and some basic examples. 
    In a follow up paper we discuss further examples.
\end{abstract}

\maketitle

\tableofcontents

\section{Introduction}\label{sec:intro}

Let $(X,\calX,m)$ be a standard probability space, $T:X\to X$ 
an ergodic measure preserving transformation, 
and $F:X\to \GL_d(\bbR)$ 
a measurable map satisfying the integrability condition
\[
    \int_X \log^+\|F(x)\|\dd m(x)<+\infty,\qquad 
    \int_X \log^+\|F(x)^{-1}\|\dd m(x)<+\infty,
\]
where $\log^+(t)=\max(\log(t),0)$
and $\|\cdot\|$ denotes the operator norm.
Let $F_n:X\to \GL_d(\bbR)$ be the cocycle generated by $F$, namely 
\[
    F_n(x)=F(T^{n-1}x)\cdots F(Tx)F(x),
\] 
and define the corresponding \textbf{Lyapunov exponents} 
$\lambda_1\ge \lambda_2\ge \dots\ge \lambda_d$ by 
\[
	\begin{split}
			&\lambda_1=\inf_{n\ge 1}\ \frac{1}{n}\int_X \log\|F_n(x)\|\dd m(x),\\
			&\lambda_1+\dots+\lambda_k=\inf_{n\ge 1}\ \frac{1}{n}\int_X \log\|\wedge^k F_n(x)\|\dd m(x)
			\qquad (1\le k\le d).
	\end{split}
\]
These Lyapunov exponents can be recorded as an element $\Lambda_F$ in the positive Weyl chamber $\mathfrak{a}^+$
of the Cartan sub-algebra $\mathfrak{a}$, consisting of diagonal elements
in the Lie algebra $\mathfrak{gl}_d$ of the reductive group $G=\GL_d(\bbR)$:
\[
	\Lambda_F=\operatorname{diag}(\lambda_1,\dots,\lambda_d)\ \in \mathfrak{a}^+.
\] 
Oseledts' Multiplicative Ergodic Theorem describes asymptotic behavior of $F_n(x)$
and its action on $\bbR^d$ in terms of $\exp(n\Lambda_F+o_x(n))$.

More generally, given a reductive real Lie group $G$, one considers the asymptotic 
behavior of the cocycle $\{F_n(-)\}$ generated by a measurable map $F:X\to G$.
Since the effect of the center can be easily analyzed separately, we will hereafter focus
on the case were
$G$ is a {\em connected semisimple real Lie group with a finite center}.
Fix a real representation with finite kernel $\tau:G\to\GL(V)$
and an inner product on $V$ to obtain an operator norm, 
which we denote $\|\cdot\|$, on $\GL(V)$.
We define a (pseudo-)length function on $G$ by
\[ 
    |\cdot|: G\to [0,\infty), \quad |g|=\log(\max\{\|\tau(g)\|,\|\tau(g^{-1})\|\}). 
\]
The integrability condition can now be stated as: 
\begin{equation}\label{e:FL1}
	\int_X |F(x)|\dd m(x)<+\infty.
\end{equation}
This condition is independent of the choices made for the representation 
$\tau$ and the inner product on $V$,
see Lemma~\ref{l:repnorm} below.

Fix a Cartan subalgebra $\mathfrak{a}<\mathfrak{g}$ in the Lie algebra of $G$, 
a Weyl chamber $\mathfrak{a}^+\subset \mathfrak{a}$, 
and a maximal compact subgroup $K<G$. We let 
$\kappa:G\to \mathfrak{a}^+$ be the corresponding Cartan projection; so $g\in K\exp(\kappa(g)) K$ for every $g\in G$.
The theorem of Oseledets describes the 
asymptotic behavior of $F_n$; in particular (see Corollary \ref{C:Oseledets}) it asserts
an almost everywhere and an $L^1$-convergence of the normalized Cartan projections 
\begin{equation}\label{e:spectrum}
	\Lambda_F:=\lim_{n\to\infty} \frac{1}{n} \kappa(F_n(x)) \in \mathfrak{a}^+.
\end{equation}
The element $\Lambda_F \in \mathfrak{a}^+$ is said to be the \textbf{Lyapunov spectrum}
associated with $F$.
If $\Lambda_F=0$ we say that the Lyapunov spectrum is \textbf{trivial}; 
this case corresponds
to sub-linear growth:  $|F_n(x)|=o_x(n)$.
Otherwise, if $\Lambda_F\ne 0$, we say that the Lyapunov spectrum is \textbf{non-trivial} or that $F$ has a \textbf{positive drift}.
We are particularly interested in the situations where the spectrum is not merely non-trivial, but is  
simple in the following sense.
\begin{defn}
We say that $F$ has a \textbf{simple Lyapunov spectrum} if $\Lambda_F$ given by
(\ref{e:spectrum}) is in the interior $\mathfrak{a}^{++}$ 
of the positive Weyl chamber $\mathfrak{a}^+\subset \mathfrak{a}$.  
\end{defn}
If $G$ is a simple real Lie group of rank one then $\dim\mathfrak{a}=1$, and the simplicity of the spectrum is equivalent to positivity of the drift. 
However, if $G$ is of higher rank the simplicity of the spectrum is a much more subtle condition than positivity of the the drift.
For example:
\begin{itemize}
	\item 
	In case of $G=\SL_d(\bbR)$ we have $\Lambda_F=\operatorname{diag}(\lambda_1,\dots,\lambda_d)$ 
	where $\lambda_1\ge \dots\ge \lambda_d$ and $\sum \lambda_i=0$. 
	The condition $\Lambda_F\ne 0$ is equivalent to $\lambda_1>0$,
	while simplicity of $\Lambda_F$ refers to the strict inequalities $\lambda_1>\lambda_2>\dots>\lambda_d$.
	\item 
	In case of $G=\operatorname{Sp}_{2d}(\bbR)$ realized as a subgroup of $\SL_{2d}(\bbR)$, 
	the Lyapunov spectrum has the form 
	$\Lambda_F=\operatorname{diag}(\lambda_1,\dots,\lambda_d,-\lambda_d,\dots,-\lambda_1)$
	with $\lambda_1\ge \lambda_2\ge \dots\ge \lambda_{d}\ge 0$.
	The condition $\Lambda_F\ne 0$ is still equivalent to $\lambda_1>0$,
	while simplicity of $\Lambda_F$ refers to the strict inequalities
	$\lambda_1>\lambda_2>\dots>\lambda_d>0$.
\end{itemize}

%

The purpose of this paper is to describe a rather general 
setting in which the Lyapunov spectrum can
be proved to be simple, and to depend continuously on certain deformations of $F$.
Our main theorem is Theorem~\ref{T:main}.
In order to state it, we need to introduce our general framework
which includes an auxiliary group $\Gamma$.

\subsection{Statement of the main results}

\begin{defn} \label{defn:Greg}\hfill{}\\
By a \textbf{Group Random Element Generator}, or \textbf{Greg} for short,
we mean a tuple $(X,\calX,m,T,w,\Gamma)$,
where $(X,\calX,m)$ is a standard probability space, 
$T:X\to X$ is an invertible measure preserving transformation,
$\Gamma$ is a locally compact second countable group 
and $w:X\to \Gamma$ is a measurable map. 
\end{defn}

A $\mu$-random sequence $\{w(T^nx)\}_{n\in\bbZ}$ can be viewed as a 
(bi-infinite) sequence of random, not necessarily independent, 
$\Gamma$-valued elements.
Taking consecutive products of these sequences, $w_n(x)=w(T^{n-1}x)\cdots w(Tx)w(x)$, 
we obtain a cocycle $w:\bbZ\times X \to\Gamma$ defined by
\[ 
    w_1(x)=w(x),\qquad w_{n+k}(x)=w_k(T^nx)w_n(x)\qquad (n,k\in\bbZ).
\] 
(see Equation~\eqref{e:f-cocycle}).
Consider the map
\[
	j:X\times \Gamma\overto{}  \Gamma^\bbZ,\qquad j(x,g)=(w_n(x)g^{-1})_{n\in\bbZ}
\]
and let $\wt{m}$ denote the measure on $\Gamma^\bbZ$ obtained 
by the push-forward of $m\times m_\Gamma$, 
$\wt{m}=j_*(m\times m_\Gamma)$, where $m_\Gamma$ is 
a Haar measure on $\Gamma$. It will be more convenient to work
with a probability measure $\wt{m}^1=j_*(m\times m_\Gamma^1)$
where $m_\Gamma^1$ is some choice of a probability measure
in the Haar measure class on $\Gamma$.
The measure $\wt{m}^1$ provides a notion of a random sequence of group elements in $\Gamma$.
We think of $j$ as an $(X,m,T)$-\emph{driven random walk} 
with $j(x,g)$ passing through $g^{-1}$ at time $n=0$.
We denote by $\wt{m}^1_+$ and $\wt{m}^1_-$
the marginal distributions of $\wt{m}^1$ on the space of futures $\Gamma^{\bbZ_+}$ 
and the space of pasts $\Gamma^{\bbZ_-}$, respectively.
Taking the ergodic components for the natural non-singular actions 
of the semi-groups $\bbZ_+$ and $\bbZ_-$ on these spaces
we obtain the spaces $(E_+,\eta_+)$ and $(E_-,\eta_-)$
of \textit{ideal futures} and \textit{ideal pasts} as measurable quotients
\[
	\beta_+:(\Gamma^\bbZ,\wt{m})\ \to \ (E_+,\eta_+),\qquad
	\beta_-:(\Gamma^\bbZ,\wt{m})\ \to \ (E_-,\eta_-).
\]
\begin{defn}\label{D:Apafi}\hfill{}\\
    The Greg system $(X,\calX,m,T,\Gamma,w)$ is said to
    satisfy the \textbf{asymptotic past and future independence condition}  
    if on the product space $E_+\times E_-$, 
    the pushed measure $(\beta_+\times \beta_-)_*\wt{m}^1$ is equivalent 
    to the product measure $\eta_+\times\eta_-$.
    In this case we will say that the Greg satisfies the \textbf{Apafi condition}, 
    or that it is an \textbf{Apafic Greg}.
\end{defn}

Hereafter we will assume that $(X,\calX,m,T,w,\Gamma)$ is an Apafic Greg.
This assumption has several consequences, including ergodicity of the 
system $(X,\calX,m,T)$.

Consider a representation $\rho:\Gamma\overto{} G$ into a connected
semi-simple real Lie group $G$ with finite center, and form the $G$-valued
measurable map 
\[
    F=\rho\circ w\ :X\overto{w} \Gamma\overto{\rho} G.
\]
Assuming $F=\rho\circ w$ satisfies integrability condition (\ref{e:FL1}),
we can associate the Lyapunov spectrum $\Lambda_{\rho\circ w}\in\mathfrak{a}^+$ 
to $\rho\in \Hom(\Gamma,G)$. 

Connected semi-simple real Lie group $G$ with finite center can be equipped 
with Zariski topology, when viewed as real points of a real algebraic group.
A representation $\rho\in \Hom(\Gamma,G)$ is called \textbf{Zariski dense} 
if $\rho(\Gamma)$ is Zariski dense in $G$,
and it is called \textbf{integrable} if $F=\rho\circ w$ satisfies (\ref{e:FL1}).
We endow the space $\Hom(\Gamma,G)$ with the compact-open topology.
A subset $\Sigma\subset \Hom(\Gamma,G)$ is said to be \textbf{uniformly integrable} 
if for every $\epsilon>0$ there exists $\delta>0$ such that for every $\rho\in \Sigma$ and measurable $E\subset X$:
\[ 	
	m(E)<\delta
	\qquad\Longrightarrow\qquad
	\int_{E} |\rho\circ w(x)|\dd m(x)<\epsilon.
\]
For example, for any $h\in L^1(X,\calX,m)$ the set of representations
\begin{equation} \label{eq:hunif}
	\Sigma_h=\setdef{\rho\in\Hom(\Gamma,G)}{|\rho\circ w(x)|\le h(x)}
\end{equation}
is uniformly integrable.

\begin{mthm}\label{T:main}\hfill{}\\
        Let $(X,\calX,m,T,\Gamma,w)$ be an Apafic Greg as in 
        Definitions~\ref{defn:Greg} and \ref{D:Apafi}.
	Let $G$ be a connected semisimple real Lie group with a finite center.
	Then
	\begin{enumerate}
		\item 
		For every integrable  Zariski dense representation $\rho:\Gamma\to G$, 
		the map $F=\rho\circ w:X\to  G$ has a simple spectrum, 
        i.e. $\Lambda_F\in \mathfrak{a}^{++}$.
		\item 
		For every uniformly integrable subset $\Sigma\subset\Hom(\Gamma,G)$,
		the Lyapunov spectrum map $\Lambda:\Sigma\to  \mathfrak{a}^+$ is continuous 
		at every Zariski dense representation $\rho\in \Sigma$.
		\item 
		For every integrable representation $\rho\in \Hom(\Gamma,G)$
		such that $\rho(\Gamma)$ is not contained in a compact extension 
        of a solvable subgroup, 
		the map $F=\rho\circ w:X\to  G$ has positive drift: $\Lambda_F\neq 0$.
	\end{enumerate}
\end{mthm}

We will describe a number of examples of Apafic Greg systems in a companion paper
\cite{BF:examples} where we treat, among other examples, Kontsevich--Zorich cocycles. 
Here let us just point out a few more basic cases.

\bigskip

\subsection{Random walks with i.i.d. increments}\hfill{}\\
Let $\mu$ be a \textit{spread-out, aperiodic} measure on a lcsc group $\Gamma$
(i.e. assume that some convolution power $\mu^{*n}$ being absolutely continuous with respect 
to the Haar measure on $\Gamma$ and that $\supp(\mu)$ is not contained 
in a coset of a proper closed subgroup of $\Gamma$).
For example, consider the case of a countable discrete group $\Gamma$
with a probability measure $\mu$ that is not supported on a coset of a proper subgroup.
Then the Bernoulli system $X=\Gamma^\bbZ$ with $m=\mu^\bbZ$ and $w:X\to\Gamma$ given by $w(x)=x_1$
is an example of an Apafic Gregs with 
$E_+$ and $E_-$ being the Furstenberg-Poisson boundaries associated and $\mu$ and 
$\check\mu$ - the reversed random walk (namely $\dd\check\mu(\gamma)=\dd\mu(\gamma^{-1})$).

In this setting Theorem~\ref{T:main} allows to recover (most cases of) by now classical 
results of Furstenberg \cite{Furstenberg-Poisson} about positivity of the drift, 
Guivarc'h--Raugi \cites{GR, GR2} and Gol'dsheid--Margulis \cite{GM} about simplicity
of the Lyapunov spectrum, and of LePage \cite{LePage} 
about continuity of the Lyapunov spectrum 
for random walks given by the probability measure $\rho_*\mu$ on $G$.

\bigskip

\subsection{Geodesic flow in negative curvature}\label{subsec:ncflow}\hfill{}\\
Let $(M,g)$ be a closed Riemannian manifold of strictly negative curvature, 
$\phi^\bbR$ the geodesic flow on the unit tangent bundle $T^1M$.
Among the $\phi^\bbR$-invariant probability measures on $T^1M$,
one notes Bowen--Margulis measure $\mu^{\operatorname{BM}}$ 
-- the unique measure of maximal entropy,
and Lebesgue--Liouville measure $\mu^{\operatorname{LL}}$
obtained from the Riemannian volume on
$M$ with a uniform measure on the tangent spheres. 
These are examples of so called Gibbs measures (see \cite{}).

Recall that there is a canonical cohomology class of measurable cocycles
\[
    c:\phi^\bbR\times T^1M \overto{} \pi_1(M)
\]
obtained by choosing a measurable fundamental domain $\mathscr{F}$ 
for the action of $\pi_1(M)$ on $T^1\wt{M}$, 
and setting $c(t,x)=\gamma$ if $\wt{\phi}^t(x)\in\gamma\mathscr{F}$, 
where $\wt{M}$ is the universal cover of $M$ and
$\wt{\phi}^\bbR$ is the geodesic flow on $T^1\wt{M}$.

\begin{theorem}\label{T:ncm-flow}\hfill{}\\
    Let $(M,g)$ be a negatively curved compact manifold,
    $\phi^\bbR$ the geodesic flow, $\mathfrak{m}$ a Gibbs measure on $T^1M$
    (e.g. $\mu^{\operatorname{BM}}$ or $\mu^{\operatorname{LL}}$).
    Let $\mathscr{F}\subset T^1\wt{M}$ be a fundamental domain for the $\pi_1(M)$-action, and assume it is bounded and with a non-empty interior. Let $c$ be the associated cocycle as defined above.
    Let $G$ be a semi-simple real Lie group and
    \[
        \rho:\pi_1(M)\overto{} G
    \]
    a representation with Zariski-dense image.
    Then the $\bbR$-cocycle $F_t:X\overto{} G$, given by
    $F_t(x)=\rho\circ c(t,x)$,
    has a simple Lyapunov spectrum 
    $\Lambda_{\rho\circ c}\in\mathfrak{a}^{++}$,
    and the latter depends continuously 
    on $\rho$ in $\Hom(\pi_1(M),G)$.
\end{theorem}

Here we used the $\bbR$-flow version of the Oseledets theorem 
to define $\Lambda_F$ via $\mathfrak{m}$-a.e. and $L^1$ convergence
\[
    \Lambda=\lim_{t\to\infty} \frac{1}{t}\kappa(F_t(x)).
\]
Since in this case the flow is mixing, the spectrum could also be defined
via a.e. convergence of the time-one map $F_1(x)=\rho\circ c_\mathscr{F}(1,x)$.

\bigskip

\subsection{Some Background and Context}\hfill{}\\ 
The study of Lyapunov spectrum has a long history and has been most
successful in the context of random walks (products of i.i.d. elements) 
on semi-simple real Lie groups.
The pioneering work of Furstenberg \cite{Furstenberg-Poisson} 
gave, among other things, the condition for positivity of the drift.
The concept of stationary measure on the flag variety $G/P$
played an important role in this analysis. 
Guivarc'h--Raugi \cites{GR, GR2} developed a sufficient condition
for simplicity of the Lyapunov spectrum
using, what they called, \textit{contracting property} 
for the closed \textit{semi-group} generated by $\supp(\mu)$.  
Further work of Gol'dsheid--Margulis \cite{GM} showed that
the contraction property  holds if $\supp(\mu)$ generates 
a \textit{Zariski dense subgroup} in $G$.
Hence, if $\supp(\mu)$ generates a Zariski dense subgroup in $G$,
and $\mu$ is integrable, then the Lyapunov spectrum is simple.
In a different direction, LePage \cite{LePage} studied the regularity of
the Lyapunov spectrum under perturbations of the law $\mu$ 
(see also Furstenberg--Kifer \cite{Fur-Kif},
and the recent Avila--Eskin--Viana \cite{AEV}).  
For more background and references see \cite{BL}, \cite{F:RT}, 
\cite{BQ:book}, \cite{Viana}. 

The breakthrough work of Avila--Viana \cite{AV2} (see also \cites{AV3, AV}),
gave a sufficient condition for simplicity of the Lyapunov spectrum
for matrix cocycles over dynamical systems under certain conditions.
This criterion in particular allowed to prove the conjectured simplicity
of the Lyapunov spectrum for Kontsevich--Zorich cocycle (\cite{AV3}). 

Our proof follows a different approach and gives a different
framework while covering some of the same examples.
It should be pointed out that our approach passes through
some ideas of Boundary Theory, that were inspired (in part) 
by the study of random walks. 
While our approach allows to recover Furstenberg, Guivarc'h--Raugi,
Gol'dsheid--Margulis results for random walks with
Zariski-dense distribution, we follow a different path
while using similar objects and constructions
such as stationary measures.  

\bigskip

\subsection{Organization of the paper}\hfill{}\\ 
Section \S\ref{sec:preliminaries} starts by recalling some ergodic-theoretic 
notions, notations, and basic facts that are used in the paper.
We also introduces such notions as metric ergodicity and 
relative metric ergodicity, and prove some new facts, such as 
Proposition~\ref{P:natext-relmeterg}.

Section \S\ref{sec:gregs} develops the framework of Gregs and some associated basic constructions.
In particular, 
the general notion of \textit{stationary maps} is introduced.
In this general context stationary maps take values 
in a compact convex affine $\Gamma$-space, which we will
later take to be $\Prob(G/P)$.

Section \S\ref{sec:ApaficGregs} focuses on Apafic Gregs. 
One of its goals is to analyze the situation
where stationary map takes extreme point of the affine space.
In a later application to $\Prob(G/P)$ this will be used to 
show that the stationary measures are \textit{proper}, i.e.
have no atoms and give zero weight to proper subspaces. 

Section \S\ref{sec:semisimple} discusses semi-simple Lie groups.
After setting the notations and some preliminaries, 
we discuss the notion of AREA (Algebraic Representations of Ergodic Actions)
and relate it to \textit{boundary systems} by proving the key statement -- Corollary~\ref{C:boundary}. 

Section \S\ref{sec:proofs} contains the proofs of the main result, 
Theorem~\ref{T:main}.

Section \S\ref{sec:realgregs} discusses the notion of a \emph{real Greg}, an analogue of a Greg with a real time parameter.

Section \S\ref{sec:geo-example} discusses negatively curved dynamics and contains the 
proof of Theorem~\ref{T:ncm-flow}.

\bigskip

\subsection*{Acknowledgments}\hfill{}\\ 
The authors would like to acknowledge grant support.
U.B. and A.F. were supported in part by the BSF grant 2008267.
U.B was supported in part by the ISF grant 704/08.
A.F. was supported in part by the NSF grants DMS 2005493.

\newpage
\section{Preliminaries}\label{sec:preliminaries}

In this section we set our notations, and present some facts and constructions 
in a form which is convenient for our later developments.
Most of the results mentioned here are known, but some, e.g
Proposition~\ref{P:natext-relmeterg}, are new.

\subsection{Lebesgue spaces and non-singular actions}
\label{sub:Lebesgue}\hfill{}\\ 
A topological space is said to be Polish if it admits a countable dense subspace and the topology is generated by a complete metric.
A Borel space (or a measurable space) is a set endowed with a $\sigma$-algebra.
A standard Borel space is a Borel space which is isomorphic as such to a Polish space.
A Lebesgue space is a standard Borel space which is endowed with a measure class.

Let $(X,\calX,m)$ be a Lebesgue space, and $(V,\calV)$ be a standard Borel space. 
We use the term \textbf{map} for an equivalence class, up to null sets, 
of measurable functions $f:X\to V$, and
denote by $\Map(X,Y)$ the space of such maps:
\[
	\Map(X,V)=\setdef{f:X\overto{} V}{\forall\ E\in \calV,\ f^{-1}(E)\in\calX}/ \sim
\]
where $f_1\sim f_2$ if $m(\setdef{x\in X}{f_1(x)\ne f_2(x)})=0$.

If $V$ is a Polish space, then $\Map(X,V)$ is endowed with a natural Polish structure as well. 
For Lebesgue spaces $X,Y$ we have, by Fubini argument, the standard identification of Polish spaces
\begin{equation} \label{eq:fubini}
	\Map(X\times Y,V)\simeq \Map\left(X,\Map(Y,V)\right).
\end{equation}

Let $p:(X,\calX,m)\to (Y,\calY,n)$ be a map of probability spaces with $p_*m=n$.
The \textbf{disintegration} of $m$ with respect to $n$, is the unique map 
$Y\to\Prob(X,\calX)$, $y\mapsto \mu_y$,
so that 
\[
	m=\int_Y \mu_y\dd n(y)\qquad\textrm{and}\qquad\forae\, y\in Y:\quad \mu_y\left(p^{-1}(\{y\})\right)=1.
\]

Let $\Gamma$ be a locally compact secondly countable group and $(X,\calX,m)$ be a Lebesgue space. 
A \textbf{non-singular action} $\Gamma\acts (X,\calX,m)$ is a measurable map $a:\Gamma\times X\to X$ satisfying a.e 
$a(1,x)=x$ and $a(\gamma_1\cdot\gamma_2,x)=a(\gamma_1,a(\gamma_2,x))$ and such that for every $\gamma\in\Gamma$
the measures $\gamma_*m$ and $m$ are equivalent, i.e. have the same $\sigma$-ideal of null sets.
A non-singular action $\Gamma\acts (X,\calX, m)$ is \textbf{ergodic} if $E\in\calX$ with $m(\gamma E\triangle E)=0$
for all $\gamma\in\Gamma$, satisfies $m(E)=0$ or $m(X\setminus E)=0$.

An action $\Gamma\acts (X,\calX, m)$ is \textbf{measure-preserving} if $\gamma_*m=m$ for all $\gamma\in\Gamma$.
We say that the action $\Gamma\acts (X,\calX,m)$ is \textbf{probability measure preserving} 
(abbreviated \textbf{p.m.p}) if in addition $m(X)=1$.

A p.m.p action of $\mathbf{Z}$ is given by an invertible measure preserving map $T$ of a probability space $(X,\calX,m)$,
with $T$ corresponding to the generator $1\in \mathbf{Z}$.
We shall sometimes denote the system $(X,\calX,m,T)$ by $\mathbf{X}=(X,\calX,m,T)$.
We shall also consider not necessarily invertible measure-preserving transformation $S$ of a probability space
$(Y,\calY,n)$, meaning $n(S^{-1}E)=n(E)$ for all $E\in \calY$.
It can be viewed as the action of the semi-group $\mathbf{N}$.

Let $T$ be a probability measure-preserving transformation of a Lebesgue space $\mathbf{X}=(X,\calX,m)$ 
and $S$ of $\mathbf{Y}=(Y,\calY,n)$.
An \textbf{equivariant quotient} $p:\mathbf{X}\to \mathbf{Y}$, is a measurable map $p:X\to Y$ with $p_*m=n$ and $p\circ T=S\circ p$ a.e.
In this case the disintegration $y\mapsto \mu_y$ is equivariant, i.e. 
\[
	\forae y\in Y:\qquad \mu_{Sy}=T_*\mu_y.
\]
Let $\calC\subset \calX$ be an $m$-complete sub-$\sigma$-algebra.
There is a (unique up to isomorphism) Lebesgue space $(Y,\calY, n)$ and a map $p:(X,\calX,m)\to (Y,\calY, n)$ for which $p^{-1}(\calY)=\calC$.
Let $T$ be a measure-preserving transformation of $(X,\calX,m)$.
Its action on the space of $m$-complete sub-$\sigma$-algebras is defined by 
\[
	T\calC:=\setdef{T^{-1}(E)}{E\in\calC}.
\]
We say that $\calC$ is $T$-\textbf{invariant} if $T\calC\subset \calC$, i.e.
$E\in\calC$ implies $T^{-1}(E)\in\calC$.
Given a $T$-invariant $m$-complete sub-$\sigma$-algebra $\calC$ there exists a measure-preserving transformation $S$ on the corresponding 
Lebesgue space $(Y,\calY,n)$ such that $p\circ T=S\circ p$.
Note that in this case $\calC\subset T^{-1}\calC\subset \dots$ and $S$ on $(Y,\calY,n)$ is invertible
iff $\calC$ is $T^{-1}$-invariant, i.e. $\calC= T^{-1}\calC$.
If $\calC$ is $T$-invariant and $\calC\subset T^{-1}\calC\subset \dots\to \calX$,
then the invertible p.m.p system $\mathbf{X}=(X,\calX,m,T)$ is the \textbf{natural extension} of
the system $\mathbf{Y}=(Y,\calY,n,S)$,
a concept that will be discussed in \S\ref{sub:naturalextension}.


%


\subsection{Ergodic components}
\label{sub:components}\hfill{}\\
Given a Lebesgue space $(X,\calX,m)$ endowed with a non-singular action of a
semigroup $S$, we may form the $\sigma$-algebra of $S$-invariants $\calX^S$
and the corresponding factor, which we call the $S$-ergodic components of $X$ and denote by $\pi:X\to X/\!\!/S$.
If $\Gamma$ is an lcsc group which acts non-singularly on $X$, normalizing $S$ we get a non-singular action of $\Gamma$ on $X/\!\!/S$ such that $\pi$ is $\Gamma$-equivariant.
This construction is functorial: for an $S$-equivariant map 
$\phi:(X,\calX,m) \to (Y,\calY,n)$ we get a corresponding map $\phi_*:X/\!\!/S \to Y/\!\!/S$
making the following diagram commutative:
\[
	\begin{tikzcd} X \ar[r, "\phi"] \ar[d] & Y \ar[d] \\
	           X/\!\!/S \ar[r, "\phi_*"] & Y/\!\!/S
	\end{tikzcd}
\]
 If $\phi$ was $\Gamma$ equivariant then so is $\phi_*$.

\subsection{Measurable cocycles}
\label{sub:cocycles}\hfill{}\\
Let $\Gamma$ be an lcsc group, $(\Omega,\omega)$ a $\Gamma$-Lebesgue space
and $G$ a polish group.
A \textbf{measurable cocycle} is a measurable map $\rho:\Gamma\times \Omega \to G$
which satisfies for every $\gamma,\gamma'\in\Gamma$ and $\omega$-a.e $z\in \Omega$ the equation 
\begin{equation} \label{eq:cocycle} 
	\rho(\gamma\gamma',z)=\rho(\gamma,\gamma'.z)\rho(\gamma',z)
\end{equation}

For a $G$-space $M$, a map $f:\Omega\to M$ is said to be {\bf $\rho$-equivariant}
or a {\bf $\rho$-map} if for every $\gamma\in \Gamma$ and for $\omega$-a.e $z\in \Omega$ we have
\[ f(\gamma z)=\rho(\gamma,z)f(z). \]

The space of all such measurable cocycles $\rho:\Gamma\times \Omega\to G$ is denoted $Z^1(\Gamma\times\Omega,G)$.
Two cocycles $\rho, \rho'\in Z^1(\Gamma\times\Omega, G)$ are \textbf{cohomologous} if there exists
a measurable map $s:\Omega\to G$ so that 
\begin{equation} \label{eq:cohom}
\rho'(\gamma,z)=s(\gamma.z)\rho(\gamma,z)s(z)^{-1}. 
\end{equation}
Note that when $\Omega=\{*\}$ is a singleton, $Z^1(\Gamma\times\Omega,G)$ is identified with $\Hom(\Gamma,G)$,
and being cohomologous is the same as being conjugate.

\begin{lemma}[Cocycle reduction lemma] \label{l:reduction}
    Assume $\rho\in Z^1(\Gamma\times\Omega, G)$ is a measurable cocycle, 
    $H<G$ a closed subgroup and $f:\Omega\to G/H$ is a $\rho$-map.
    
    Then there exists 
    measurable map $s:\Omega\to G$ such that for $\omega$-a.e $z\in \Omega$, 
    $s(z)f(z)=eH\in G/H$ and the cocycle
    $\rho'$ defined in \eqref{eq:cohom} satisfies for 
    every $\gamma\in \Gamma$ and $\omega$-a.e $z\in \Omega$,
    $\rho'(\gamma,z)\in H$, thus it may be viewed as an element 
    of $Z^1(\Gamma\times\Omega,H)$.
\end{lemma}

\begin{proof}
Fix a Borel section $\sigma:G/H\to G$ and set $s=\iota\circ \sigma\circ f:\Omega\to G$ 
where $\iota:G\to G$ is the inversion map.
\end{proof}

We define also a measurable cocycles for a semi-group actions by Equation~\eqref{eq:cocycle}.
In particular, we will consider in this paper actions of $\mathbb{N}$. 
Such an action is generated by a single transformation $T:\Omega\to \Omega$
and a cocycle $\rho:\mathbb{N}\times\Omega \to G$ is determined by its value on the generator, 
that is by the function $\rho(1,\cdot):\Omega\to G$.
Moreover, for every measurable function $w:\Omega\to G$ one can define a unique cocycle
$\rho:\mathbb{N}\times\Omega \to G$ satisfying  $\rho(1,\cdot)=w$
using the formula
\[
    \forall n \in \mathbb{N}, \quad 	\rho(n,z)=w_n(z):=w(T^{n-1}z)\cdots w(Tz)\cdot w(z).
\]
In case $T$ is invertible, the $\mathbb{N}$-action extends to a $\mathbb{Z}$-action
and the above cocycle extends as well to a cocycle $\rho:\mathbb{Z}\times \Omega\to G$ by the formula
\begin{equation}\label{e:f-cocycle}
	\rho(n,z)=w_n(z):=\left\{\begin{array}{lll}
	w(T^{n-1}z)\cdots w(Tz)w(z) & \textrm{if} & n>0,\\
	1 & \textrm{if} & n=0,\\
	w(T^{-n}z)^{-1}\cdots w(T^{-1}z)^{-1} & \textrm{if} & n<0.
	\end{array}\right.
\end{equation}

%
%
%
%
%


\subsection{Compact convex spaces} 
\label{sub:compact_convex_spaces}\hfill{}\\
Hereafter we use the term \textbf{compact convex space} as a shorthand for non-empty, convex, compact space 
that has enough continuous affine functionals to separate its points.
Such a space can be viewed as a non-empty closed convex subset $Q\subset [0,1]^A$, 
where  $A$ is a subset of the space $Q^*$ of affine functionals 
\[
	Q^*:=\setdef{\lambda:Q\to[0,1]}{\lambda\ \textrm{is\ continuous\ and\ affine}}
\]
and $A$ separates points of $Q$.
The class of convex compact spaces is closed under products, and under taking closed convex subsets.
A version of Hahn-Banach theorem shows that continuous affine functionals on a closed convex subset can be extended to an ambient space,
and that disjoint closed convex subsets can be separated by continuous affine functionals.

Denote by $\Aff(Q)$ the group of all continuous affine bijections $T:Q\to Q$.
It is a complete topological group with respect to the compact-open topology, i.e. the topology defined by
$g\mapsto \langle \lambda,g.q\rangle$, $\lambda\in Q^*$, $q\in Q$.

In what follows we will focus on metrizable convex compact spaces, i.e. spaces which are isomorphic to non-empty closed convex subsets of $[0,1]^\bbN$.
For such a space $Q$, the group $\Aff(Q)$ is a Polish group, and a homomorphism 
$\Gamma\to\Aff(Q)$ from a lcsc group $\Gamma$ is continuous iff the action map
$\Gamma\times Q\to Q$, $(\gamma,q)\mapsto \gamma.q$, is continuous. 

\begin{example}\label{Ex:convexcompact}
	Let $V$ be a separable Banach space, and $Q\subset V^*$ be a non-empty closed convex 
        subset of a ball in $V^*$.
	Then $Q$ is a convex compact space with respect to the weak* topology.
        In all of our examples $V$ is separable, hence $Q$ embeds in $[0,1]^\bbN$ 
        and is therefore metrizable.  
	There is a continuous homomorphism of Polish groups
	\[
		\setdef{ g\in \GL(V)}{g(Q)=Q}\ \overto{} \ \Aff(Q)
	\] 
	where $\GL(V)$ denotes the group of all bounded invertible linear automorphisms 
        of the Banach space $V$,
	equipped with the weak topology.
	The following is a particular example of this general setting, 
        that we shall use below.
		Let $M$ be a compact metrizable space and $Q=\Prob(M)$ the space of all Borel probability 
		measures on $M$ equipped with the weak* topology,
		i.e. $\mu_n\to\mu$ iff $\int f\dd \mu_n\to\int f\dd \mu$ for every $f\in C(M,\bbR)$.
		Then every $h\in \Homeo(M)$ defines an affine map $\alpha(h)\in \Aff(\Prob(M))$, and 
		$h\mapsto \alpha(h)$ is a continuous homomorphism (in fact isomorphism) of Polish groups.
\end{example}

Let $Q$ be a metrizable convex compact space. The \textbf{barycenter map}    
\[
	\bary:\Prob(Q)\to Q
\] 
is defined using integration of affine functionals as follows: for $\mu\in\Prob(Q)$ and $\lambda\in Q^*$ define
\[
	\langle \lambda,\mu\rangle :=\int_Q \langle \lambda,q\rangle\dd \mu(q)
\]
and let $\bary(\mu)\in Q$ be the unique point satisfying
\begin{equation}\label{e:def-bar-mu}
	\langle \lambda, \bary(\mu)\rangle = \langle \lambda,\mu\rangle\qquad (\lambda\in Q^*).
\end{equation}
It follows from the definitions that the affine continuous map
\[	
	\bary:\Prob(Q)\overto{} Q
\]
is $\Aff(Q)$-equivariant.
If $(X,\calX,\mu)$ is a probability space, $\psi\in\Map(X,Q)$ defines a push-forward measure
$\psi_*\mu\in\Prob(Q)$. 
One can denote the barycenter $\bary(\psi_*\mu)\in Q$ of this measure by $\int_X \psi(x)\dd \mu(x)$
meaning 
\[
	\langle \lambda,\int_X \psi(x)\dd \mu(x)\rangle
	=\int_X \langle\lambda,\psi(x)\rangle\dd \mu(x).
\]



\medskip

\subsection{Conditional Expectation and Martingales} \hfill{}\\
\label{sub:martingales}
Let $(X,\calX,m)$ be a standard probability space, 
$\eta:X\to Q$ and $\calF\subset \calX$ a complete sub-$\sigma$-algebra. 
Then there exists a unique, modulo null sets, 
$\calF$-measurable map 
\[
	\bbE(\nu \mid \calF):X\ \overto{} \ Q,
\]
called \textbf{conditional expectation} of $\nu$ relative to $\calF$, characterized by the property that for all $E\in\calF$
\[
	\int_E \langle \lambda, \nu(x)\rangle\dd m(x)= \int_E \langle \lambda, \bbE(\nu \mid \calF)\rangle\dd m
	\qquad(\lambda\in Q^*).
\]
Uniqueness, up to a modification on a $m$-null set, and existence of $\bbE(\nu \mid \calF)$ can be deduced from
the usual real valued conditional expectation by considering functions $x\mapsto \langle \lambda_i, \nu(x)\rangle$ for a countable 
separating family of functionals $\{\lambda_i\}_{i\in I}\subset Q^*$. 
Equivalently, one can use the barycenter map and disintegration of measures to construct $\bbE(\nu \mid \calF)$ as follows.
\begin{equation}\label{e:cond-via-bary}
	\begin{split}
		(X,\calF,m|_\calF)&\to  \Prob(X,\calX) \overto{\nu_*} \Prob(Q) \overto{\bary} Q,\\
		\bbE(\nu \mid \calF)(x)&=\bary(\nu_*(\mu_x))
	\end{split}
\end{equation}
where the $\calF$-measurable map $x\mapsto \mu_x\in \Prob(X,\calX)$ is the disintegration of 
the map $(X,\calX,m)\to (X,\calF,m|_\calF)$.

%

\medskip

The following version of the Martingale Convergence Theorem for functions taking values in convex compact spaces
follows from the classical one 
applied to the $L^\infty(X,m)$ functions $\{x\mapsto\langle \lambda_i, \eta_n(x)\rangle\}_{i\in I}$, 
where  $\{\lambda_i\}_{i\in I}\subset Q^*$ is a countable separating family.

\begin{theorem}[Martingale Convergence]\label{T:MCT}\hfill{}\\
	Let $\calF_1\subset \calF_2\subset \dots\subset \calF_n \nearrow \calF_\infty\subset \calX$ be an increasing sequence 
	of complete sub-$\sigma$-algebras of a standard probability space $(X,\calX,m)$, 
	$Q$ be a non-empty, convex, compact, metrizable space, 
	and $\{\eta_n:X\to Q \mid n\ge 1\}$ a sequence of maps satisfying 
	\[
		\eta_n=\bbE(\eta_{n+1} \mid \calF_n)\qquad (n\ge 1).
	\]
	Then there exists a unique (modulo null sets) $\calF_\infty$-measurable map $\eta:X\to Q$,  
	so that 
	\[
		\eta(x)=\lim_{n\to\infty} \eta_n(x)
	\]	
	for $m$-a.e $x\in X$.
\end{theorem}
%


\subsection{Metric Ergodicity and Relative Metric Ergodicity} 
\label{sub:metric_ergodicity}\hfill{}\\
Let $\Gamma\acts (A,\alpha)$ be a measure-class-preserving action of a lcsc group on a standard probability space.
We say that this action is \textbf{metrically ergodic} if given any continuous homomorphism $\Gamma\to \Isom(W,d)$
into isometry group of a Polish metric space $(W,d)$ (we use compact open topology on $\Isom$)
the only $\Gamma$-equivariant maps $A\to W$ are constant ones:
\[
	\Map_\Gamma(A,W)=W^\Gamma.
\]
Here, as usual, $\Map(A,W)$ stands for the space of classes of measurable maps $A\to W$ where two maps that
agree $\alpha$-a.e are identified. 
The subset $\Map_\Gamma(A,W)$ consists of (classes of) such maps $\phi:A\to W$ 
that satisfy $\phi(\gamma.a)=\gamma.\phi(a)$ for a.e $a\in A$ and a.e $\gamma\in\Gamma$.
Note that by replacing metric $d$ by an equivalent bounded metric (e.g. $\min(d,1)$ or $d/(1+d)$)
we may assume $(W,d)$ to have diameter $\le 1$.

If $\Gamma\acts (A,\alpha)$ is metrically ergodic then it is ergodic (consider a two point space
$W=\{0,1\}$ with the trivial $\Gamma$-action) and weakly mixing in the following sense.

\begin{prop}[Weak Mixing]\hfill{}\\
	Let $\Gamma\acts (A,\alpha)$ be a metrically ergodic measure-class-preserving action,
	and let $\Gamma\acts (\Omega,\omega)$ be an ergodic p.m.p action.
	Then the diagonal $\Gamma$-action on $(A\times\Omega,\alpha\times\omega)$ is ergodic.
\end{prop} 
\begin{proof}
	Since $\Gamma\acts (\Omega,\omega)$ is an ergodic p.m.p actio, 
	the unitary $\Gamma$-representation $\pi$ on $L^2_0(\Omega,\omega)=L^2(\Omega,\omega)\ominus \bbC$
	has no non-zero invariant vectors. 
	Let $W$ to be the unit sphere of $L^2_0(\Omega,\omega)$.
	Let $E\subset A\times \Omega$ be a $\Gamma$-invariant subset.
	For $a\in A$ denote $E_a:=\setdef{z\in\Omega}{(a,z)\in E}$
	and note that the measurable function $A\to [0,1]$, defined by  
	$a\mapsto \omega(E_a)$, is $\Gamma$-invariant.
	By ergodicity there is a constant $c$ so that a.e $\omega(E_a)=c$.
	If $c\ne 0,1$ then $F(a)=(1_{E_a}-c)/\sqrt{c\cdot (1-c)}$ is an equivariant 
	map to $W$, which is impossible. 
	Hence $c=0$ or $c=1$, and so $E$ is null or conull in $A\times \Omega$.
\end{proof}

\medskip

A \textbf{separable metric extension} is a Borel map between standard Borel spaces
$q:W\to V$ together with a Borel map $d:W\times_V W\to [0,+\infty]$ where for each $v\in V$, 
the restriction $d_v$ of $d$ to the fiber $W_v\times W_v$ gives a separable metric on $W_v=q^{-1}(\{v\})$.
A separable metric extension is said to be $\Gamma$-invariant if $\Gamma$ acts measurably on $W$ and $V$, 
the extension $q$ is $\Gamma$-equivariant, and the function $d$ is $\Gamma$-invariant. 
Alternatively, we say in this case that $\Gamma$ acts \textbf{isometrically} on $W$ \textbf{relatively} to $V$.
Note that for each $v\in V$ and $\gamma\in\Gamma$ the map  
\[
	\gamma:(W_v,d_v)\ \to \ (W_{\gamma.v},d_{\gamma.v}) 
\]
is an actual isometry.

\begin{example} \label{ex:invmet}
    Let $G$ be a Polish topological group and $L<H<G$ closed subgroups such that the coset space $H/L$ admits an $H$-invariant separable Borel metric $d$.
Consider $d$ as a map
\[ L\backslash H /L \simeq (H/L\times H/L)/H \to [0,\infty). \]
Using the identification 
\[ (G/L \times_{G/H} G/L)/G \simeq L\backslash H /L, \]
$d$ gives rise to a well defined $G$-invariant Borel separable metric associated with the map $\pi:G/L\to G/H$, making the later a separable metric extension.
\end{example}

\begin{defn}\label{D:field-of-polish}\hfill{}\\
Given $\Gamma$-Lebesgue spaces $(A,\calA,\alpha)$ and $(B,\calB,\beta)$,
a $\Gamma$-equivariant map $p\in \Map_\Gamma(A,B)$ is \textbf{relatively metrically ergodic} if, 
given any $\Gamma$-equivariant separable metric extension $q:W\to V$, and any 
measurable $\Gamma$-maps $F\in \Map_\Gamma(A,W)$ and $f\in \Map_\Gamma(B,V)$ satisfying $q\circ F=f\circ p$,
there exists a $\Gamma$-map $\phi\in \Map_\Gamma(B,W)$ so that $F=\phi\circ p$ as in the diagram
\[
		\begin{tikzcd} (A,\alpha) \ar[rr, "F"] \ar[d, "p"] & & W \ar[d, "q"] \\
		           (B,\beta) \ar[rr, "f"] \ar[urr, dotted, "\phi"] & & V.
		\end{tikzcd}
\]
\end{defn}

\begin{remark}
In Definition~\ref{D:field-of-polish}, upon replacing $d_v$ with $\min\{d_v,1\}$ for $v\in V$,
one may assume without loss of generality that $d$ is bounded by $1$.
\end{remark}

\begin{remark}
The role of the group $\Gamma$ In Definition~\ref{D:field-of-polish} coulde be replaced
by a semi-group. In particular, we may consider the semi-group $\mathbb{N}$,
generated by a single transformation.
This will be the context in Proposition~\ref{P:natext-relmeterg}.
\end{remark}

\begin{prop} \label{p:RME}
\begin{itemize}
\item The property of relative metric ergodicity is closed under composition of $\Gamma$-maps.
\item If $A\to B$ and $B\to C$ are $\Gamma$-maps, and $A\to C$ is relatively metrically ergodic,
then so is $B\to C$, but $A\to B$ need not be relatively metrically ergodic.
\item For Lebesgue $\Gamma$-spaces $A$ and $B$, if the projection $A\times B\to B$ is relatively
metrically ergodic then $A$ is (absolutely) metrically ergodic. This is an
``if and only if" in case $B$ is a p.m.p action, but not in general.
\end{itemize}
\end{prop}

\begin{proof}
    The verification is easy and left to the reader.
\end{proof}

We record the following for a later use, in \S\ref{sub:proof_of_theorem_ref_t_boundary}.

\begin{lemma} \label{L:rme-toGH}
Let $G$ be a Polish topological group and $L<H<G$ closed subgroups such that the coset space $H/L$ admits an $H$-invariant separable Borel metric $d$.
Let $p:(A,\alpha)\to  (B,\beta)$ be relatively metrically ergodic map,
	and 
	assume that there exist $F\in\Map_\Gamma(A, G/L)$ and $f\in\Map_\Gamma(B, G/H)$ 
	so that $f\circ p=\pi\circ F$. 
	
	Then there exists $\phi\in\Map_\Gamma(B,G/L)$ 
	which makes the following diagram commute:
	\[
		\begin{tikzcd}
			A \ar[rr, "F"] \ar[d, "p"] & & G/L \ar[d, "\pi"] \\
			B \ar[rr, "f"] \ar[urr, dotted, "\phi"] & & G/H.
		\end{tikzcd}
	\]
\end{lemma}

\begin{proof}
   Immediate from Example~\ref {ex:invmet}.
\end{proof}

\medskip


\subsection{Boundary systems}\hfill{}\\
This section deals with ``boundary theory".
Throughout the section we fix a locally compact second countable group $\Gamma$.
We also fix a triple $B,B_+,B_-$ of $\Gamma$-Lebesgue spaces together with $\Gamma$-maps
$p_\pm:B\to B_\pm$.
We assume throughout that they form a boundary system:

\begin{defn} \label{def:BS}
The touple $(B,B_+,B_-,\pi_+,\pi_-)$ forms a \textbf{boundary system} for $\Gamma$ if a 
the $\Gamma$-actions on $B_+,B_-$ and $B$ are Zimmer amenable, the $\Gamma$-maps $\pi_\pm$ are relatively metrically ergodic 
and the map $\pi_+\times\pi_-$
is measure class preserving, when endowing $B_+\times B_-$ with the product measure class.
\end{defn}

\begin{remark}
For the boundary system given by $(B,B_+,B_-,\pi_+,\pi_-)$, the pair $(B_+,B_-)$ is a boundary pair a la \cite[Definition~2.3]{BF:icm}.
Conversely, for every boundary pair $(B_+,B_-)$, the system $(B_+\times B_-,B_+,B_-,\pi_+,\pi_-)$,
where $\pi_\pm$ are the coordinate projections, forms a boundary system.
\end{remark}

Applying Proposition~\ref{p:RME} we get the following.

\begin{lemma} \label{L:BSME}
Each of the $\Gamma$-Lebesgue spaces $B_+,B_-$ and $B$ is metrically ergodic.
\end{lemma}

The following result is of a great importance.
Its proof, which was sketched in \cite{BF:icm}, will be given in full details in \S\ref{sub:proof_of_theorem_ref_t_boundary} below.

\begin{theorem}[{\cite[Theorem~3.4]{BF:icm}}]\label{T:boundary} \hfill{}\\
Let $G$ be a connected semisimple real Lie group with finite center
and $\rho:\Gamma\to G$ a continuous homomorphism with a Zariski dense image.
Let $P<G$ be a minimal parabolic.
	Then there exist $\Gamma$-maps $f_\pm:B_\pm\to G/P$
and, considering the map $\delta:G/P\to \Prob(G/P)$ that takes a point to the corresponding Diark measure,
we have 
\[
    \begin{split}
        &\Map_\Gamma(B_-,\Prob(G/P))=\{\delta\circ f_-\},\\
        & \Map_\Gamma(B_+,\Prob(G/P))=\{\delta\circ f_+\}.    
    \end{split}		
\]
Letting $Z<P$ be the centralizer of a maximal split torus $A<P$
and considering the standard map $\pr_+:G/Z\to G/P$
and the associated map $\pr_-=\pr_+\circ \wlong:G/Z\to G/P$
where $\wlong$ is the longest element in the Weyl group associated with $A$,
there exists a $\Gamma$-map $h:B\to G/Z\ $ such that 
\[
    f_-\circ \pi_-=\pr_-\circ h,
    \qquad
    f_+\circ \pi_+=\pr_+\circ h.
\]
\end{theorem}





\subsection{The natural extension}
\label{sub:naturalextension}\hfill{}\\
Let $(Y,\calY,\eta)$ be a Lebesgue space and $S:Y \to Y$ a measure preserving transformation, which is not necessarily invertible.
Denote $Y_n=Y$ for every $n$ and consider the backward sequence 
\[ 
	\cdots \overto{S} Y_2 \overto{S} Y_1 \overto{S} Y_0.
\]
We let $(X,\calX,\mu)=\varprojlim Y_n$ and record the map $p:X\to Y_0 =Y$.
The inverse limit of the sequence of maps $S:Y_n\to Y_n$ gives rise to a map that we will denote $T:X\to X$.
The latter map is invertible, as its inverse is the inverse limit of
the sequence of identity maps $Y_n\to Y_{n+1}$. 
Thus $(X,\calX,\mu,T)$ is an invertible measure preserving system.
Clearly, $p$ is $(T,S)$-equivariant.
In case $S$ is invertible, the map $p$ establishes an isomorphism between $X$ and $Y$, otherwise $Y$ is a proper factor.

Note that $p:X\to Y$ satisfy a universal property:
for every invertible system $(Z,\calZ,\zeta,R)$ and $R,S$-equivariant map $q:Z\to Y$,
the limit of the sequence of maps $q\circ R^{-n}:Z\to Y_n$ gives an $R,T$-equivariant map $q':Z\to X$ such that $q=p\circ q'$.
Therefore the system $(X,T)$ is called \textbf{the natural invertible extension} of $(Y,S)$.

\begin{prop}\label{P:natext-relmeterg}\hfill{}\\
	The natural extension $X\to Y$ is relatively metrically ergodic.\\
	In particular, $X$ is metrically ergodic iff $Y$ is metrically ergodic.
\end{prop}

Let us introduce some notations and preliminary lemmas before proving this proposition.

\medskip

Let $(X,\calX,m)$ be a standard probability space, $\calF\subset \calX$ a complete sub-$\sigma$-algebra.
It corresponds to a measurable quotient $p:(X,\calX,m)\to (Y,\calY,n)$ so that $\calF=p^{-1}(\calY)$,
and let $m=\int_Y \mu_y\dd n(y)$ be the disintegration of measures. 
We shall denote by $(X\times_Y X,m\times_Y m)$ the fiber product probability space
\[
	\begin{split}
		X\times_Y X&:=\setdef{(x,x')\in X\times X}{p(x)=p(x')},\\
		m\times_Y m&:=\int_Y \mu_y\times\mu_y\dd n(y).
	\end{split}
\] 
\begin{lemma}\label{L:vanishing-vars}
	Let $(X,\calX,m)$ be a standard probability space, $\calF_1\subset \calF_2\subset \dots \subset\calX$
	a sequence of complete $\sigma$-algebras with $\calF_n\nearrow \calX$,
	Let $X\to \dots \to Y_2\to Y_1$ be the corresponding quotients, 
	and let $F:X\to W$ be a measurable map into a Polish metric space $(W,d)$ with $\diam(W,d)\le 1$.
	Then 
	\[
		\lim_{n\to\infty}\ \int_{X\times_{Y_n} X}d(F(x),F(x'))\dd (m\times_{Y_n} m)(x,x') = 0.
	\]
\end{lemma}
\begin{proof}
	Take an arbitrarily small $\epsilon>0$. 
	Since $(W,d)$ is separable, we can cover $W$ by a countable union of $\epsilon$-balls
	$W=\bigcup_{i=1}^\infty B(w_i,\epsilon)$ for some $\setdef{w_i}{i\in\bbN}$ in $W$.
	Let $W_1=B(w_1,\epsilon)$ and 
	\[
		W_{j+1}=B(w_{j+1},\epsilon)\setminus\ \bigcup_{i=1}^j B(w_i,\epsilon)\qquad (j\in\bbN).
	\]
	Then $W=\bigsqcup_{i=1}^\infty W_i$ where $\diam(W_i)\le 2\epsilon$. 
	The sets $A_i=F^{-1}(W_i)$ form a countable 
	measurable partition of $X$, and so there exists $k\in\bbN$ so that 
	\[
		m(A_1\sqcup \dots \sqcup A_k)>1-\epsilon.
	\]
	Let $\delta=(\epsilon/2k)^2$. 
	Since $\calF_n\nearrow \calX$, there is $N$ and $\calF_N$-measurable sets $E_1,\dots, E_k$
	so that $m(E_i\triangle A_i)<\delta$. 
	Denoting $E_0=X\setminus \bigcup_{i=1}^k E_i$ we note that $m(E_0)<\epsilon+k\delta<2\epsilon$.
	
	For $n\in\bbN$ we denote $p_n:(X,\calX,m)\to (Y_n,\calY_n,m_n)$ the quotient of measure spaces
	so that $\calF_n=p_n^{-1}(\calY_n)$, and let 
	\[
		m=\int_{Y_n} \mu_{n,y}\dd m_n(y)
	\]
	the corresponding disintegration. Fix any $n\ge N$. Since $\calF_N\subset\calF_n$
	the sets $E_0,E_1,\dots,E_k$ are $\calF_n$-measurable, 
	and we denote by $\bar{E}_0,\bar{E}_1,\dots,\bar{E}_k\in\calY_n$ 
	the corresponding sets in $Y_n$. 
Note that for $y\in \bar{E}_i$, we have $\mu_{n,y}(X\setminus E_i)=0$. 
    
	Define $\Var_n:Y_n\to [0,1]$ by
	\[
		\Var_n(y):=\int_X\int_X d(F(x),F(x'))\dd\mu_{n,y}(x)\dd\mu_{n,y}(x').
	\]
	Then we have
	\[
		\begin{split}
			\int_{X\times_{Y_n} X} &d(F(x),F(x'))\dd m\times_{Y_n} m(x,x')=\int_{Y_n} \Var_n(y)\dd m_n(y)\\
				&\le \sum_{i=1}^k \int_{\bar{E}_i} \Var_n\dd m_n+\int_{\bar{E}_0} \Var_n\dd m_n\\
				&\le \sum_{i=1}^k \int_{\bar{E}_i} \Var_n\dd m_n+m_n(\bar{E}_0).
		\end{split}
	\]
	Let $i\in \{1,\dots,k\}$. Since 
	\[
		\int_{\bar{E}_i} \mu_{n,y}(E_i\setminus A_i)\dd m_n(y)=m(E_i\setminus A_i)<\delta
	\] 
	the set $\bar{E}_i^*:=\setdef{y\in \bar{E}_i}{\mu_{n,y}(E_i\setminus A_i)<\sqrt{\delta}}$
	satisfies $m_n(\bar{E}_i\setminus \bar{E}_i^*)<\sqrt{\delta}$ by Chebyshev's inequality.
    Since for $y\in \bar{E}_i$, we have $\mu_{n,y}(X\setminus E_i)=0$,
    we get that for $y\in \bar{E}_i^*$, $\mu_{n,y}(X\setminus A_i)<\sqrt{\delta}$.
	For $y\in \bar{E}_i^*$ we have $d(F(x),F(x'))\le 2\epsilon$ for those pairs $(x,x')\in X\times_{\calF_n}X$
	and $x,x'\in A_i$, therefore
	\[
		\begin{split}
			\Var_n(y)&  \le 2\int_{X\setminus A_i}\int_X d(F(x),F(x'))\dd\mu_{n,y}(x)\dd\mu_{n,y}(x') \\
            & +\int_{A_i}\int_{A_i} d(F(x),F(x'))\dd \mu_{n,y}(x)\dd \mu_{n,y}(x') \\ 
            & \leq 2\sqrt{\delta}  +2\epsilon\mu_{n,y}(A_i)^2 \\ 
			&\le 2\sqrt{\delta}+2\epsilon\cdot \mu_{n,y}(A_i).
		\end{split}
	\]
	We thus get
	\[
		\begin{split}
			\sum_{i=1}^k \int_{\bar{E}_i} \Var_n\dd m_n &\le 2k\sqrt{\delta}+2\epsilon\cdot \sum_{i=1}^k\int_{Y_n} \mu_{n,y}(A_i)\dd m_n(y)\\
			&=2k\sqrt{\delta}+2\epsilon\cdot \sum_{i=1}^k m(A_i)\le 2k\sqrt{\delta}+2\epsilon<4\epsilon.
		\end{split}
	\]
	We just proved that given $\epsilon>0$ there is $N$ so that for $n\ge N$
	\[
		\begin{split}
			\int_{X\times_{Y_n} X} &d(F(x),F(x'))\dd (m\times_{Y_n} m)(x,x')=\int_{Y_n} \Var_n\dd m_n\\
				&<\sum_{i=1}^k \int_{\bar{E}_i} \Var_n\dd m_n+m_n(\bar{E}_0)<6\epsilon.
		\end{split}
	\]
	This proves the lemma.
\end{proof}

\begin{proof}[Proof of Proposition~\ref{P:natext-relmeterg}]
	Let $p:(X,\calX,m,T)\to (Y,\calY,\nu,S)$ be a natural extension of a non-invertible p.m.p system,  
	let $q:W\to V$ be a separable metric extension, $F:X\to W$ and $f:Y\to V$ measurable equivariant maps with 
	$q\circ F=f\circ p$.
	Then the measurable function
	\[
		D:X\times_Y X\to [0,1],\qquad D(x,x'):=d_{f(y)}(F(x),F(x'))
	\]
	(where $p(x)=p(x')=y$) is invariant: $D(Tx,Tx')=D(x,x')$.
	We claim that
	\begin{equation}\label{e:intD}
		\int_{X\times_Y X} D\dd (m\times_Y m)=0.
	\end{equation}
	This would imply that for a.e $y\in Y$ the values of $F(x)$ are a.e constant $\phi(y)\in W_{f(y)}$ 
	over the fiber $X_y=p^{-1}(\{y\})$;
	the map $\phi:Y\to W$ is the required measurable descend of $F$ to $Y$.
	
	The $\sigma$-algebra $\calF:=p^{-1}(\calY)\subset\calX$ is $T$-invariant.
	Consider the increasing sequence of $\sigma$-algebras $\calF_0\subset\calF_1\subset\dots$ defined by
	\[
		\calF_n:=\setdef{T^nE\in\calF}{E\in\calF}\qquad (n\in\bbN).
	\]
	The fact that $p:X\to Y$ is a natural extension implies that $\calF_n\nearrow \calX$.
	We can view $X$ as the inverse limit of the quotient maps $Y_n\to\dots\to Y_2\to Y_1\to Y$.
	We have quotient maps
	\[
		X\times_{Y_n} X\to  Y_n\to Y
	\]
	that allow us to disintegrate these fibered products with respect to $Y$. 
	For a.e $y\in Y$ the fiber of $X\times_{Y_n} X\to Y$ can be viewed as
	the fibered product
	\[
		\left(X_y\times_{Y_{n,y}} X_y,m_y\times_{Y_{n,y}} m_y\right)
	\] 
	where $(X_y,m_y)$ is the fiber of $X\to Y$ at $y\in Y$, and the quotient $X_y\to Y_{n,y}$ 
	corresponds to the sub-$\sigma$-algebra $\calF_n|_{X_y}\subset \calX|_{X_y}$ on $(X_y,m_y)$.
	
	For a.e $y\in Y$ the measurable map $F:(X_y,m_y)\to W_{f(y)}$ takes values in a Polish space.
	Thus Lemma~\ref{L:vanishing-vars} implies that for a.e $y\in Y$ 
	\[
		\lim_{n\to\infty}\int_{X_y\times_{Y_{n,y}}X_y} D\dd (\mu_y\times_{Y_{n,y}}\mu_y) = 0.
	\] 
	Integrating over $y\in Y$ (since $D$ is bounded the dominated convergence theorem applies), we deduce that 
	\[
		\lim_{n\to\infty}\int_{X\times_{Y_{n}}X} D\dd (m\times_{Y_n}m) = 0.
	\]
	On the other hand, the fact that $D(T^nx,T^nx')=D(x,x')$ on $X\times_Y X$ for $n\in\bbN$ yields  
	\[
		\int_{X\times_{Y_{n}}X} D\dd (m\times_{Y_n} m)=\int_{X\times_{Y}X} D\dd (m\times_Y m).
	\]	
	This proves the claim about vanishing of integral on the right hand side, and thereby completes the proof of the proposition.
\end{proof}

\subsection{A useful Lemma}\hfill{}\\
Let $p:(X,m,T)\to (Y,n,S)$ be an equivariant map between ergodic p.m.p systems, where $T$ is invertible and $S$ is not necessarily.
Let 
\[
	y\mapsto \Prob(X),\qquad m=\int_Y\mu_y\dd n(y)
\]
be the corresponding disintegration of measures. 
\begin{lemma}\label{L:Markov}\hfill{}\\
	Let $f\in L^2(Y,n)$ be such that  
	$
		\ f(y)\le \int f(p(T^{-1}x))\dd \mu_{y}(x)\ 
	$
	for $n$-a.e $y\in Y$. \\ Then $f$ is a.e constant.
\end{lemma}
\begin{proof}
	The quotient map $p$ defines an isometric embedding $p:L^2(Y,n)\to L^2(X,m)$;
	the dual operator $p^*:L^2(X,m)\to L^2(Y,n)$
	is an orthogonal projection, explicitly defined by
	\[
		p^*F(y)=\int_X F(x)\dd \mu_y(x)\qquad (F\in L^2(X,m)).
	\]  
	Let $V_S$ on $L^2(Y,n)$ and $V_T$ on $L^2(X,m)$ be the isometries defined by $S$ and $T$.
	Then $V_S=p^*V_Tp$ and therefore the dual operator $U=V_S^*$ is given by
	$pV_{T^{-1}}p^*$, meaning
	\[
		Uf(y)=\int f(p(T^{-1}x))\dd \mu_{y}(x)\qquad (f\in L^2(Y,n)).
	\]
	The assumption of a.e inequality $f\le Uf$ implies that for all $k\in \bbN$
	\[
		f\le Uf\le U^2\le \dots\le U^kf.
	\]
	Ergodicity of $(Y,n,S)$ implies by von Neumann's mean ergodic theorem an $L^2$-convergence
	\[
		\frac{1}{k}(f+Uf+\dots U^{k-1}f)\ \overto{L^2}\ \int_Y f.
	\]
	Since $f$ is a.e dominated by the averages on the LHS for any $\epsilon>0$ one has 
	\[
		n\setdef{y\in Y}{ f(y)\ge \int f\, +\epsilon}
		\le \frac{1}{\epsilon^2}\cdot \left\|\frac{1}{k}(f+Uf+\dots U^{k-1}f)-\int f\right\|_2^2\to 0.
	\] 
	Therefore we get a.e inequality
	\[
		f(y) \le \int_Y f
	\]
	which implies that $f$ is a.e constant, as claimed.
\end{proof}

\subsection{Some lemmas regarding spaces of measures} \label{subsection:measures}\hfill{}\\
Let $M$ be a compact metrizable space
and denote by $Q$ the compact convex space of positive measures of total mass in $[0,1]$ on $M$;
it is a cone on $\Prob(M)$.
We define the map
\[ \ev:Q\times M\overto{} [0,1], \quad (\mu,\xi)\mapsto \nu(\{\xi\}). \]

\begin{lemma} \label{l:ev}
The map $\ev$ is upper semi-continuous (usc).
\end{lemma}

\begin{proof}
We fix $(\nu,\xi)\in Q\times M$.
Next we fix $t\in \mathbb{R}$ such that 
$\ev(\nu,\xi)<t$.
we find an open neighborhood $m\in U\in M$ such that $\nu(U)<t$
and then find a continuous function $f:M\to [0,1]$ such that $f|_U=1$ and $\nu(f)<t$.
The integration of $f$ is a continuous function on $Q$,
we thus find a neighborhood of $\nu\in V\in Q$ consisting of measures $\nu'$ satisfying $\nu'(f)<t$.
Then for every $(\nu',\xi')\in V\times U$, $\ev(\nu',\xi')<t$.
As $t$ is arbitrary, this shows that $\ev(-,-)$ is usc at $(\nu,\xi)$.
As $(\nu,\xi)$ is arbitrary, this show that $\ev$ is usc.
\end{proof}

We recall the following standard result.

\begin{lemma} \label{l:usc}
Let $X,Y$ be topological spaces and $f:X\times Y\to \mathbb{R}$ be an upper semi-continuous function.
Assume $Y$ is compact and let $\bar{f}:X\to \mathbb{R}$ 
be defined by $\bar{f}(x)=\max\{f(x,y)\mid y\in Y\}$.
Then $\bar{f}$ is upper semi-continuous.
\end{lemma}

\begin{proof}
We fix $x\in X$. Next we fix $t\in \mathbb{R}$ such that 
$\bar{f}(x)<t$.
For every $y\in Y$ we have $f(x,y)<t$, thus we can find an open neighborhood $(x,y)\times U_x\times V_y\subset X\times Y$
such that $f(U_x\times V_y)\subset (-\infty,t)$.
Considering the open cover $\{V_y\mid y\in Y\}$ of $Y$,
we find a finite subcover $\{V_{y_i}\}$.
Defining $U=\bigcap U_{x_i}$ we get that $f(U\times Y)\subset (-\infty,t)$,
thus $\bar{f}(U\times Y)\subset (-\infty,t)$.
As $t$ is arbitrary, this show that $\bar{f}$ is usc at $x$.
As $x$ is arbitrary, this show that $\bar{f}$ is usc.
\end{proof}

We define the map
\[ \maxev:Q\overto{} [0,1], \quad \maxev(\nu)=\max\setdef{\nu(\{\xi\})}{ \xi\in M}. \]

\begin{lemma} \label{l:maxev}
The map $\maxev(-)$ is convex and upper semi-continuous.
\end{lemma}

\begin{proof}
For $\nu_1,\nu_2\in Q$ and $t\in [0,1]$
we pick an atom $\xi\in M$ which is of maximal weight for the measure $t\nu_1+(1-t)\nu_2$
and observe that
\begin{equation} \label{eq:convex}
	\begin{split}
		\maxev(t\nu_1+(1-t)\nu_2) &=t\cdot\nu_1(\{\xi\})+(1-t)\cdot\nu_2(\{\xi\}) \\
		&\le t\cdot\maxev(\nu_1)+(1-t)\cdot\maxev(\nu_2).
	\end{split}
\end{equation}
This shows that $\maxev(-)$ is convex.
Semi-continuity follows from Lemma~\ref{l:ev} and Lemma~\ref{l:usc}.
\end{proof}

Next we fix $t\in (0,1]$ and consider the subset $Q_t=\maxev^{-1}([0,t])\subset Q$.

\begin{lemma} \label{l:Qalpha}
The subset $Q_t$ is a convex $G_\delta$-subset of $Q$.
\end{lemma}

\begin{proof}
The convexity follows from the convexity of $\maxev(-)$
and the $G_\delta$ property from the usc:
indeed, for every $s\in (0,1]$ the set $Q^o_s=\maxev^{-1}([0,s))$
is open and $Q_t=\bigcap_n Q^o_{t+1/n}$.
\end{proof}

%
%
%

We note that the set $\ext(Q_t)$ of extreme points of $Q_t$ consists of exactly 
those measures in $Q_t$ which are purely atomic
with weight $t$ at each point of their (finite) support.
We define the map
\[ 
    \maxpart:Q_t\overto{} \ext(Q_t) 
\]
taking a measure $\nu\in Q_t$ to the maximal measure in $\ext(Q_t)$ which is dominated by $\nu$.

\begin{lemma} \label{l:maxpart}
The map $\maxpart(-)$ is Borel measurable.
\end{lemma}

\begin{proof}
By Lemma~\ref{l:Qalpha}, $Q_t$ is a $G_\delta$-set in $Q$,
thus $Q_t$ forms a Polish space.
It is easy to see that $\ext(Q_t)$ a closed subset of $Q_t$, hence it is also Polish.
The set $Q^o_t=\maxev^{-1}([0,t))$ is an open subset of of $Q$
by the usc of $\maxev(-)$; hence it is also Polish.
We obsereve that every measure in $Q_t$ could be decomposed uniquely as a sum of 
a measure in $\ext(Q_t)$ and a measure in $Q^o_t$
and get a continuous bijection
\[ 
    \ext(Q_t)\times Q^o_t\overto{} Q_t, \qquad (\nu_1,\nu_2)\mapsto \nu_1+\nu_2. 
\]
Since this is a Borel bijection between Polish spaces, we get that its inverse is also Borel,
and the result follows by composing this inverse with the projection on the first factor:
\[
    \ext(Q_t)\times Q^o_t \overto{} \ext(Q_t).
\]
This pproves the Lemma.
\end{proof}


\section{Basic constructions associated with Gregs}\label{sec:gregs}

Throughout this section we let 
$(X,\calX,m,T,w,\Gamma)$ be a Greg as in Definition~\ref{defn:Greg}.

\subsection{The Past and Future Factors associated with a Greg}\label{sub:basic_construction}\hfill{}\\
We consider the cocycle $\bbZ\times X\to  \Gamma$ generated by $w$
and denote its value at $(n,x)$ by $w_n(x)$ as in (\ref{e:f-cocycle}) from \S\ref{sub:cocycles}.
It is characterized by the property $w_1(x)=w(x)$ and
\[
	w_{n+k}(x)=w_k(T^n x)\cdot w_n(x)\qquad(n,k\in\bbZ,\ \forae x\in X).
\]
Let $\calF\subset \calX$ denote the smallest $m$-complete $\sigma$-algebra of $\calX$ for which $w$ is measurable.
Define for $-\infty\leq m \leq n \leq \infty$, 
\[	\calF_m^n:=\bigvee_{k=m}^m T^k\calF. \]
Note that for $n\geq 0$, $w_n \in \calF_0^n$ and for $n\leq 0$, $w_n \in \calF_n^0$.
Since we are interested in invariants that depend only on $\{w_n, n\in\bbZ\}$, we may assume without loss of generality in what follows that $\calF_{-\infty}^\infty=\calX$.
We make this assumption. Equivalently,
we assume that the map
\begin{equation} \label{eq:assumptionX}
	i:X \to \Gamma^\bbZ,\quad x \mapsto (w_n(x))_{n\in \bbZ}, \quad 
	\mbox{is injective.}
\end{equation}

We let $p_+:(X,m)\to (X_+,m_+)$ be the factor of $(X,m)$ corresponding to the $\sigma$-algebra $\calF_0^\infty$. Note that $T\calF_0^\infty=\calF_1^\infty \subset \calF_0^\infty$, thus $T$ gives rise to a (typically non-invertible) 
transformation $T_+:X_+\to X_+$ such that 
\begin{equation} \label{eq:p_+eq}
\mbox{for $m$-a.e $x\in X$,} \quad p_+(Tx)=T_+(p_+(x)). 
\end{equation}
$T_+$ is invertible iff $p_+$ is an isomorphism, in fact, 
\[
	p_+:(X,m,T) \to  (X_+,m_+,T_+)
\] 
is the natural extension of the system $(X_+,m_+,T_+)$ (see \S\ref{sub:naturalextension}).
Since $w_n:X\to\Gamma$ with $n\ge 0$ are $\calF_0^\infty$-measurable they
(and in particular $w=w_0$) descend to $X_+\to \Gamma$. 

Similarly, we define the factor map $p_-:(X,m)\to (X_-,m_-)$ corresponding to the $\sigma$-algebra $\calF_{-\infty}^{-1}$
and obtain the system $(X_-,m_-,T_-)$ satisfying for $m$-a.e $x\in X$,
$p_-(T^{-1}x)=T_-\left(p_-(x)\right)$.
We note that $w_n:X\to\Gamma$ with $n\le 0$ (and in particular $w=w_0$)
descend to $X_-\to \Gamma$.

\subsection{The space of Random Walks} \label{subse:RW}\hfill{}\\
We consider the map 
\begin{equation}\label{e:GammaZ-map}
	j:X\times \Gamma\to  \Gamma^\bbZ\qquad
	\textrm{given\ by}\qquad (x,g)\mapsto (w_n(x)g^{-1})_{n\in\bbZ}
\end{equation}
and note that it is injective,
as follows clearly from \eqref{eq:assumptionX} and the fact that $w_0(x)=1$ for all $x\in X$.
We fix once and for all a probability measure $m^1_\Gamma$ on $\Gamma$ 
which is in the Haar class.
In particular, $m^1_\Gamma$ is fully supported 
and absolutely continuous with respect to both left and right Haar measure/s.
We let $\wt{m}^1=j_*(m\times m^1_\Gamma)$ be the measure on $\Gamma^\bbZ$ obtained by 
the push forward of $m\times m^1_\Gamma$.
We conclude that 
\begin{equation}\label{e:GammaZ-map-Leb}
    j:(X\times \Gamma,m\times m_\Gamma^1) \overto{}  (\Gamma^\bbZ, \wt{m}^1) 
    \quad\mbox{is a Lebesgue isomorphism.}
\end{equation}
Note that $(\Gamma^\bbZ,\wt{m}^1)$ is a Lebesgue space with a measure class preserving 
action of $\Gamma$ and a commuting measure class preserving action of $\wt{T}^\bbZ$ given by
\[
	g:(\gamma_i)\mapsto (\gamma_ig^{-1}),
	\qquad
	\wt{T}^n:(\gamma_i)\mapsto (\gamma_{i+n}).
\]
We denote $j_g(x)=j(g,x)$
and get that for every $g\in\Gamma$, 
$j_g=g\circ i$ and in particular, $j_1=i$.
We set $m_g=(j_g)_*(m)$ and note that $m_g=g_*(m_1)$.
By (\ref{eq:assumptionX}), $i:(X,m) \simeq (\Gamma^\bbZ,m_1)$ as Lebesgue spaces,
thus for every $g\in \Gamma$, $j_g:(X,m) \simeq (\Gamma^\bbZ,m_g)$ as Lebesgue spaces.

The measures $m_g$ are probability measures on $\Gamma^\bbZ$ 
which are supported on the set 
$\setdef{(\gamma_n)_{n\in\bbZ}}{\gamma_0=g^{-1}}$
and 
\[ 
	\wt{m}^1=\int_\Gamma m_g \dd m_\Gamma^1(g). 
\]

The action of $\Gamma$ on $\Gamma^\bbZ$ is free
and the set $\setdef{(\gamma_n)_{n\in\bbZ}}{\gamma_0=1}$
is a fundamental domain for this action.
In case $\Gamma$ is discrete, this set has a positive measure and the restriction 
of $\wt{m}$ to this sets is $m_1$.
The original system $X$ may be viewed as the quotient map $\theta$ defined as follows:
\[
	(\Gamma^\bbZ,\wt{m},\wt{T})\to(\Gamma^\bbZ,\wt{m},\wt{T})/\Gamma\ \overto{\cong}\ (X,m,T).
\]
This in particular means that $T$ is the map induced by $\wt{T}$.
The $\bbZ$-cocycle $\{w_n:X\to \Gamma\}$ is defined by the relation 
\begin{equation} \label{eq:compcocycle}
	j_g(T^nx)gw_n(x)g^{-1}=\wt{T}^nj_g(x).
\end{equation}
that holds for $m_\Gamma$-a.e $g\in \Gamma$, $m$-a.e $x\in X$ and every $n\in \bbZ$:

\subsection{Future and Past Random Walks}\hfill{}\\
We now consider the map 
\[ 
	j^+:X_+\times \Gamma\overto{}  \Gamma^{\bbZ_+}, \qquad
	j^+(x,g)=(w_n(x)g^{-1})_{n=0}^\infty 
\]
and
\[ 
	j^-:X_-\times \Gamma\overto{}  \Gamma^{\bbZ_-}, \qquad
	j^-(x,g)=(w_n(x)g^{-1})_{n=-\infty}^{-1}, 
\]
as well as $j^+_g=j^+(\cdot,g):X_+\to \Gamma^{\bbZ_+}$
and $j^-_g=j^-(\cdot,g):X_-\to \Gamma^{\bbZ_-}$.
Analogously to (\ref{eq:compcocycle}) we have for 
$m_\Gamma$-a.e $g\in \Gamma$, $m_+$-a.e $x\in X_+$ and every $n\in \bbZ_+$:
\begin{equation} \label{eq:compcocycle+}
    \begin{split}
        &j^+_g(T_+^nx)gw_n(x)g^{-1}=\wt{T}_+^nj^+_g(x), \\
        &j^-_g(T_-^nx)gw_n(x)g^{-1}=\wt{T}_-^nj^-_g(x),
    \end{split}
\end{equation}
where $\wt{T}_+$ and $\wt{T}_-$ denote the corresponding (non-invertible) 
shift operators on $\Gamma^{\bbZ_+}$ and $\Gamma^{\bbZ_-}$.
As in (\ref{e:GammaZ-map-Leb}), we get that
\begin{equation} \label{e:GammaZ-map-Leb+}
    \begin{split}
        j^+:(X_+\times \Gamma,m_+\times m^1_\Gamma) \overto{}  (\Gamma^\bbZ_+, \wt{m}^1_+),\\
        j^-:(X_-\times \Gamma,m_-\times m^1_\Gamma) \overto{}  (\Gamma^\bbZ_-, \wt{m}^1_-)
    \end{split}
\end{equation}
are Lebesgue isomorphisms, where $\wt{m}^1_\pm=(j^\pm)_*(m_\pm\times m^1_\Gamma)$.
By construction, the projections 
$(\Gamma^\bbZ,\wt{m}^1)\overto{} (\Gamma^{\bbZ_+},\wt{m}^1_+)$ and 
$(\Gamma^\bbZ,\wt{m}^1)\overto{} (\Gamma^{\bbZ_-},\wt{m}^1_-)$ are measure preserving
and $\Gamma$-equivariant.
They are also equivariant with respect to the respective actions of 
$\wt{T}$, $\wt{T}_+$ and $\wt{T}_-$.

\subsection{The spaces of Ideal Futures and Ideal Pasts}
\label{subs:idealpastfuture}\hfill{}\\
In view of the construction of the space of ergodic components given in \S\ref{sub:components}
we define 
\[ 
	(E,\eta)=(\Gamma^{\bbZ},\wt{m}^1)/\!\!/\wt{T}^\bbZ
\]
and 
\[
	(E_+,\eta_+)=(\Gamma^{\bbZ_+},\wt{m}^1_+)/\!\!/\wt{T}^{\bbZ_+},
        \qquad
        (E_-,\eta_-)=(\Gamma^{\bbZ_-},\wt{m}^1_-)/\!\!/\wt{T}^{\bbZ_-}.
\]
Note that the measures $\eta,\eta_+$ and $\eta_-$ depend on the choice of the 
probability measure $m^1_\Gamma$ on $\Gamma$, 
but the corresponding measure classes do not.
We denote the corresponding factor maps by $\beta$, $\beta_+$, $\beta_-$,
respectively.
We obtain the following commutative diagram.
\begin{equation} \label{eq:timetable}
	\begin{tikzcd} 
		(E_-,\eta_-) & (E,\eta) \ar[l, "\pi_-"] \ar[r, "\pi_+"] & (E_+,\eta_+) \\
		(\Gamma^{\bbZ_-},\wt{m}^1_-)  \ar[u, "\beta_-"] & 
			(\Gamma^{\bbZ},\wt{m}^1) \ar[l] \ar[r]  \ar[u, "\beta"] & (\Gamma^{\bbZ_+},\wt{m}^1_+) \ar[u, "\beta_+"] \\
            		(\Gamma^{\bbZ_-},\wt{m}_-) \ar[d, "\theta_-"] \ar[u, "\sim"] & 
			(\Gamma^{\bbZ},\wt{m}) \ar[l] \ar[r] \ar[d, "\theta"] \ar[u, "\sim"] & (\Gamma^{\bbZ_+},\wt{m}_+) \ar[d, "\theta_+"] \ar[u, "\sim"]\\
		(X_-,m_-) & (X,m) \ar[l, "p_-"] \ar[r, "p_+"] & (X_+,m_+)
	\end{tikzcd}
\end{equation}

\begin{remark} \label{r:Mackey}
Note that $(E,\eta)$ and $(E_\pm,\eta_\pm)$ could be identified with the Mackey ranges of the 
cocycle $w$ on $(X,m)$ and $(X_\pm,m_\pm)$ correspondingly.
In particular, $X$ is $T$-ergodic iff $E$ is $\Gamma$-ergodic.
\end{remark}

\subsection{Correspondences between spaces of maps}\hfill{}\\
Let $V$ be a Polish space and $\Gamma\times V\to V$, $\gamma:v\mapsto \gamma.v$, a Borel measurable action.
We define  
\[
	\begin{split}
		&\Map_\Gamma(E,V):=\setdef{\phi\in\Map(E,V)}{\forall \gamma\in\Gamma,\ \phi(\gamma.y)=\gamma.\phi(y)},\\
&\Map_\Gamma(E_+,V):=\setdef{\phi\in\Map(E_+,V)}{\forall \gamma\in\Gamma,\ \phi(\gamma.y)=\gamma.\phi(y)},\\
		&\Map_\Gamma(E_-,V):=\setdef{\phi\in\Map(E_-,V)}{\forall \gamma\in\Gamma,\ \phi(\gamma.y)=\gamma.\phi(y)},\\
		&\Map_{w}(X,V):=\setdef{f\in\Map(X,V)}{f(Tx)=w(x)f(x)},\\
		&\Map_{w}(X_-,V):=\setdef{f\in\Map(X_-,V)}{w(x).f(T_-x)=f(x)},\\
       &\Map_w(X_+,V):=\setdef{f\in\Map(X_+,V)}{f(T_+x)=w(x).f(x)}.
	\end{split}
\]
\begin{lemma}\label{L:equi-inv}\hfill{}\\
	There are bijections $\Map_\Gamma(E,V)\cong \Map_w(X,V)$
and $\Map_\Gamma(E_\pm,V)\cong \Map_w(X_\pm,V)$
which fit into the following two commutative diagrams 
	\[ 	
		\begin{tikzcd}
			\Map_\Gamma(E_\pm,V) \ar[d, hook, "\pi_\pm^*"]\ar[r, "\cong"] & \Map_w(X_\pm,V) \ar[d, hook, "p_\pm^*"] \\
			\Map_\Gamma(E,V) \ar[r, "\cong"] & \Map_w(X,V).
		\end{tikzcd}
	\]
and these diagrams are natural with respect to $V$,
that is a $\Gamma$-map $V\to U$ gives rise to cube diagrams in the obvious way.
\end{lemma}
\begin{proof}
We will prove the existing of the bijection 
$\Map_\Gamma(E_+,V)\cong \Map_w(X_+,V)$,
the others being similar.
Note first that $\Map_\Gamma(E_+,V)$ is isomorphic to the space of $\wt{T}$-invariant, $\Gamma$-equivariant maps from $\Gamma^{\bbZ_+}$ to $V$, $\Map_\Gamma(\Gamma^{\bbZ_+},V)^{\wt{T}}$, as $E_+=\Gamma^{\bbZ_+}/\!\!/{\wt{T}}$.
We will show $\Map_\Gamma(\Gamma^{\bbZ_+},V)^{\wt{T}}\simeq \Map_w(X_+,V)$.

Fix $\Phi_0$, a representative of a map in $\Map_\Gamma(\Gamma^{\bbZ_+},V)^{\wt{T}}$.
	For $m_+$-a.e $y\in X_+$ for $m_\Gamma$-a.e $g\in G$
	\[
		\begin{split}
			\Phi_0(j^+_g(y))&=\Phi_0(\wt{T}_+j^+_g(y))=
				\Phi_0(j^+_g(T_+y)gw(y)g^{-1})\\
				&=gw(y)^{-1}g^{-1}\Phi_0(j^+_g(T_+y))).
		\end{split}
	\]
	Using Fubini theorem, for $m_\Gamma$-a.e $g\in\Gamma$ the function $f_0:X_+\to V$ 
	defined by $f_0(y)=g^{-1}\Phi_0(j^+_g(y))$  
	is measurable and satisfies
	\[
		f_0(T_+y)=w(y)f_0(y)
	\]
	for $m_+$-a.e $y\in X_+$. It follows from the definitions that changing $\phi_0$ on an $\eta_+$-null 
	set results
	in a change of $\Phi_0$ on an $\wt{m}_+$-null set, and a change in $f_0$ on an $m_+$-null set.
	Therefore we get a well-defined map 
        \[
            \Map_\Gamma(E_+,V)\overto{} \Map_w(X_+,V).
        \]
	For the converse, given $f\in \Map_w(X_+,V)$ choose a representative $f_1:X_+\to V$.
        By (\ref{e:GammaZ-map-Leb+}) there exists maps $\alpha:\Gamma^{\bbZ_+}\to X_+$ 
        and $\beta:\Gamma^{\bbZ_+}\to \Gamma$
        such that for $\wt{m}_+$-a.e $t \in \Gamma^{\bbZ_+}$,
        $j^+(\alpha(t),\beta(t))=t$.
        We define $\Phi_1:\Gamma^{\bbZ_+}\to V$ by $\Phi_1(t)=\beta(t)f_1(\alpha(t))$.
        Note that $\beta(t)=t_0^{-1}$ and observe that $\Phi_1$ 
        is a.e $\Gamma$-equivariant:
        \[ 
	\Phi_1(t\gamma^{-1})=\gamma t^{-1} f_1(\alpha(t))=\gamma\Phi_1(t). 
        \]
	Since $m_+$-a.e $w_{n+1}(y)=w_n(T_+y)\cdot w(y)$ and $f_1(T_+y)=w(y).f_1(y)$ it follows that 
	\[
		\begin{split}
			\Phi_1(w_{n+1}(y)\gamma_0)&=\Phi_1(w_n(T_+y)\cdot w(y)\gamma_0)
			=(w(y)\gamma_0)^{-1}.f_1(T_+y)\\
			&=\gamma_0^{-1}w(y)^{-1}w(y)f_1(y)=\Phi_1(w_n(y)\gamma_0).
		\end{split}
	\]
	Thus we have an $\wt{m}_+$-a.e equality $\Phi_1\circ \wt{T}=\Phi_1$.
	Thus $\Phi_1$ defines a map $\Phi\in\Map_\Gamma(\Gamma^{\bbZ_+},V)$ that is $\wt{T}$-invariant.
\end{proof}

\subsection{Ergodic properties of the spaces of Ideal Futures and Ideal Pasts} \hfill{}\\
In this subsection we prove certain ergodic properties that will be used later to show
that for \textit{Apafic} Gregs the tripple $(E,E_+,E_-)$ forms a boundary system.

\begin{theorem}\label{T:rME}\hfill{}\\
	The maps $\pi_\pm:(E,\eta)\overto{}  (E_\pm,\eta_+)$ 
	are relatively metrically ergodic maps.
\end{theorem}

\begin{proof} We prove that $\pi_+:E\to E_+$ is relatively metrically ergodic, the proof for $\pi_-:E\to E_-$
	is completely analogous.

	Let $q:W\to V$ be $\Gamma$-equivariant separable metric extension, and $F\in\Map_\Gamma(E,W)$ and $f\in\Map_\Gamma(E_+,V)$
	be compatible, i.e. $f=q\circ F$. 
	Lemma~\ref{L:equi-inv} provides corresponding maps $F'\in \Map_w(X,W)$ and $f'\in\Map_w(X_+,V)$ 
	so that $f'=q\circ F'$.
	Note that the space $X_+\times V$ is endowed with the $\mathbb{N}$-actions generated by 
$(x,v) \mapsto  (Tx,w(x).v)$
and similarly for 
$X_+\times W$.
Moreover, the map $\id\times q:X_+\times W\to X_+\times V$ is clearly an $\mathbb{N}$-equivariant isometric extension.
Using Proposition~\ref{P:natext-relmeterg}
we obtain the diagonal map in 
	\[
		\begin{tikzcd}
			X   \ar[rr, "p_+\times F'"] \ar[d, "p_+"] & \quad &  X_+\times W \ar[d, "q"] \\
			X_+  \ar[rr, "\id\times f'"] \ar[urr, dashed] & \quad &  X_+\times V.
		\end{tikzcd}
	\]
which is clearly of the form $\id\times\phi'$ for some $\phi'\in\Map_w(X_+,W)$.
By Lemma~\ref{L:equi-inv}, $\phi'$ corresponds to some $\phi\in\Map_w(E_+,W)$,
and one checks easily that the following diagram commutes:
	\[
		\begin{tikzcd}
			E   \ar[rr, "F"] \ar[d, "\pi_+"] & \quad &  W \ar[d, "q"] \\
			E_+  \ar[rr, "f"] \ar[urr, dashed, "\phi"] & \quad &  V.
		\end{tikzcd}
	\]
	This proves that $\pi_+:E\to  E_+$ is relatively metrically ergodic. 
\end{proof}

\begin{lemma}\label{L:amen}\hfill{}\\
The $\Gamma$-Lebesgue spaces $E_+,E_-$ and $E$ are amenable in the sense of Zimmer
that is for any non-empty metrizable compact convex space $Q$ and continuous homomorphism $\Gamma\to \Aff(Q)$, 
convex, compact spaces $\Map_\Gamma(E_+,Q)$, $\Map_\Gamma(E_-,Q)$ and $\Map_\Gamma(E,Q)$
are non-empty.
\end{lemma}
\begin{proof}
	We prove $\Map_\Gamma(E_+,Q)\neq \emptyset$, the argument for $\Map_\Gamma(E_-,Q)\neq \emptyset$
being completely analogous and $\Map_\Gamma(E,Q)\neq \emptyset$ follows immediately.
By Lemma~\ref{L:equi-inv} it is enough to show that $\Map_w(X_+,Q)\neq \emptyset$.
Since $Q$ is non-empty, the space $\Map(X_+,Q)$ is non-empty as well
and it is a convex compact space in view of existence of barycenters (\ref{e:def-bar-mu}).
	We consider the continuous affine self-map $\tau$ of the compact convex space $\Map(X_+,Q)$
defined by
	\[
		(\tau f)(y)=w(y). f(T_+ y),\qquad (f\in \Map(X_+,Q)).
	\]
and observe that it has a fixed point, by Kakutani's Fixed Point Theorem, that is by the amenability of the semigroup $\mathbb{N}$.
Such a fixed point is indeed an element of $\Map_w(X_+,Q)$,
which is thus proven to be non-empty.
\end{proof}

\subsection{Stationary Maps}\label{sub:stat}\hfill{}\\
In this subsection we define the notion of stationary maps and prove their basic properties.
We consider the disintegration of the measure $m$ under the map $p_+:\mathbf{X}\to \mathbf{X}_+$,
giving rise to the measures $\mu_y$ on $X$, defined for a.e $y\in X_+$:
\begin{equation} \label{eq:disq_+}
m=\int \mu_{y}\dd m_+(y)
\end{equation}

\begin{defn}[Stationary maps]\label{D:stat}\hfill{}\\
	Let $Q$ be a metrizable convex compact space and $\Gamma\to\Aff(Q)$ a continuous homomorphism.
	A map $\psi\in\Map(X_+,Q)$ is called \textbf{stationary} 
	if it satisfies
	\[
		\psi(y)=\int_{X} w(T^{-1}x).\psi(p_+(T^{-1}x))\dd \mu_{y}(x)
	\]
	where the integration over $(X,\mu_y)$ is understood in the sense of (\ref{e:def-bar-mu}).
	We denote by $\Stat(X_+,Q)$ the subspace of stationary maps.
\end{defn}

The space $\Stat(X_+,Q)$ of stationary maps is clearly convex and closed subspace of $\Map(X,Q)$.
Therefore it is convex compact. 
The following Theorem identifies it with a familiar object. 

\begin{theorem}\label{T:structure-of-stat}\hfill{}\\
	Let $Q$ be a metrizable convex compact space and $\Gamma\to\Aff(Q)$ a continuous homomorphism.
	Then there is a continuous affine bijection
	\[
		\Map_w(X,Q)\cong \Stat(X_+,Q),\qquad \phi\ \longleftrightarrow\ \psi
	\]
	that in one direction is given by the formula
	\[
		\psi(y)=\int_X \phi(x)\dd \mu_{y}(x)\qquad (y\in X_+),
	\]
	and in the opposite direction by a.e weak-* convergence:
	\[
		\begin{split}
			\phi(x)&=\lim_{n\to\infty}\ w_{-n}(x)^{-1}.\psi(p_+(T^{-n}x))\\
			&=\lim_{n\to\infty}\ w(T^{-1}x)w(T^{-2}x)\cdots w(T^{-n}x).\psi(p_+(T^{-n}x)).
		\end{split}
	\]
\end{theorem}
\begin{proof}
	Given $\phi\in \Map_w(X,Q)$ we define $\psi\in \Map(X_+,Q)$ using the formula
	\[
		\psi(y)=\int_X \phi(x)\dd \mu_{y}(x)\qquad (y\in X_+).
	\]
	Maps $X_+\to Q$ can be identified with $\calF_0^\infty$-measurable maps $X\to Q$, and in this proof 
	we write $\psi(x)$ instead of $\psi(p_+(x))$ to simplify notations.
	Thus $\psi$ is given by the conditional expectation: $\psi=\bbE(\phi \mid \calF_0^\infty)$.
	Using equivariance of $\phi$ in the form $\phi=w_{-1}^{-1}.\phi\circ T^{-1}$ and the fact 
	that $w_{-1}^{-1}$ is $\calF_{-1}^\infty$-measurable we have the relation
	\[
		\begin{split}
			\psi&=\bbE(\phi \mid \calF_0^\infty)=\bbE(\bbE(\phi \mid \calF_{-1}^\infty) \mid \calF_0^\infty)
			=\bbE(\bbE(w_{-1}^{-1}.\phi\circ T^{-1} \mid \calF_{-1}^\infty) \mid \calF_0^\infty)\\
			&=\bbE(w_{-1}^{-1}.\bbE(\phi \mid \calF_{0}^\infty)\circ T^{-1} \mid \calF_0^\infty)
			=\bbE(w_{-1}^{-1}.\psi\circ T^{-1} \mid \calF_0^\infty)
		\end{split}
	\]
	that can be rewritten as
	\[
		\psi(y)=\int_X w(T^{-1}x).\psi(p_+(T^{-1}x))\dd \mu_{y}(x)\qquad (y\in X_+).
	\]
	In other words, $\psi$ is stationary: $\psi\in \Stat(X_+,Q)$. 
	
	Conversely, we consider $\psi\in\Stat(X_+,Q)$, viewed as a $\calF_0^\infty$-measurable map $X\to Q$.
	The fact that it is stationary is expressed by the relation
	\[
		\psi=\bbE(w_{-1}^{-1}.\psi\circ T^{-1} \mid \calF_0^\infty).
	\]
	Consider the sequence of maps $\psi_{n}:=\bbE(\psi \mid \calF_{-n}^{\infty})$ with $n\ge 0$. Then 
	\[
		\bbE(\psi_{n+1} \mid \calF_{-n}^{\infty})=\bbE(\bbE(\psi \mid \calF_{-n-1}^{\infty}) 
			\mid \calF_{-n}^{\infty})
		=\bbE(\psi \mid \calF_{-n}^{\infty})=\psi_n.
	\]
	Thus $\{\psi_n\}$ form a martingale, and by Theorem~\ref{T:MCT} for $m$-a.e $x\in X$ there is a weak-* convergence
	\[
		\phi:=\lim_{n\to\infty} \psi_n(x).
	\]
	Note that since $w_{-n}$ and $\psi\circ T^{-n}$ are $\calF_{-n}^\infty$-measurable  
	\[
		\psi_n=\bbE(\psi \mid \calF_{-n}^{\infty})=\bbE(w_{-n}^{-1}.\psi\circ T^{-n} 
		\mid \calF_{-n}^{\infty})=w_{-n}^{-1}.\psi\circ T^{-n}
	\]
	so we have $m$-a.e weak-* convergence as $n\to\infty$
	\[
		w_{-n}(x)^{-1}.\psi(T^{-n}x)=\psi_n(x)\ \to \ \phi(x).
	\]
	Since $\psi_{n+1}(x)=w_{-1}(x)^{-1}.\psi_n(T^{-1}x)$ it follows that $\phi(x)=w_{-1}(x)^{-1}.\phi(T^{-1}x)$.
	Hence $\phi$ is $w$-equivariant, i.e. $\phi\in\Map_w(X,Q)$.
	
	Finally, the two maps between $\Map_w(X,Q)$ and $\Stat(X_+,Q)$ described above are affine 
	and inverses of each other.
	They are also continuous, since pointwise convergence of uniformly bounded functions 
	implies convergence in measure.
\end{proof}

The following lemma will be useful later on.

\begin{lemma} \label{lem:convex}
Assume the Greg $(X,\calX,m,T,w,\Gamma)$ is ergodic.
Let $Q$ be a non-empty metrizable convex compact space and $\Gamma\to\Aff(Q)$ a continuous homomorphism.
Let $h:Q\overto{} \mathbb{R}$ be a $\Gamma$-invariant bounded measurable convex map.
Then for every map $\psi\in \Stat(X_+,Q)$, $h\circ\psi$ is essentially constant.

More generally, Let $Q_0\subset Q$ be a measurable $\Gamma$-invariant convex subset and 
$h_0:Q_0\to \mathbb{R}$ be a $\Gamma$-invariant bounded measurable convex map.
If $\psi\in \Stat(X_+,Q)$ is such that $m_+(\psi^{-1}(Q_0))=1$ then $h_0\circ\psi$ is essentially constant.
\end{lemma}

\begin{proof}
This is an immediate application of Lemma~\ref{L:Markov}
taken with $Y=X_+$ and $f=h_0\circ \psi$.
%
%
\end{proof}

\subsection{Deterministic stationary maps}\hfill{}\\
A stationary map $\psi\in \Stat(X_+,Q)$
is said to be determinstic if it is $w$-equivariant, that is $\psi\in \Map_w(X_+,Q)$.

\begin{example}
Consider a rotation action of $\mathbb{Z}$ on $X=\mathbb{R}/\mathbb{Z}$, 
$Tx=x+\alpha$ for some irrational $\alpha\in \mathbb{R}$.
Let $\Gamma=\mathbb{R}/\mathbb{Z}$ and $w:X\to \Gamma$ be the identity map.
Check that $X,X_\pm,E$ and $E_\pm$ are all isomorphic to the action of $\Gamma$ on itself and the maps $p_\pm$ and $\pi_\pm$ are isomorphisms.
Consider $\psi:X_+\to \Prob(\mathbb{R}/\mathbb{Z})$ to be the map taking a point to the corresponding Dirac measure.
This is a determinstic stationary map.
\end{example}

We will see in Lemma~\ref{lem:constcrit} that under the Apafi condition, every deterministic stationary map is trivial.
Below we give a criterion for a stationary map to be deterministic.

\begin{lemma} \label{l:deter}
Let $Q$ be a non-empty metrizable convex compact space and $\Gamma\to\Aff(Q)$ a continuous homomorphism.
Fix a stationary map $\psi\in \Stat(X_+,Q)$.
If the image of $\psi$ is a.e an extreme point of $Q$ then $\psi$ is determinstic,
that is $\psi\in \Map_w(X_+,Q)$.

More generally, 
let $Q_0\subset Q$ be a measurable $\Gamma$-invariant convex subset
such that image of $\psi$ is a.e in $Q_0$. 
If the image of $\psi$ is a.e an extreme point of $Q_0$ then $\psi$ is determinstic.
Assume, moreover, that the Greg $(X,\calX,m,T,\Gamma,w)$ is ergodic.
If $m_+(\psi^{-1}(\ext(Q_0)))>0$ then $\psi$ is determinstic.
\end{lemma}

\begin{proof}
We prove the ``more general" and ``moreover" case.
We assume as we may that $\psi(X_+)\subset Q_0$
and conclude from the 
statoinarity equation for $y\in X_+$
	\[
		\psi(y)=\int_{X} w(T^{-1}x).\psi(p_+(T^{-1}x))\dd \mu_{y}(x)
	\]
that if $\psi(y)\in \ext(Q_0)$ then for $\mu_y$-a.e $x\in X$, $\psi\circ p_+(x)=w(T^{-1}x).\psi\circ p_+(T^{-1}x)$.
Applying Lemma~\ref{lem:convex} to the characteristic function of $\ext(Q_0)\subset Q_0$, $h=\chi_{\ext(Q_0)}:Q_0\to [0,1]$, 
we get that $m_+(\psi^{-1}(\ext(Q_0)))=1$.
Thus $m((\psi\circ p_+)^{-1}(\ext(Q_0)))=1$ and we cocnlude that $\psi\circ p_+$ is in fact a $w$-map.
It follows that indeed, $\psi$ is a $w$-map.
\end{proof}

\subsection{Stationary measures}\label{sub:examples-of-stat}\hfill{}\\
The following example illustrates the concept of stationarity and will be used later 
to determine the Lyapunov spectra in certain situations.

\begin{example}[Stationary measures] \label{ex:sm} \hfill{}\\
	Let $M$ be a compact metrizable space, $\rho:\Gamma\to \Homeo(M)$ a continuous homomorphism and consider
	the space $X\times M$ with the transformation 
	\[
		S:X\times M\overto{} X\times M, \qquad S:(x,\xi)\mapsto (Tx,\rho\circ w(x).\xi).
	\]
	Let $\Prob_m(X\times M)$ denote all probability measures on $X\times M$ 
	that project to $m$.
	Using disintegration, such measures are in an affine bijection with $\Map(X,\Prob(M))$,
	where the closed convex subspace $\Prob_m(X\times M)^S$ of $S$-invariant measures corresponds to the
	closed convex subspace $\Map_w(X,\Prob(M))$ of equivariant maps.
	Let 
	\[
		\Map_w(X,\Prob(M))\ni\phi\longleftrightarrow \psi\in \Stat(X_+,\Prob(M))
	\] 
	be as in Theorem~\ref{T:structure-of-stat}, and define
	\[
		\bar{m}:=\int_X \delta_x\times \psi(x)\dd m(x),\qquad \bar{m}_+:=\int_{X_+} \delta_y\times \psi(y)\dd m_+(y).
	\]
	Then the following diagram commutes:
	\[
		\begin{tikzcd} 
			(X\times M,\bar{m},S) \ar[d] \ar[r]  &  (X_+\times M,\bar{m}_+,S_+) \ar[d]\\
	               (X,m,T) \ar[r, "p_+"] & (X_+,m_+,T_+).
		\end{tikzcd}
	\]
	The horizontal rows describe natural extensions of the non-invertible systems in the RHS. 
	In this case $\{\psi(y)\in\Prob(M)\}_{y\in X_+}$ can be referred to as a system
	of \textbf{stationary measures} (cf. Ledrappier \cite{Ledrappier}).
	Ergodicity of the top systems can be proved to be equivalent to $\phi$ and $\psi$ being extremal points of
	the convex compact spaces $\Map_w(X,\Prob(M))$ and $\Stat(X_+,\Prob(M))$, respectively.
	%
\end{example}

\subsection{Past oriented stationary maps}\label{subs:past-oriented}\hfill{}\\
We conclude this section by introducing a key definition and 
listing some of the correspondences between the various spaces we had discussed.

\begin{defn}\label{D:good-stat}\hfill{}\\
Denote by $\Stat_-(X_+,Q)\subset \Stat(X_+,Q)$ the subset that corresponds to the non-empty 
convex compact subset $\Map_w(X_-,Q)\subset \Map_w(X,Q)$.
The elements of this subset are called \textit{past oriented stationary maps}.
\end{defn}

\begin{theorem}\label{T:equi-inv}\hfill{}\\
Let $Q$ be a non-empty metrizable convex compact space and $\Gamma\to\Aff(Q)$ a continuous homomorphism.
Then the following is a commutative diagram of non-empty metrizable compact convex spaces
and continuous affine maps:
\[
    \begin{tikzcd}
			\Map_\Gamma(E_-,Q) \ar[d, hook]\ar[r, "\cong"] & \Map_w(X_-,Q) \ar[r, "\cong"]\ar[d, hook] & \Stat_-(X_+,Q) \ar[d, hook]\\
			\Map_\Gamma(E,Q) \ar[r, "\cong"] & \Map_w(X,Q) \ar[r, "\cong"] & \Stat(X_+,Q).
    \end{tikzcd}
\]
The vertical maps are the obvious inclusions stemming from the maps $E\to E_-$, $X\to X_-$ and $X\to X_+$,
while the horizontal maps are the ones given in Lemma~\ref{L:equi-inv} and Theorem~\ref{T:structure-of-stat}.
These diagrams are natural with respect to $Q$,
that is a $\Gamma$-affine map $Q\to Q'$ gives rise to a map of diagrams in the obvious way.
\end{theorem}

\begin{proof}
As indicated in the formulation of the theorem,
the commutativity of the diagram follows from Lemma~\ref{L:equi-inv} and Theorem~\ref{T:structure-of-stat}.
The spaces are all non-empty, as $\Map_\Gamma(E_-,Q)\neq\emptyset$ by Lemma~\ref{L:amen}.
The facts that all the arrows represent continuous affine maps as well as the naturality statement are immediate and left to the reader.
\end{proof}

\section{Apafic Gregs}\label{sec:ApaficGregs}

In this section we work with a fixed Greg $(X,\calX,m,T,w,\Gamma)$
as in Definition~\ref{defn:Greg} and retain the notations introduced 
in the previous section.
In particular we have the past and future factors $X_-,X_+$, 
the spaces of ideal past and futures $E_-,E_+$, 
and the associated space $E$.
Throughout this section we assume further that 
the fixed Greg satisfies the asymptotic 
past and future independence condition
given in Definition~\ref{D:Apafi},
that is that our \textit{Greg is Apafic}.

\subsection{From Gregs to boundary systems} \label{subsec:gregtobs}\hfill{}\\
The main significance of the Apafi assumption is the following.

\begin{theorem} \label{T:BSfromGreg}
The triple of $\Gamma$-Lebesgue spaces $E,E_+,E_-$, together with the maps $p_\pm:E\to E_\pm$ forms a boundary system,
see Definition~\ref{def:BS}.
\end{theorem}

\begin{proof}
The amenability of the $\Gamma$-actions on $E$ and $E_\pm$ is given in Lemma~\ref{L:amen}
and the relative metric ergodicity of $\pi_\pm$ is given in Theorem~\ref{T:rME}.
By the Apafi condition we have that, indeed, $\pi_-\times \pi_+$ is measure class preserving.
\end{proof}

Some ergodicity results are implied automatically.

\begin{lemma} \label{L:EIPFME}
The $\Gamma$ spaces $E_+,E_-$ and $E$ are all metrically ergodic.
If $(V,d)$ is a separable metric space on which $\Gamma$ acts isometrically 
then every $f\in \Map_w(X,V)$
is essentially constant.
In particular, $X$, hence also $X_-$ and $X_+$, are ergodic systems.
\end{lemma}

\begin{proof}
The first claim follows from Lemma~\ref{L:BSME}
and the second now follows by Lemma~\ref{L:equi-inv}.
The last statement follows by Remark~\ref{r:Mackey}.
\end{proof}

\subsection{The atomic part of stationary measures}\hfill{}\\
We keep the assumption that our Greg
$(X,\calX,m,T,w,\Gamma)$ is Apafic.

\begin{lemma} \label{lem:constcrit}
Let $Q$ be a non-empty metrizable convex compact space
and let $\Gamma\to\Aff(Q)$ a continuous homomorphism.
The following are criteria for triviality of $Q$-valued maps.
\begin{enumerate}
\item
A map $f\in \Map_w(E,Q)$ is essentially constant with a $\Gamma$-invariant essential image if and only if
there exist $f_\pm\in \Map_w(E_\pm,Q)$ such that $f=f_\pm\circ \pi_\pm$.
\item
A map $\phi\in \Map_w(X,Q)$ is essentially constant with a $\Gamma$-invariant 
essential image if and only if
there exist $\phi_\pm\in \Map_w(X_\pm,Q)$ such that $\phi=\phi_\pm\circ p_\pm$.
\item
A map $\phi_+\in \Map_w(X_+,Q)$ is essentially constant with a $\Gamma$-invariant 
essential image if and only if
$\phi_+\in \Stat_-(X_+,Q)$.
\item
A map $\psi\in \Stat_-(X_+,Q)$ is essentially constant with a $\Gamma$-invariant
essential image if and only if
there exists a measurable $\Gamma$-invariant convex subset $Q_0\subset Q$ such that
\[
    m_+(\psi^{-1}(Q_0))=1,\qquad m_+(\psi^{-1}(\ext(Q_0)))>0.
\]
In particular, this is the case if $m_+(\psi^{-1}(\ext(Q)))>0$.
\end{enumerate}
\end{lemma}

\begin{proof}
In all statements the ``only if" part is trivial. We argue below to show the ``if" parts.

In the first statement, $f=f_-\circ \pi_-=f_+\circ\pi_+$,
thus $f$ factors via $E_+\times E_-$, by the Apafi assumption,
that is $f=\bar{f}\circ (\pi_+\times \pi_-)$ 
for some $\bar{f}\in \Map_\Gamma(E_+\times E_-,Q)$.
We get that $\bar{f}=f_\pm\circ \pr_\pm$, where $\pr_\pm:E_+\times E_-\to E_\pm$ are the coordinate projections,
which implies that $\bar{f}$ is essentially constant with a $\Gamma$-invariant essential image, thus so is $f$ as well.

In the second statement, we consider the corresponding maps $f\in \Map_w(E,Q)$ and $f_\pm\in \Map_w(E_\pm,Q)$
given by Lemma~\ref{L:equi-inv} and observe that $f=f_\pm\circ \pi_\pm$.
We apply the first statement to $f$ and conclude that the result holds also for $\phi$.

In the third statement we define $\phi=\phi_+\circ p_+\in \Map_w(X,Q)$ and let 
$\phi_-\in \Map_w(X_-,Q)$ be the map given in Theorem~\ref{T:structure-of-stat},
thus for a.e $x\in X$,
\[ \phi_-(p_-(x))=\lim_{n\to\infty}\ w_{-n}(x)^{-1}.\phi(T^{-n}x)
=\lim_{n\to\infty}\ \phi(x)=\phi(x). \]
We apply the second statement to $\phi$ and conclude that the result holds also for $\phi_+$.

The fourth statement follows from the third by Lemma~\ref{l:deter}.
%
\end{proof}

Specializing to $Q=\Prob(M)$ where $M$ is a compact $\Gamma$-space,
we get from the special case discussed in Lemma~\ref{lem:constcrit}(4) 
that if the values of $\psi\in \Stat_-(X_+,\Prob(M))$
are Dirac measures on a non-null subset of $X_+$ then $\psi$ is essentially constant
and its image is the Dirac measure at a $\Gamma$-invariant point of $M$. 
In fact, more is true as we will see in first in 
Remark~\ref{r:uniform} and then in Proposition~\ref{p:non-atomic} below.

\begin{defn}
For a measure space $\Omega$ and $n\in\bbN$, a map $f\in \Map(\Omega,\Prob(M))$ 
is said to be {\bf $n$-uniform} if for a.e $z\in \Omega$, $f(z)$ is a uniform measure on $n$-points.
It is said to be {\bf finite-uniform} if it is $n$-uniform for some $n$.
\end{defn}

\begin{example} \label{ex:uniform}
If $M$ is a $\Gamma$-space and $\Omega$ is an ergodic $\Gamma$-space 
then every map in $\Map_\Gamma(\Omega,\Prob(M))$ which values have a non-tirvial atomic part on a non-null subset of $\Omega$ is
finite-uniform.
indeed, this follows easily from the measurability of $\maxev$ and $\maxpart$,
see Lemma~\ref{l:maxev} and Lemma~\ref{l:maxpart}.
\end{example}

\begin{remark} \label{r:uniform}
If $\psi\in \Stat_-(X_+,\Prob(M))$ is finite-uniform
then it is essentially constant
and its image is a finite $\Gamma$-invariant subset of $M$. 
This follows from Lemma~\ref{lem:constcrit}(4),
as the $n$-uniform measures on $M$ are the extreme points of $Q_0=\maxev^{-1}([0,1/n])$,
see Lemma~\ref{l:maxev} and the discussion before.
\end{remark}

The following result is a generalization of this remark.

\begin{prop} \label{p:non-atomic}
Let $M$ be a compact metrizable $\Gamma$-space
which admits no finite invariant subset
and fix a map $\psi\in \Stat_-(X_+,\Prob(M))$.
Then for a.e $x\in X_+$, the measure $\psi(x)$
has no atomic part.
\end{prop}

For the proof of Proposition~\ref{p:non-atomic}
we will need the following.

\begin{lemma} \label{l:atomic}
Let $M$ be a compact metrizable $\Gamma$-space
and fix a stationary map $\psi\in \Stat(X_+,\Prob(M))$.
If the values of $\psi$
have non-trivial atomic part on a non-null subset of $X_+$ then
$\psi$ could be represented as a convex combination $\psi=t\psi_0+(1-t)\psi_1$,
where $\psi_0\in \Stat(X_+,\Prob(M))$ is a deterministic finite-uniform map and $t\in(0,1]$.

In particular, if $\psi$ is extremal in $\Stat(X_+,\Prob(M))$, then $\psi$ is 
a deterministic finite-uniform map. 
Furthermore, if $\psi\in \Stat_-(X_+,\Prob(M))$,
then $\psi$ is essentially constant,
and its image is a finite $\Gamma$-invariant subset of $M$. 
\end{lemma}

\begin{proof}
We will use the notation and results of \S\ref{subsection:measures}.
In particular, we let $Q$ be the compact convex $\Gamma$-space 
consisting of positive measures on $M$
with total mass in $[0,1]$ and view $\Prob(M)$ as a $\Gamma$-invariant 
convex compact subset of $Q$,
thus $\psi\in \Stat(X_+,Q)$.

By Lemma~\ref{l:maxev}, the map
\[ 
    \maxev:Q\to [0,1], \quad \maxev(\nu)=\max\setdef{\nu(\{\xi\})}{ \xi\in M} 
\]
is convex and measurable, thus by Lemma~\ref{lem:convex} it is essentially constant.
We denote this constant by $s$
and note that by the assumption on $\psi$, $s>0$.
We denote $Q_s=\maxev^{-1}([0,s])$.
By Lemma~\ref{l:Qalpha} it is a $\Gamma$-invariant measurable convex subset of $Q$
such that $m_+(\psi^{-1}(Q_s))=1$.

We consider the map 
\[ 
    \maxpart:Q_s\overto{} \ext(Q_s) 
\]
taking a measure $\nu\in Q_s$ to the maximal measure in $\ext(Q_s)$ 
which is dominated by $\nu$.
This map is measurable by Lemma~\ref{l:maxpart}.
Precomposing this map with $\psi$ and extending its codomain to $Q$ we get the map
\[
    \maxpart\circ\psi:X_+\overto{} Q.
\]
The stationarity equation for $\psi$, together with an obvious extremality argument,
gives that this map is stationary.
By Lemma~\ref{l:deter}, this map is deterministic.
Applying Lemma~\ref{lem:convex} to $\maxpart\circ\psi$,
taking $h$ to be the total mass, we get that the total 
mass of the values of $\maxpart\circ\psi$
are essentially a non-zero constant which we denote by $t$.
If $t=1$ then $\psi=\maxpart\circ\psi$ is a deterministic finite-uniform map 
and there is nothing to prove.
Otherwise, we denote  
\[ 
    \psi_0=\frac{1}{t}\cdot\maxpart\circ\psi, \qquad 
    \psi_1=\frac{1}{1-t}\cdot (\psi-\maxpart\circ\psi)
\]
and get the decomposition $\psi=t\psi_0+(1-t)\psi_1$.
We note that $\psi_0$ is deterministic, as $\maxpart\circ\psi$ is,
and it is $n$-uniform for $n=t/s$.
This proves the first part of the lemma.
The last part of the lemma follows from Remark~\ref{r:uniform}.
%
%
%
%
\end{proof}

\begin{proof}[Proof of Proposition~\ref{p:non-atomic}]
We will prove the contrapositive: we will assume that the values of $\psi$
have a non-trivial atomic parts on a non-null subset of $X_+$, 
and argue to show $M$ admits a finite invariant subset.

As the evaluation maps $\ev_m$ are convex (in fact, affine),
there must be also an extreme point of $\Stat_-(X_+,\Prob(M))$
such that its values
have non-trivial atomic parts on a non-null subset of $X_+$.
We thus assume as we may that $\psi$ is extremal in $\Stat_-(X_+,\Prob(M))$.
We note that if we knew that $\psi$ is extremal also in $\Stat(X_+,\Prob(M))$
then we would be done by the last part of Lemma~\ref{l:atomic},
as the essential image of $\psi$ is a finite $\Gamma$-invariant subset of $M$. 
However, this lemma gives only that $\psi=t\psi_0+(1-t)\psi_1$
for some $t>0$, where $\psi_0,\psi_1\in \Stat(X_+,\Prob(M))$ and $\psi_0$ is
a deterministic finite-uniform map.
We will show that this is enough for our proof.

Using Theorem~\ref{T:equi-inv}, we find maps corresponding to $\psi$ and $\psi_i$,
\[ 
    \begin{split}
        &\phi\in \ext(\Map_w(X_-,\Prob(M))) \quad \mbox{and} \quad 
            f\in \ext(\Map_\Gamma(E_-,\Prob(M))),\\
        &\phi_i\in \ext(\Map_w(X,\Prob(M))) \quad \mbox{and} \quad 
            f_i\in \ext(\Map_\Gamma(E,\Prob(M))).
    \end{split}
\]
which satisfy similar relations:
$\phi=t\phi_0+(1-t)\phi_1$ and $f=tf_0+(1-t)f_1$.
Since $\psi_0$ is determinstic we get $\phi_0=\psi_0$
and in particular $\phi_0$ is finite-uniform.
It follows from Lemma~\ref{L:equi-inv} that $f_0$ is finite-uniform as well.
We conclude in particular that $f$ has a non-trivial atomic part.
By the ergodicity of $E_-$, we get that also $f$ is $n$-uniform for some natural $n$, see Example~\ref{ex:uniform}.

If $n=1$ then $f$ takes values in the Dirac measures on $M$, that is $\ext(\Prob(M))$,
and the same goes to the map $f\circ \pi_-:E\to \Prob(M)$.
thus $f\circ \pi_-$ is extremal in the space $\Map(E,\Prob(M))$.
In particular $f\circ \pi_-$ is extremal in its subspace $\Map_\Gamma(E,\Prob(M))$ 
and by Theorem~\ref{T:equi-inv}, $\psi$ is extremal in the space 
$\Stat(X_+,\Prob(M))$.
As mentioned above, in this case we are done by the last part of Lemma~\ref{l:atomic}.
Below we will exaplain how to treat the general case, 
essentially by reducing it to the case $n=1$.

We consider the space of unordered $n$-tuples in $M$ (possibly with non-trivial multiplicities),
$M_n=M^n/S_n$ and the map $\delta_n:M_n\to \Prob(M)$ taking a tuple to the uniform measure it supports
(taking multiplicities into account, on the diagonals). Note that this map is an isomorphism on its (closed) image.
As the image of $f$ is in $\delta(M_n)$ a.e, we have $\delta^{-1}\circ f:E_-\to M_n$ and identifying  points of $M_n$ with Dirac 
measures on $M_n$, we get a map $f'\in \Map_\Gamma(E_-,\Prob(M_n))$.
Note that $f'\circ\pi_-$ is extremal in $\Map_\Gamma(E,\Prob(M_n))$, as it takes values in Dirac measures.
Using Theorem~\ref{T:equi-inv}, we find corresponding maps
\[ \phi'\in \Map_w(X_-,\Prob(M_n)) \quad \mbox{and} \quad 
\psi'\in \Stat_-(X_+,\Prob(M_n)) \]
and we note that $\psi'$ is extremal in $\Stat(X_+,\Prob(M_n))$, 
by Theorem~\ref{T:equi-inv}.
We extend the map $\delta_n$ to $(\delta_n)_*:\Prob(M_n)\to \Prob(M)$ 
by pushing measures and taking barycenters.
By construction we have $f=(\delta_n)_*\circ f'$, thus $\phi=(\delta_n)_*\circ \phi'$ 
and $\psi=(\delta_n)_*\circ\psi'$.
Since the values of $\psi$
have a non-trivial atomic parts on a non-null subset of $X_+$, we deduce that the same holds for $\psi'$.
We are done by the last part of Lemma~\ref{l:atomic},
as the essential image of $\psi'$ is a finite $\Gamma$-invariant subset of $M$. 
\end{proof}

\section{Semi-simple Lie groups}\label{sec:semisimple}

In this section we fix a connected semisimple real algebraic group $\mathbf{G}$ 
and consider its group of real points $G=\mathbf{G}(\mathbb{R})$.
We let $K<G$ be a fixed maximal compact subgroup and $P<G$ be a minimal parabolic subgroup.
We fix a maximal split torus $A<P$ and let $Z$ denote its centralizer in $G$.
Note that $Z<P$.

\subsection{The flag variety $G/P$ and the space $G/Z$.}\hfill{}\\
The minimal parabolic subgroup $P<G$ is solvable by compact, hence amenable,
and it is cocompact. In fact, $G/P$ is (the real points of) a projective variety,
called the {\bf flag variety} or the {\bf Furstenberg boundary} of $G$.
It will play an important role in our discussion.

The inclusion $Z<P$ gives rise to a natural map $\pr_+:G/Z\overto{} G/P$.
We let $W=N_G(A)/C_G(A)=N_G(A)/Z$ be the {\bf Weyl group} associated with $A$ 
and note that it acts on $G/Z$ by $G$-automorphisms:
$w \in W$ gives the $G$-map $G/Z\to G/Z$, $gZ\mapsto g\tilde{w}Z$,
where $\tilde{w}\in N_G(A)$ is a lift of $w$.
By a standard abuse of notation, we denote this automorphism of $G/Z$ by $w$ as well.

For every $w\in W$, we get a $G$-map
\[ 
    \Phi_w=(\pr_+,\pr_+\circ w):G/Z\overto{} G/P\times G/P. 
\]
We denote the corresponding image by $B_w\subset G/P\times G/P$.
These subsets are called the {\bf Bruhat cells}.
Note that $B_e$ is the diagonal in $G/P\times G/P$.
The following lemma is well known.

\begin{lemma}
There exists a unique element $w_0\in W$ such that $B_{w_0}$
is open and dense in $G/P\times G/P$.
\end{lemma}

This element $w_0$ is called the {\bf longest element} of $W$ and $B_{w_0}$ is usually refered to as 
the {\bf big Bruhat cell}.
In this case we denote 
\begin{equation}
\pr_-=\pr_+\circ w_0:G/Z\overto{} G/P.
\end{equation}

We let $A^o<A$ be the connected component of the identity, $\mathfrak{a}$ be its Lie algebra and
\[ 
    \log:A^o\overto{}\mathfrak{a}, \qquad \exp:\mathfrak{a}\overto{} A^o<A<Z 
\]
be the corresponding logarithm and exponent maps.
We set $M=Z\cap K$ and note that this is a compact normal subgroup of $Z$ such that $Z=MA^o$
and $M\cap A^o=\{e\}$. Therefore, we get a natural identification $Z/M\simeq A^o$.
Precomposing with $Z\to Z/M$ and postcomposing with the logarithm map, 
we get a continuous homomorphism
\begin{equation}\label{eq:q}
    q:Z\overto{} \mathfrak{a} \qquad \mbox{such that} \qquad
    q\circ \exp =\id:\mathfrak{a} \overto{} \mathfrak{a}.
\end{equation}
The map $q$ and its section $\exp$ give an isomorphism $Z\simeq M\times \mathfrak{a}$.

Recall that a bi-invariant metric on a group is a metric which is invariant under both left and right translation and
observe that both compact group and abelian groups admit such bi-invariant metrics.
We get the following.

\begin{lemma}
The group $Z$ admits a bi-invariant metric.
\end{lemma}

\subsection{Roots, weights and the order on $\mathfrak{a}$.}\hfill{}\\
We denote by $\mathfrak{g}$ the Lie algebra of $G$
and let $\Sigma< \mathfrak{a}^*$ by its root system corresponding to the choice of $A$.
That is $\chi\in \Sigma$ iff $\chi\neq 0$ and the weight space
\[ \mathfrak{g}_\chi=\{v\in \mathfrak{g} \mid \forall h\in \mathfrak{a},~[h,v]=\chi(h)v \} \]
is non-trivial.
We denote by $\Sigma^+$ the subset of $\Sigma$ consisting roots which are positive with respect to $P$,
that is $\chi\in \Sigma$ such that $\mathfrak{g}_\chi$ is in the Lie algebra of $P$.
We denote by 
\[
	\mathfrak{a}^+:=\setdef{x\in \mathfrak{a}}{\forall \alpha\in \Sigma^+,~ \alpha(x)\ge 0}
\]
the corresponding {\bf positive Weyl chamber}
and we let $\Pi\subset \Sigma^+$ be the set of simple roots,
that is roots $\alpha\in \Sigma^+$ such that $\ker\alpha\cap \mathfrak{a}^+$ is a face of $\mathfrak{a}^+$.
We let $\mathfrak{a}^{++}$ be the interior of $\mathfrak{a}^+$, that is
\[
	\mathfrak{a}^{++}:=\setdef{x\in \mathfrak{a}}{\forall \alpha\in \Sigma^+,~ \alpha(x)> 0}
=\setdef{x\in \mathfrak{a}}{\forall \alpha\in \Pi,~ \alpha(x)> 0}.
\]
We endow $\mathfrak{a}$ with the {\bf Killing form}
and let $\mathfrak{a}_+$ be the dual cone of $\mathfrak{a}^+$,
that is 
\[ \mathfrak{a}_+=\{v\in \mathfrak{a} \mid \forall u\in\mathfrak{a}^+,~\langle u,v \rangle \geq 0\}. \]
We define a partial order on $\mathfrak{a}$ by
\[ u\leq v \quad \Longleftrightarrow \quad v-u\in \mathfrak{a}_+. \]
The {\bf Weyl group} $W$ acts on $\mathfrak{a}$ and $\mathfrak{a}^+$ is a fundamental domain for this action.
We use the convention that $\bar{v}$ denotes the representative of $v\in \mathfrak{a}$ in $\mathfrak{a}^+$,
that is $Wv\cap \mathfrak{a}^+=\{\bar{v}\}$.
It is a standard fact that $\bar{v}$ is the largest element in the set $Wv$.

An element of $\mathfrak{a}^*$ is called a {\bf weight}.
Upon identifying $\mathfrak{a}$ with $\mathfrak{a}^*$ via the Killing form,
$\mathfrak{a}_+$ corresponds to the cone of {\bf positive weights} 
and $\mathfrak{a}^+$ corresponds to the cone of {\bf dominant weights} in $\mathfrak{a}^*$.
We thus have 
	\[ v\geq 0 \quad  \Longleftrightarrow \quad  \mbox{ for every dominant weight } \chi\in \mathfrak{a}^*,~\chi(v)\geq 0. \]
%
%
%
The set $\Pi$ forms a basis for $\mathfrak{a}^*$ as a vector space and a basis for the cone of positive weights as a cone.
The elements of $\Sigma$ have integer coefficients when presented in the basis $\Pi$
and $\Sigma^+$ consists of the roots for which all the coefficients are non-negative in this presentation.
For $\alpha\in \Pi$, we define $\Sigma_\alpha$ to be the subset of $\Sigma$ consisting of roots for which the
coefficient of $\alpha$ in this presentation is non-negative.
We let $\mathfrak{z}$ be the Lie algebra of $Z$ and set 
\[ \mathfrak{q}_\alpha=\mathfrak{z} \oplus \bigoplus_{\chi\in \Sigma_\alpha} \mathfrak{g}_\chi. \]
This is a parabolic subalgebra of $\mathfrak{g}$.
Under the Lie correspondence, it corresponds to the parabolic subgroup $Q_\alpha<G$.

We consider the $G$-space $G/Q_\alpha$ and denote the base point $eQ_\alpha$ by $x_\alpha$.
This is an $A$-fixed point, as $A<Q_\alpha$, thus the tangent space of $G/Q_\alpha$ at $x_\alpha$ becomes an $A$-module
under the differential action.
We get the identification of $A$-modules,
\begin{equation} \label{eq:tangent}
T_{x_\alpha}(G/Q_\alpha) \simeq \bigoplus_{\chi\in \Sigma-\Sigma_\alpha} \mathfrak{g}_\chi.
\end{equation}
Note that $\Sigma-\Sigma_\alpha$ is the set of roots for which the
coefficient of $\alpha$ in is strictly negative.

\begin{lemma}\label{L:UVQ}\hfill{}\\
Fix $\alpha\in \Pi$.
	Let $C\subset G/Q_\alpha$ be a non-empty closed subset not containing $x_\alpha$.
	Then there exists a sequence of open neighborhoods $x_\alpha\in U_k \subset G/Q_\alpha$ 
such that 
for every $z\in Z$,
	\[
		(z.C)\cap U_k\ne \emptyset
		\qquad\Longrightarrow\qquad
		\alpha\circ q(z)>k,
	\]
where $q$ is the homomorphism defined in \eqref{eq:q}.
\end{lemma}

\begin{proof}
Since $x_\alpha$ is $M$-invariant and $M$ is compact, we
can replace $C$ by $MC$ and assume it is $M$-invariant as well.
Since $Z=MA^o$ and $M$ is in the kernel of $q$, it is enough to show 
that there exists $U_k$ such that for every $a\in A^o=\exp(\mathfrak{a})$
\[
		(a.C)\cap U_k\ne \emptyset
		\qquad\Longrightarrow\qquad
		\alpha\circ \log(a)>k.
	\]
This follows easily by \eqref{eq:tangent}.
\end{proof}

Note that for every $\alpha\in \Pi$ we have $P<Q_\alpha$. Thus we have a $G$-map 
\[
    \pi_\alpha:G/P\overto{} G/Q_\alpha,\quad  gP\mapsto gQ_\alpha.
\]
Further, we have $P=\cap_{\alpha\in \Pi} Q_\alpha$. 
Thus the map $G/P\to \prod_{\alpha\in \Pi} G/Q_\alpha$ is injective.
We define the subvariety $L\subset G/P$ by
\begin{equation} \label{eq:defL}
    L=\bigcup_{\alpha\in \Pi} \pi_\alpha^{-1}(\{x_\alpha\}).
\end{equation}
The following lemma is an immediate corollary of Lemma~\ref{L:UVQ}.

\begin{lemma}\label{L:UV}\hfill{}\\
    Let $C\subset G/P$ be a non-empty closed set such that $C\cap L=\emptyset$.
    Then there exists a sequence of open neighborhoods $eP \in V_k \subset G/P$ 
    such that 
    for every $z\in Z$ 
	\[
		(z.C)\cap V_k\ne \emptyset
		\qquad\Longrightarrow\qquad
		\min_{\alpha\in\Pi}\ \alpha\circ q(z)>k,
	\]
    where $q$ is the homomorphism defined in \eqref{eq:q}.
\end{lemma}

\begin{proof}
For $V_k$ we take the intersection of the various preimages of $U_k\subset G/Q_\alpha$,
running over $\alpha\in \Pi$.
\end{proof}

\subsection{The Iwasawa and Cartan decompositions} 
\label{sub:cartan}\hfill{}\\
We denote by $U<P$ the unipotent radical.
It is well known that $KA^oU=G$.
The product map 
\[ K\times A^o\times U \overto{} KA^oU = G \]
is called the {\bf Iwasawa decomposition} of $G$.
It provides the {\bf Iwasawa projection} $\iota:G\to A^o\overto{} \mathfrak{a}$.
The following important result is called {\bf Kostant's convexity theorem}.

\begin{theorem}[{\cite[Theorem~4.1]{Kostant}}] \label{t:kostant}
For $v\in \mathfrak{a}$, $\iota(\exp(v)K)$ coincides with the convex hull of the orbit $Wv$ in $\mathfrak{a}$.
\end{theorem}

For $g\in G$ and $\xi\in G/P$ we find $k\in K$ 
such that $\xi=kP$ and define the map
\begin{equation}\label{e:Iwasawa}
	\sigma:\ G\times G/P\ \overto{} \ \mathfrak{a}, \quad \sigma(g,\xi)=\iota(gk)
\end{equation}
which is a well defined continuous cocycle by \cite{BQ:book}*{Lemma 6.29}.
We denote it the {\bf Iwasawa cocycle}.

We let $A^+=\exp(\mathfrak{a}^+)$. This is a sub-semigroup of $A^o$.
It is well known that $KA^+K=G$.
The product map 
\[ K\times A^+\times K \overto{} KA^+K = G \]
is called the {\bf Cartan decomposition} of $G$.
It provides the {\bf Cartan projection} 
\[
\kappa:G\overto{} A^+\overto{} \mathfrak{a}^+.
\]
\begin{remark} \label{rem:kappa=q}
Since $Z=MA^o$ and $M<K$, we get that for $z\in Z$, $\kappa(z)=\kappa(q(z))$.
In particular, if $q(z)\in A^+$ then $\kappa(z)=q(z)$.
\end{remark}

\bigskip

\subsection{Length functions}

\begin{defn}
Let $H$ be a compactly generated locally compact group.
A {\bf length function} on $H$ is a map $|\cdot|: H\to [0,\infty)$
which is bounded on compact subsets,
satisfying for every 
$h,h'\in G$, $|hh'|\leq |h|+|h'|$.
It is said to  be {\bf symmetric}
if for every 
$h\in G$, $|h|=|h^{-1}|$.

For two functions, $|\cdot|,|\cdot|':H\to [0,\infty)$,
we say that $|\cdot|$ dominates $|\cdot|'$
if there exist $\alpha\geq 1$ and $\beta\geq 0$ such that for every $h\in H$,
$|h|' \leq \alpha |h|+\beta$.
We say that $|\cdot|$ and $|\cdot|'$ are {\bf equivalent}
if both dominate each other.
\end{defn}

Note that by our convention, a length function 
need not be continuous and it might have $|e|>0$ and $|h|=0$ for $h\neq e$.
The word distance with respect to a non-open precompact generating set is an example of a discontinuous length function.
Clearly, the word distance dominates every other length function.
A length function is said to be in the {\bf word class} if it is equivalent to the word distance.

\begin{example}
Let $H$ be the additive group of a finite dimensional real vector space.
An {\bf asymmetric norm} on $H$
is a length function $\|\cdot \|$ which is homogeneous with respect to positive scalars
such that for $v\in H$, $\|v\|=0$ iff $v=0$.
Note that such an asymmetric norm is necessarily symmetric.
Every asymmetric norm on $H$ gives a word class length function.
In this case we use the term {\bf norm class} as a synonymous to word class.
\end{example}

We are now back to consider the semsimple group $G$.
A prominent example of a length function $|\cdot|$ on $G$
is given by $|g|=d(gx,x)$, where $d$ is the Reiamannian metric on the symmetric space $X=G/K$
and $x=eK$. Note that for $v\in \mathfrak{a}$, the function $v \mapsto |\exp(v)|$ gives the 
Hilbertian norm on $\mathfrak{a}$ associated with the restriction of the Killing form of $\mathfrak{g}$ to $\mathfrak{a}$.
In particular, this is a norm class length function on $\mathfrak{a}$.

\begin{lemma} \label{l:word}
A length function $|\cdot|$ on $G$ is in the word class iff
the length function $v \mapsto |\exp(v)|$ on $\mathfrak{a}$ is in the norm class.
\end{lemma}

\begin{proof}
By the Cartan decomposition, a length function $|\cdot|$ on $G$ is equivalent to the 
bi-$K$-invariant function $kak' \mapsto |a|$.
Thus, $|\cdot|$ is determined by its restriction to $A^o$, which could be identified with $\mathfrak{a}$ via $\exp$.
The corresponding length function on $\mathfrak{a}$ is dominated by the norm class.
The symmetric space example shows that it could actually be a norm,
so the class of length functions on $G$ which corresponds to the norm class on $\mathfrak{a}$ is not empty,
thus it must coinside with the word class on $G$, by maximality.
\end{proof}

The following is a fundamental fact.

\begin{prop} \label{p:Wnorm}
For every $W$-invariant asymmetric norm $\|\cdot \|$ on $\mathfrak{a}$,
the function $|g|=\|\kappa(g)\|$ is a word class continuous length function on $G$.
\end{prop}

This proposition follows immediately from the following lemma.

\begin{lemma} \label{l:sigmabound}
Fix an asymmetric norm $\|\cdot\|$ on $\mathfrak{a}$.
Then the function
\[ |\cdot|: G\to [0,\infty), \quad |g|=\sup_{\xi\in G/P} \|\sigma(g,\xi)\| \]
is a word class continuous length function on $G$ and we have 
\[ |\cdot |=\max_{w\in W} \|w\kappa(\cdot)\|. \]
\end{lemma}

\begin{proof}
The function $|\cdot|$ is a composition of the sup norm on $C(G/P,\mathfrak{a})$, given by  
$\|f\|_\infty=\max_{\xi\in G/P} \|f(\xi)\|$,
and the map $G\to C(G/P,\mathfrak{a})$, $g\mapsto \sigma(g,\cdot)$.
The cocylce property of $\sigma$ thus give for every $\xi\in G/P$,
\[ \|\sigma(gg',\xi)\| \leq \|\sigma(g,g'\xi)\|+\|\sigma(g',\xi)\| \leq
\|\sigma(g,\cdot)\|_\infty+\|\sigma(g',\cdot)\|_\infty, \]
thus also 
\[ \|\sigma(gg',\cdot)\|_\infty \leq \|\sigma(g,\cdot)\|_\infty+\|\sigma(g',\cdot)\|_\infty, \]
The continuity of $|\cdot|$ is clear and by continuity, we have boundedness on compact subsets of $G$. 
Therefore $|\cdot|$ is indeed a continuous length function on $G$.

One checks easily that $|\cdot|$ is bi-$K$-invariant, thus for $g\in G$, $|g |=|\exp\circ\kappa(g)|$.
Setting $v=\kappa(g)$ we get by Equation~\eqref{e:Iwasawa} and Theorem~\ref{t:kostant},
\[ |g|=|\exp(v)|=\max_{\xi\in G/P} \|\sigma(\exp(v),\xi)\| = 
\max_{k\in K} \|\iota(\exp(v)k)\|=
\]
\[
\max_{w\in W} \|wv\|=\max_{w\in W} \|w\kappa(g)\|. \]
Finally, we note that $v\mapsto |\exp(v)|=\max_{w\in W} \|wv\|$ is an asymmetric norm on $\mathfrak{a}$,
thus $|\cdot|$ is in a word class length function on $G$, by Lemma~\ref{l:word}.
\end{proof}

Concrete examples of asymmetric norms on $\mathfrak{a}$ are given by dominant weights.

\begin{lemma} \label{l:chipos}
For a dominant weight $\chi \in \mathfrak{a}^*$,
$v\mapsto \chi(\bar{v})$
is an asymmetric norm on $\mathfrak{a}$.
Accordingly, $g\mapsto \chi(\kappa(g))$ is a word class length function on $G$.
\end{lemma}

\begin{proof}
Since $\bar{u}+\bar{v}$ is the maximal element over $(w,w')\in W\times W$ of the collection $wu+w'v$,
while $\overline{u+v}$ is the maximal element over the diagonal subset $w=w'$,
we get $\overline{u+v}\leq \bar{u}+\bar{v}$,
thus $\chi(\overline{u+v}) \leq \chi(\bar{u}+\bar{v})=\chi(\bar{u})+\chi(\bar{v})$.
\end{proof}

We obtain another set of examples of continuous word class length functions on $G$
as follows.

\begin{lemma} \label{l:repnorm}
Fix a real representation with finite kernel $\tau:G\to\GL(V)$
and a norm on $V$
and denote by $\|\cdot \|$ the corresponding operator norms on $\GL(V)$ .
Then the functions
\[ g\mapsto \log^+(\|\tau(g)\|) \quad \mbox{and} \quad
g\mapsto \log(\max\{\|\tau(g)\|,\|\tau(g^{-1})\|\}) \] 
are continuous word class length functions on $G$.
\end{lemma}

\begin{proof}
It is clear that these are continuous length functions on $G$
and we have for $g\in G$, $\log^+(\|\tau(g)\|) \leq \log(\max\{\|\tau(g)\|,\|\tau(g^{-1})\|\})$,
thus it is enough to show that $|\cdot|=\log^+(\|\tau(\cdot)\|)$
is in the word class.
We assume as we may that $\tau$ is irreducible.
We observe that the class of $|\cdot|$ is independent of the choice made for the norm on $V$
and assume that $V$ is endowed with an Hilbertian norm which has an orthonormal basis consisting of 
common eigenvectors of $A$.
We let $\chi$ be the highest weight associated with $\tau$
and observe that the corresponding length function on $\mathfrak{a}$ is the norm
given by $v\mapsto \chi(\bar{v})$.
It follows that $|\cdot|$ is indeed in the word class.
\end{proof}

\begin{remark}
Given any length function $|\cdot|$ on $G$, we get further length functions by 
considering $f(|\cdot|)$, where $f:[0,\infty)\to [0,\infty)$ is a non-decreasing subadditive function.
We thus get examples of length functions which are not in the word class,
We note, however, that every unbounded length function is proper,
by \cite{YCor}.
\end{remark}

The next result is surely known to the experts, however, we couldn't locate a reference for it.

\begin{prop}\label{P:Cartan-subadd}
The Cartan projection $\kappa:G \to \mathfrak{a}^+$ is 
an $\mathfrak{a}^+$-valued length function on $G$ in the sense
that it is bounded on compact sets and 
\[ 
    \kappa(gg')\leq \kappa(g)+\kappa(g')
    \qquad(g,g'\in G)
\]
with respect to the order on $\mathfrak{a}$.
It is symmetric iff the longest element of the Weyl group acts as $-1$ on $\mathfrak{a}$.
\end{prop}

\begin{proof}
By Lemma~\ref{l:chipos}, for every dominant weight $\chi$,
$\chi(\kappa(\cdot))$ is a length function on $G$.
Given $g,g'\in G$, we have 
$\kappa(gg')\leq \kappa(g)+\kappa(g')$
as for every dominant weight $\chi$,
$\chi(\kappa(gg'))\leq \chi(\kappa(g))+\chi(\kappa(g'))=\chi(\kappa(g)+\kappa(g'))$.
\end{proof}

\begin{remark}
The cone $\mathfrak{a}^+$ has an order preserving automorphism given by $v\mapsto -w_0(v)$,
where $w_0$ is the longest element of the Weyl group $W$.
Accordingly, we get the $\mathfrak{a}^+$-valued length function $g\mapsto -w_0(\kappa(g))$ on $G$
and the symmetric $\mathfrak{a}^+$-valued length function $g\mapsto 1/2(\kappa(g)-w_0(\kappa(g)))$.

Each of these give rise to a $G$-invariant  asymmetric $\mathfrak{a}^+$-valued metric on $G/K$,
by $(gK,hK)\mapsto |h^{-1}g|$. In particular, 
\[ 
    d(gK,hK)=\half(\kappa(h^{-1}g)-w_0(\kappa(h^{-1}g))) 
\]
is a symmetric $\mathfrak{a}^+$-valued metric on $G/K$.
\end{remark}

\subsection{Subadditive Ergodic Theorems}\hfill{}\\
A translation invariant partial order on a real vector space $V$ is determined by
the positive cone $V_+=\{v\in V\mid v\geq 0\}$.
A partial order on $V$ is said to be {\bf compatible} if it is translation invariant, 
its positive cone $V_+$ is closed
and it satisfies $V_+\cap (-V_+)=\{0\}$.
Kingman's Subadditive Ergodic Theorem stated below is usually stated for real valued functions.
However, the version we give here follows at once, upon applying positive functionals.

\begin{theorem}[Kingman]\label{T:Kingman}\hfill{}\\
Let $(X,\calX,m,T)$ be an ergodic p.m.p system
and let $V$ be a finite dimensional real vector space endowed with a compatible partial order.
Let $f_n\in L^1(X,V)$ be a sequence satisfying
for every $n,k\in\mathbb{N}$, for a.e $x\in X$,
\[ f_{n+k}(x)\leq f_n(x)+f_k(T^nx). \]
Then there exists a vector $v\in V$ so that there is an $m$-a.e and $L^1$-convergence
\[ 
   v= \lim_{n\to\infty} \frac{1}{n}\cdot f_n. 
\]
\end{theorem}

\begin{cor}\label{C:Oseledets}\hfill{}\\
    Let $(X,\calX,m,T)$ be an ergodic p.m.p system, 
    let $G$ be a connected semisimple real Lie group with a finite center
    and $F:X\to G$ be a measurable map so that $\log\|\Ad\circ F\|\in L^1(X,m)$.
    Then there exists $\Lambda_F\in\mathfrak{a}$ so that there is an $m$-a.e and $L^1$-convergence
    \[
		\Lambda_F=\lim_{n\to \infty} \frac{1}{n}\cdot \kappa\circ F_n.
    \] 
\end{cor}

This corollary is part of the Multiplicative Ergodic Theorem of Oseledets,
that also produces a measurable $F$-equivariant map $\xi_+:X_+\overto{}G/Q_+$, 
and $\xi_-:X_-\to G/Q_-$, 
where $P<Q_-,Q_+<G$ are some parabolic subgroups determined by $\Lambda_F$.
In the case of simple Lyapunov spectrum, i.e. when $\Lambda_F\in\mathfrak{a}^{++}$,
one has $Q_-=Q_+=P$ is the minimal parabolic. 
Oseledets theorem asserts $m$-a.e. and in $L^1$-convergence of the Iwasawa cocycle
\[
    \lim_{n\to \infty} \frac{1}{n}\cdot \sigma(F_n(x),\xi)=\Lambda_F
\]
for any $\xi$ that is in general position with respect to $\xi_+(x)$.
Similarly, denoting by $\iota:\mathfrak{a}^+\to \mathfrak{a}^+$ Cartan involution,
we have
\[
    \lim_{n\to -\infty} \frac{1}{n}\cdot \sigma(F_{n}(x),\xi)=\iota(\Lambda_F)
\]
for any $\xi\in G/P$ in general position with respect to $\xi_-(x)$.
This is a paraphrasing of Oseledets theorem; we refer to 
Kaimanovich \cite{Kaimanovich:DE}, Karlsson--Margulis \cite{KM}, 
Filip \cite{Filip:MET} for the full treatment. 

Here let us point out the following consequence of Oseledets theorem,
and the dynamics of \textit{regular} elements $g$ on $G/P$.
Such an element has $|W|$-many fixed points, where $W$ is the Weyl group of $G$.
\begin{remark}\label{R:XxGPdynamics}
    Let $(X,m,T)$ be an invertible ergodic dynamical system, 
    let $G$ be a semi-simple real Lie group, $F:X\to G$ be
    an integrable measurable map, and assume that the associated Lyapunov spectrum $\Lambda$ is simple. Consider the skew-product $T_F$ on $X\times G/P$,
    given by $T_F(x,\xi)=(Tx,F(x).\xi)$. 
    Then there are precisely $|W|$-many $T_F$-invariant and ergodic 
    probability measures on $X\times G/P$ that project onto $m$ on $X$;
    among these measures there is exactly one that is measurable
    with respect to $X_+$, it is given by the graph of $\xi_+$;
    and exactly one which is measurable with respect to $X_-$,
    it is given by the graph of $\xi_-(x)$.
\end{remark}
Our definition of past-oriented stationary measures 
(see \S\ref{subs:past-oriented})
was motivated by the above observation.

\subsection{The Representation Variety}\hfill{}\\
In this subsection we fix a locally compact group $\Gamma$
and consider the space of continuous homomorphisms $\Hom(\Gamma,G)$.
Note that the group $G$ could be identified with the real points of a real algebraic group, and accordingly, 
in case $\Gamma$ is finitely generated, the
space $\Hom(\Gamma,G)$ is endowed with the structure of an algebraic variety
called the \textbf{representation variety} of $\Gamma$ in $G$.
We refer to $\Hom(\Gamma,G)$ as the representation variety also in the general case.

We endow the space $\Hom(\Gamma,G)$ with the compact-open topology.
A representation $\rho\in \Hom(\Gamma,G)$ is called \textbf{Zariski dense} 
if $\rho(\Gamma)$ is Zariski dense in $G$.

\begin{prop}
\label{p:Zopen}
The subset of $\Hom(\Gamma,\mathbf{G}(\mathbb{C}))$ consisting of Zariski 
dense representations is open
in the compact-open topology.
Accordingly, the subset of $\Hom(\Gamma,G)$ consisting of Zariski 
dense representations is open
in the compact-open topology.
\end{prop}

\begin{remark}
Note that $\Hom(\mathbb{Z},S^1)\simeq S^1$ has a dense subset of representations with finite image,
thus the semisimplicity assumption on $G$ in Proposition~\ref{p:Zopen} is crucial.
\end{remark}

\begin{lemma} \label{l:mas}
Every proper algebraic subgroup of $\mathbf{G}$ is contained in a maximal proper algebraic subgroup and there is a finite collection of conjugacy classes of maximal proper algebraic subgroups.
\end{lemma}

\begin{proof}
Denote by $\mathcal{M}$
the collection of maximal proper subgroups of $\mathbf{G}$ which contain a non-trivial normal subgroup of $\mathbf{G}$.
If $\mathbf{G}$ is simple and adjoint than this collection is empty.
Since $\mathbf{G}$ has a finite number of normal subgroups, it follows by an induction argument that $\mathcal{M}$ has a finite number of conjugacy classes.

Denote by $\mathcal{S}$ the collection of all connected semisimple proper subgroups which are not normal in $\mathbf{G}$ 
and by $\mathcal{R}$ the collection of normalizers of elements in $\mathcal{S}$.
The collection $\mathcal{S}$ has a finite number of conjugacy classes.
Indeed, in every dimension $d$ the number of conjugacy classes of $d$-dimensional elements of $\mathcal{S}$ is bounded by the number of 
irreducible components of the variety of $d$-dimensional subalgebras of the Lie algebra of $\mathbf{G}$, as every conjugation orbit is Zariski open by
\cite[Proposition 12.1]{Richardson}.
Thus also the collection $\mathcal{R}$ has a finite number of conjugacy classes.

Denote by $\mathcal{A}$ the collection of all maximal tori in $\mathbf{G}$ 
and by $\mathcal{N}$ the collection of normalizers of elements in $\mathcal{A}$.
It is a standard fact that the collection $\mathcal{A}$ consists of a single conjugacy class and thus also the collection $\mathcal{N}$ consists of a single conjugacy class.

Denote by $\mathcal{P}$ the collection of all proper parabolic subgroups in $\mathbf{G}$. It is a standard fact that the collection $\mathcal{P}$ has a finite number of conjugacy classes.

By Jordan Theorem, there exists $n$ such that every finite subgroup of $\mathbf{G}$ of order greater than $n$
has a non-central normal abelian subgroup. 
Denote by $\mathcal{F}$ the collection of all finite subgroups of $\mathbf{G}$ of order bounded by $n$.
The collection $\mathcal{F}$ has a finite number of conjugacy classes.
Indeed, there is a finite number of isomorphism classes of groups of order bounded by $n$ and for every such group $L$, every homomorphism into $\mathbf{G}(\mathbb{C})$ could be conjugated into $K$, a maximal compact subgroup, and the compact set $\Hom(L,K)$ has finitely many $K$-conjugation orbits by local rigidity, as $H^1(L,\textrm{Lie}(K))=0$.

The proof now follows as every
proper algebraic subgroup $\mathbf{H}<\mathbf{G}$ is contained in an element of either 
$\mathcal{M}$, $\mathcal{R}$, $\mathcal{N}$, $\mathcal{P}$ or $\mathcal{F}$
and all of these have finitely many conjugacy classes, as explained above.
We will explain this.
Fix $\mathbf{H}<\mathbf{G}$ and 
denote by $\mathbf{H}^o$ its identity component. 
Assume $\mathbf{H}$ is not contained in an element of $\mathcal{M}$.
If $\mathbf{H}$
has a non-trivial unipotent radical then $\mathbf{H}$ is contained in a parabolic subgroup of $\mathbf{G}$.
Assume this is not the case, thus $\mathbf{H}$ is reductive.
If $\mathbf{H}$ normalizes an element of $\mathcal{S}$ then it is contained in an element of $\mathcal{R}$.
Assume this is not the case, thus $\mathbf{H}^o$ is a torus.
If $\mathbf{H}^o$ is contained in a unique maximal torus then $\mathbf{H}$ is contain in an element of $\mathcal{N}$.
Assume this is not the case.
If $\dim(\mathbf{H})>0$ then $\mathbf{H}$ is contained in the normalizer of the derived subgroup of the identity component of the centralizer of $\mathbf{H}^0$, which is an element $\mathcal{S}$, thus $\mathbf{H}$ is contained in an element of $\mathcal{R}$.
Thus we assume as we may that $\mathbf{H}$ is finite.
If the order of $\mathbf{H}$ is greater than $n$ then $\mathbf{H}$ contains a normal non-central abelian subgroup and it is contained in its normalizer which is of positive dimension, and we are reduced to the above cases.
Assuming this is not the case, we get that $\mathbf{H}$ is an element of $\mathcal{F}$.
\end{proof}

\begin{proof}[Proof of Proposition~\ref{p:Zopen}]
Clearly, the second statement follows from the first, as $\Hom(\Gamma,G)$ could be naturally identified with a closed subset 
of $\Hom(\Gamma,\mathbf{G}(\mathbb{C}))$,
so we focus on proving the first statement.

By Lemma~\ref{l:mas} there exists a finite list $\mathbf{H}_1,\mathbf{H}_2,\ldots,\mathbf{H}_n$
of maximal proper algebraic subgroups of $\mathbf{G}$ representing every conjugacy class of such.
Using Chevalley's Theorem, we fix for every $i=1,\ldots, n$ an irreducible projective representation $\mathbf{G} \to \PGL(V_i)$ with $\dim(V_i)>1$
such that $\mathbf{H}_i$ fixes a point in $\mathbb{P}(V_i)$.
We let $U=\cup_i \mathbb{P}(V_i)$ and we note that 
$U$ is a compact $\mathbf{G}(\mathbb{C})$-space, thus the subset of $\rho\in \Hom(\Gamma,\mathbf{G}(\mathbb{C}))$
such that $\rho(\Gamma)$ has no fixed point in $U$ is open.
This subset contains the Zaiski dense representations, by the irreducibility of the representations $V_i$
and the condition $\dim(V_i)>1$,
while every non-Zaiski dense representation fixes a point in $U$, by the choice of the groups $\mathbf{H}_i$.
It follows that this subset coinsides with the subset of Zaiski dense representations, which in turn must be open.
\end{proof}

\subsection{AREA -- Algebraic Representations of Ergodic Actions} \label{sub:area}\hfill{}\\ 
In this subsection we define the notion of AREA, relate it to the relative metric ergodicity, and will
use these results in the proof of Theorem~\ref{T:boundary} in the following subsection.
The idea of AREA was developed in \cite{BF:AREA}, and used in \cite{BF:products} and \cite{BF:Margulis} for
the proofs of super-rigidity for lattice representations. 
Here we will focus on the case of algebraic groups and varieties defined over $\bbR$ rather than more general fields. 
Let $G$ be a real algebraic Lie group. We shall use the following facts that are mostly due to Zimmer:
\begin{enumerate}
	\item \label{i:locally-closed}
	Let $G\acts V$ be an algebraic action of an algebraic group on an algebraic variety.
	Then the orbits of the corresponding $G$-action on $\Prob(V)$ are locally closed,
	and as a consequence the space $\Prob(V)/G$ of $G$-orbits on $\Prob(V)$ is standard Borel
	(see  \cite{zimmer-book}*{Corollary 3.2.17}). 
	\item \label{i:stabilizers}
	For any $\mu\in\Prob(V)$ the stabilizer $\Stab_G(\mu)=\setdef{g\in G}{g_*\mu=\mu}$
	is a real algebraic subgroup of $G$ (see \cite{zimmer-book}*{Corollary 3.2.18}).
	\item \label{i:amen-stab}
	For the $G$-action on $V=G/P$ for any $\mu\in\Prob(G/P)$ the stabilizer $\Stab(\mu)$
	is a compact extension of solvable subgroup (see \cite{zimmer-book}*{Theorem 3.2.22}).
\end{enumerate}
We note that working over $\bbR$ is essential for (\ref{i:stabilizers}).
However, in other respects the scope can be significantly broadened to Polish fields --
see \cite{BDL}.

\medskip

Let us fix a countable group $\Gamma$ and a homomorphism 
\[
	\rho:\Gamma \overto{}  G
\] 
with Zariski dense image, taking values in a real semisimple Lie group $G$. 
We shall also consider measure class preserving ergodic actions $\Gamma\acts (A,\alpha)$.
In the proof of Theorem~\ref{T:boundary} the role of the Lebesgue $\Gamma$-space $A$ 
will be played by $B_-$, $B_+$, and $B$.

A $G$-\textbf{variety} is an algebraic variety over $\bbR$ with an algebraic action 
of $G$. Given a $G$-variety $V$, and a measure class preserving $\Gamma$-action on $(A,\alpha)$
a $\Gamma$-\textbf{map} is a measurable map $\phi:A\to  V$ satisfying
\[
	\phi(\gamma.x)=\rho(\gamma).\phi(x)\qquad (\gamma\in\Gamma)
\]
for $\alpha$-a.e $x\in A$.
We denote by $\Map_\Gamma(A,V)$ the set of all equivalence classes of such 
$\Gamma$-maps, where $\phi\sim\psi$ if $\phi(a)=\psi(a)$ for 
$\alpha$-a.e $a\in A$.

Varying the source and the target of $\Map_\Gamma(A,V)$ 
we obtain two types of maps:
\begin{itemize}
	\item 
	A $G$-equivariant algebraic map $\theta:W\to  V$ between $G$-varieties gives
	a map
	\[
		\Map_\Gamma(A,W)\overto{} \Map_\Gamma(A,V)
	\]
	by postcomposition $\phi\mapsto \theta\circ \phi$.
	\item 
	A $\Gamma$-equivariant map $p:(A',\alpha')\overto{}  (A,\alpha)$ between 
	measure class preserving $\Gamma$-actions defines a map
	\[
		\Map_\Gamma(A,V)\overto{} \Map_\Gamma(A',V)
	\]
	by precomposition $\psi\mapsto \psi\circ p$.
\end{itemize}

\begin{prop}[AREA - the initial object]\label{P:ininitial}\hfill{}\\
	Let $(A,\alpha)$ be an ergodic measure-class preserving $\Gamma$-space.
	Then there exists a homogeneous real algebraic $G$-variety $V_0=G/H_0$ and 
	$\phi_0\in \Map_\Gamma(A,V_0)$ so that
	for any $G$-variety $V$ and $\phi\in \Map_\Gamma(A,V)$
	there exists a unique $G$-equivariant algebraic map $\theta_{V,\phi}:V_0\to  V$
	so that $\theta_{V,\phi}\circ \phi_0=\phi$.
	The $G$-variety $V_0=G/H_0$ is unique up to $G$-automorphisms.
\end{prop}
\begin{proof}
  See \cite{BF:products}*{Theorem~4.3}.  
\end{proof}

The following proposition is special for the real field.

\begin{prop} \label{P:ininitial-amenable}
    Let $(A,\alpha)$ be an ergodic measure-class preserving $\Gamma$-space
    which is Zimmer amenable.
    Then the real algebraic subgroup $H_0<G$ given in Proposition~\ref{P:ininitial}
    is amenable.
\end{prop}

The following proof is a slight variation on the now classical argument of Zimmer (see \cite{zimmer-book}*{Section 3.2}).

\begin{proof}
We let $P<G$ be a minimal parabolic, thus $G/P$ is compact and $\Prob(G/P)$ is a compact convex $G$-space.
By the amenability of the $\Gamma$-action on $A$ there exists an element $\psi\in \Map_\Gamma(A,\Prob(G/P))$.
We now use the key property that the $G$-orbits on $\Prob(G/P)$ are locally closed.
	This fact (due to Zimmer \cite{zimmer-book}*{Corollary 3.2.17}) implies 
        that the space of $G$-orbits in $\Prob(G/P)$ 
	is a standard Borel space, and therefore the image of the ergodic space $A$ (equivalently, the measure $\psi_*(\alpha)$)
	is supported on a single $G$-orbit $G.\mu$ which is $G$-equivariantly homeomorphic to the coset space $G/G_\mu$, where $G_\mu=\setdef{g\in G}{g_*\mu=\mu}$. 
 Thus we can view $\psi$ as a measurable equivariant map into a $G$-homogeneous space $G/G_\mu$, 
	\[ \psi:A\overto{}  G/G_\mu.
	\]
By \cite{zimmer-book}*{Theorem 3.2.22}, $G_\mu$ 
is a real algebraic subgroup of $G$ that is a solvable by compact extension,
and in particular it is amenable.
By the defining property of the initial object for $A$, $\phi_0:A \to  G/H_0$,
there exists a $G$-map $G/H_0\to G/G_\mu$. 
Therefore, $H_0$ is contained in a conjugate of $G_\mu$.
It follows that $H_0$ is indeed amenable.
\end{proof}

\begin{lemma} \label{lem:iden}
    Let $(A,\alpha)$ be an ergodic measure-class preserving $\Gamma$-space and fix $\phi\in \Map_\Gamma(A,G/H)$ for some real algebraic subgroup $H<G$.
    Then the support of $\phi_*(\alpha)$ is Zariski dense in $G/H$. 
    If $\theta:G/H\to G/H$ is a $G$-equivariant map such that $\phi\circ \theta=\phi$ then $\theta$ is the identity map.
\end{lemma}

\begin{proof}
The Zariski closure of the support of $\phi_*(\alpha)$ is $\rho(\Gamma)$-invariant, hence $G$-invariant, as $\rho(\Gamma)$ is Zariski dense.
By the transitivity of the $G$-action on $G/H$ it must be everything. 
The map $\theta$ is the identity $\phi_*(\alpha)$-a.e, and its set of fixed points is Zariski closed, so it must be the identity map.
\end{proof}

\begin{prop} \label{prop:L=H_0}
    Let $(A,\alpha)$ be an ergodic measure-class preserving $\Gamma$-space.
    Let $H_0<G$ be the real algebraic subgroup and $\phi_0\in \Map_\Gamma(A,G/H_0)$ be the map given in Proposition~\ref{P:ininitial}.
    Assume there exists a real algebraic subgroup $L<H_0$ such that the coset space $H_0/L$ admits an $H$-invariant separable Borel metric $d$ and let $\theta:G/L\to G/H_0$ be the associated map.
    If there exist a relatively metrically ergodic map $p:(A',\alpha')\to  (A,\alpha)$ 
	and a map $F\in\Map_\Gamma(A', G/L)$ so that $\phi_0\circ p=\theta\circ F$
 then $L=H_0$.	
 \end{prop}

 \begin{proof}
     By Lemma~\ref{L:rme-toGH} there exists $\phi\in \Map_\Gamma(A,G/H_0)$ such that $\phi_0=\theta\circ \phi$.
     Then, by Proposition~\ref{P:ininitial}, there exists a $G$-equivariant map $\theta':G/H_0\to G/L$ such that $\phi=\theta'\circ \phi_0$.
     These are, correspondingly, the dashed and dotted maps appearing in 
     the following commutative diagram. 
     	\[
		\begin{tikzcd}
			A' \ar[rr, "F"] \ar[d, "p"] & & G/L \ar[d, "\theta"] \\
			A \ar[rr, "\phi_0"] \ar[urr, dashed, "\phi"] & & G/H_0 \ar[u, bend right=60, dotted, swap, "\theta'"]
		\end{tikzcd}
	\]
We use Lemma~\ref{lem:iden} to deduce that $\theta'\circ \theta$ is the identity map of $G/L$.
By the definition of $\theta$, $\theta(eL)=eH_0$, thus $\theta'(eH_0)=eL$ and we conclude the inclusion of the corresponding stabilizers, $H_0<L$.
Since $L<H_0$ we get that indeed $L=H_0$.
 \end{proof}

\medskip

\subsection{Proof of Theorem~\ref{T:boundary}} 
\label{sub:proof_of_theorem_ref_t_boundary}\hfill{}

This section is devoted to the proof of Theorem~\ref{T:boundary}.
We first make some preparation.
The following lemma concerns algebraic groups and it is valid over an arbitrary field, which could be assumed to be algebraically closed. 

\begin{lemma} \label{lem:alggrplem}
Let $G$ be a connected reductive algebraic group.
Fix algebraic subgroups $H<H_+,H_-<G$ and denote by $U_+$ and $U_-$ the unipotent radicals of $H_+$ and $H_-$ respectively.
Assume that $H_+=H\cdot U_+$, $H_-=H\cdot U_-$ and $U_+\cdot H\cdot U_-$ is Zariski dense in $G$.
Then $H_+$ and $H_-$ are parabolic subgroups of $G$.
\end{lemma}

\begin{remark}
    This is \cite[Lemma 3.5]{BF:icm}, which was given without a proof, modulo a slight formulation correction.
\end{remark}

\begin{proof}
By symmetricity, we argue to show that $H_+$ is parabolic.
We will show that it contains a Borel subgroup.
Upon replacing them by their identity components, we assume as we may that the groups $H$, $H_+$ and $H_-$ are connected.

Assume first that $H_+$ is solvable.
Let $B$ be a Borel subgroup containing $H_+$.
Denote $B=TV_+$, where $V_+$ is its unipotent radical and $T$ a maximal torus of $G$.
Let $V_-$ be an opposite unipotent subgroup, so that $V_-TV_+$ is Zariski dense in $G$.
Note that $V_+$ and $V_-$ are unipotent subgroups of $G$ of maximal dimension, thus
\[ \dim(G)\leq \dim(U_-)+\dim(H_+) \leq \dim(V_-)+\dim(B)=\dim(G) \]
and we deduce that $\dim(H_+)=\dim(B)$, thus $H_+=B$.
This finishes the proof for $H_+$ solvable.

Consider now the general case.
Consider a Zariski dense subset of $H$ of the form  $V'_-T'V'_+$, where $T'<H$ is a maximal torus and $V'_-,V'_+$ are unipotent subgroups of $H$ normalized by $T'$. To be clear, both $V'_-$ and $V'_+$ contain the unipotent radical of $H$ and they could be thought of as the preimages of unipotent groups associated with a similar decomposition in the quotient reductive group, with respect to the image of $T'$ in it.
We denote $H'=T'$, $H'_+=T'V'_+U_+$ and $H'_-=T'V'_-U_-$ and note that this is a triple of solvable subgroups of $G$ which satisfy the conditions of the theorem.
We conclude that $H'_+$ is a Borel subgroup.
This finishes the proof, as $H'_+<H_+$.
\end{proof}

\begin{lemma} \label{lem:bi-inv}
    If $H$ is a separable Lie group
    such that the identity connected component $H^0$ is compact modulo its center then $H$ admits a separable bi-invariant metric,
    and for every closed subgroup $L<H$, the coset space $H/L$ admits a separable $H$-invariant metric.
    
    If $H$ is a separable amenable Lie group and $L<H$ is a closed subgroup which contains the nilpotent radical of $H^0$ then $H/L$ admits a separable $H$-invariant metric.
\end{lemma}

\begin{proof}
Let $H$ be a separable Lie group
    such that the identity connected component $H^0$ is compact modulo its center.
    Observe that $H^0$ admits a bi-invariant Riemannian metric $d^0$, as its Lie algebra admits an Ad-invariant inner product, due to the compactness of the image of $H^0$ under the adjoint representation.
Then
\[ d(h_1,h_2)= \left\{
\begin{matrix}
\min\{h_1^{-1}h_2,1\} & h_1^{-1}h_2\in H^0 \\
1 & \mbox{otherwise}
\end{matrix}
\right.
\]
is an $H$-bi-invariant separable metric on $H$.
For a closed subgroup $L<H$, an $H$-invariant separable metric on $H/L$ is given by
\[ (h_1L,h_2L) \mapsto \inf \{d(h_1,h'_2) \mid h_1^{-1}h'_2\in L\}. \]
The last part of the lemma follows as $H^0$ modulo its nil radical is amenable and reductive, thus compact modulo its center.
\end{proof}

\begin{lemma} \label{lem:P-inv}
    Let $G$ be a semisimple Lie group and $P<G$ a minimal parabolic.
    Then the only $P$-invariant probability measure on $G/P$ is the Dirac measure $\delta_{eP}$.
\end{lemma}

\begin{proof}
    Let $\mu\in\Prob(G/P)$ be a $P$-invariant measure and let $U<P$ be the unipotent radical.
    Then $\mu$ is $U$-invariant. By \cite[Corollary 1.10]{BDL}, $\mu$ is supported on the $U$-fixed points in $G/P$. Since $eP$ is the unique $U$-fixed point, we deduce that indeed, $\mu=\delta_{eP}$.
\end{proof}

\begin{proof}[Proof of Theorem~\ref{T:boundary}]
We apply Proposition~\ref{P:ininitial} taking $A$ to be $B,B_-$ and $B_+$ and get correspondingly homogeneous real algebraic $G$-varieties $V,V_-$ and $V_+$ of $G$ and maps $\phi\in \Map_\Gamma(B,V)$, $\phi_-\in \Map_\Gamma(B_-,V_-)$ and $\phi_+\in \Map_\Gamma(B_+,V_+)$, forming initial objects in the corresponding categories of AREAs.
We let $H<G$ be a real algebraic group such that $V\simeq G/H$.
The defining property of $\phi$ gives rise to the dashed $G$-maps in the following diagrams: 
\[
	\begin{tikzcd}
		B \ar[r,"\phi"] \ar[d, "\pi_-"] & G/H \ar[d, dashed] & & B \ar[r,"\phi"] \ar[d, "\pi_+"] & G/H \ar[d, dashed] \\
		B_- \ar[r,"\phi_-"] & V_- & &  B_+ \ar[r,"\phi_+"] & V_+ 
	\end{tikzcd}
\]
and we let $H_-,H_+$ be the stabilizer in $G$ of the corresponding images of the base coset $eH$, thus $H<H_-,H_+$ 
In what follows we keep the above identifications and consider the codomains of the maps $\phi,\phi_-$ and $\phi_+$ to be $G/H,G/H_-$ and $G/H_+$ correspondingly.

By the definition of a boundary pair, the map $\pi_-\times \pi_+:B\to B_-\times B_+$ is measure class preserving.
We consider the diagram
\begin{equation} \label{eq:bpm}
	\begin{tikzcd}
		B \ar[rr,"\phi"] \ar[d, "\pi_-\times \pi_+"] & & G/H \ar[d, dashed, "\theta=\theta_-\times \theta_+"]  \\
		B_-\times B_+ \ar[rr,"\phi_-\times \phi_+"] & & G/H_-\times G/H_+
	\end{tikzcd}
\end{equation}
where the dashed map $\theta=\theta_-\times \theta_+$ is the one given by Proposition~\ref{P:ininitial}.
We note that the map $\phi_-\times \phi_+$ is in fact $\Gamma\times \Gamma$-equivariant, with respect to the homomorphism $\rho\times \rho:\Gamma\times \Gamma \to G\times G$, which has a Zariski dense image, as $\rho$ does by assumption.
Applying Lemma~\ref{lem:iden} to this map, we conclude that the essential image of $B_-\times B_+$ in $G/H_-\times G/H_+$ is Zariski dense.
Chasing the diagram, we get that the image of $\theta$ is Zariski dense in $G/H_-\times G/H_+$.
Taking its preimage under the map $G\times G\to G/H_-\times G/H_+$, we conclude that the subset
\[ \{(gh_-,gh_+) \mid g\in G,~h_-\in H_-,~h_+\in H_+\} \subset G\times G\]
is Zariski dense in $G\times G$. 
Taking its image under the map
\[ G\times G \to G,\quad (g,h) \mapsto h^{-1}g, \]
we conclude that the subset 
\[ H_+\cdot H_- = \{(h_+^{-1}h_-) \mid h_-\in H_-,~h_+\in H_+\} \subset G \]
is Zariski dense in $G$.

We denote by $U_-$ and $U_+$ the unipotent radicals of (the identity connected components of) $H_-$ and $H_+$ correspondingly and consider the subgroup $H\cdot U_-<H_-$ and $H\cdot U_+<H_+$.
These are closed subgroup.
To see this, we use the standard fact that the image of a real algebraic group under a real algebraic morphism is closed, and apply it to the images of $H$ in $H_\pm/U_\pm$.
We claim that $H\cdot U_-=H_-$ and $H\cdot U_+=H_+$.
Since the spaces $B_\pm$ are Zimmer amenable, we get by Proposition~\ref{P:ininitial-amenable} that the groups
$H_\pm$ are amenable and we conclude by Lemma~\ref{lem:bi-inv} that the coset spaces $H_\pm/H\cdot U_\pm$ admit separable $H_\pm$-invariant metrics.
by Proposition~\ref{prop:L=H_0}, taken with $A'=B$, $A=B_\pm$ and $F$ being the composition $B\to G/H\to G/H\cdot U_\pm$, we conclude that $H\cdot U_\pm=H_\pm$, proving the claim.

It follows that $H_+\cdot H_-=U_+\cdot H\cdot U_-$
and we conclude by Lemma~\ref{lem:alggrplem} that $H_-$ and $H_+$ are parabolic subgroups of $G$.
By the amenability of $H_\pm$, these parabolic subgroups must be minimal in $G$, thus conjugated to $P$ and we get $G$-isomorphisms $G/H_\pm \to G/P$.
Composing the later maps with $\phi_\pm$, we obtain the required maps $f_\pm:B_\pm\to G/P$.
It follows also that these maps satisfy the condition of Proposition~\ref{P:ininitial}.

We now show that 
\[ \Map_\Gamma(B_\pm,\Prob(G/P))=\{\delta\circ f_\pm\}. \]
We fix $\psi\in\Map_\Gamma(B_\pm,\Prob(G/P))$ and argue as in the proof of Proposition~\ref{P:ininitial-amenable} to show that $\psi$ takes values in a single $G$-orbit $G/G_\mu$, where $G_\mu$ is the stabilizer of a measure $\mu$, which is a real algebraic subgroup of $G$.
By Proposition~\ref{P:ininitial}, we get a $G$-map $G/P\to G/G_\mu$.
Up to conjugation, we assume as we may that $eG_\mu$ is the image of $eP$ under this map, thus $\mu$ is $P$-fixed.
By Lemma~\ref{lem:P-inv} we get that $\mu=\delta_{eP}$, and we conclude that indeed, $\psi=\delta\circ f_\pm$.

Since $H_+\cdot H_-$ is Zariski dense in $G$, we conclude that these are opposite parabolic, thus $H_+\cap H_-$ is conjugated to $Z$.
Fix $g\in G$ such that $H_+\cap H_-=Z^g$.
In particular, we get $H<Z^g$.
Since 
    
\[ 
    \begin{split}
        \dim(G)&\leq \dim(U_-)+\dim(H)+ \dim(U_+)\\
        &\leq \dim(U_-)+\dim(Z)+ \dim(U_+) =\dim(G),
    \end{split}
\]
we conclude that $\dim(H)=\dim(Z)=\dim(Z^g)$.
Thus $H=Z^g$, as $Z^g$ is connected.
Therefore the diagram \eqref{eq:bpm} extends as follow:
\[
    \begin{tikzcd}
		B \ar[rr,"\phi"] \ar[rrr, dashed, bend left=20, swap, "h'"] 
        \ar[rrrr, dotted, bend left=20, "h"] \ar[d, "\pi_-\times \pi_+"] & & G/H \ar[d, "\theta_-\times \theta_+"]  \ar[r, "\sim"] & G/Z \ar[d, dashed, "\theta'_-\times \theta'_+"] & G/Z \ar[l, dotted, swap, "w"] \ar[ld, bend left, "\pr_-\times\pr_+"]\\
		B_-\times B_+ \ar[rr,"\phi_-\times \phi_+"] \ar[rrr, bend right=20,"f_-\times f_+"] & & G/H_-\times G/H_+ \ar[r, "\sim"] & G/P\times G/P &
    \end{tikzcd}
\]
where the dashed maps $h'$ and $\theta'_-\times \theta'_+$ are defined by the diagram's commutativity and the dotted arrows are yet to be determined.
We recall the fact that the Weyl group $W$ of $G$, which acts on $G/Z$ by $G$-automorphisms, acts on the set of $G$-maps $G/Z\to G/P$ transitively by precomposition.
We use it to find $w\in W$ such that $\theta'_+\circ w=\pr_+$.
We define $h=w^{-1}\circ h'$
and get that
\[ 
    f_+\circ \pi_+=\theta'_+\circ h' =(\theta'_+\circ w) \circ (w^{-1}\circ h')
    =  \pr_+\circ h. 
\]
Since the image of $\theta=\theta_1\times \theta_2$ 
is Zariski dense in $G/H_-\times G/H_+$,
we get that the image of $\theta'_-\times \theta'_+$ 
is Zariski dense in $G/P\times G/P$,
thus also the image of $(\theta'_-\times \theta'_+)\circ w$ 
is Zariski dense in $G/P\times G/P$.
Since $\theta'_+\circ w=\pr_+$, we conclude that $\theta'_- \circ w =\pr_-$,
thus
\[ f_-\circ \pi_-=\theta'_-\circ h' =(\theta'_-\circ w) \circ (w^{-1}\circ h')= \pr_- \circ h. \]
This completes the proof.
\end{proof}

In the course of the proof of Theorem~\ref{T:boundary} above, we in fact have shown the following.

\begin{theorem}
In the setting of Theorem~\ref{T:boundary},
the maps $f_+:B_+\to G/P$, $f_-:B_-\to G/P$ and $h:B\to G/Z$ are the initial objects, as described in Proposition~\ref{P:ininitial},
when the $\Gamma$-space $A$ is taken to be $B_+$, $B_-$
and $B$ correspondingly.    
\end{theorem}

\bigskip

\subsection{Applications of Boundary Theory to Apafic Gregs}\hfill{}

In this subsection we fix an Apafic Greg $(X,\calX,m,T,w,\Gamma)$.
The following is an immediate corollary of Theorem~\ref{T:boundary},
in view of Theorem~\ref{T:BSfromGreg} and Lemma~\ref{L:equi-inv}.

\begin{cor}[{Theorem~\ref{T:boundary} for Apafic Gregs}]\label{C:boundary} \hfill{}\\
Let $G$ be a connected semisimple real Lie group with finite center,
let $\rho:\Gamma\to G$ be a continuous homomorphism with a Zariski dense image
and set $F=\rho\circ w:X\to G$.
Let $P<G$ be a minimal parabolic
and let $Z<P$ be the centralizer of a maximal split torus in $P$.
Then there exist 
measurable $F$-equivariant maps
\[
    \theta:X\ \overto{} \ G/Z\quad \mbox{and}
    \qquad  \phi_\pm:\ X\overto{} X_\pm\overto{} \ G/P
\]
such that $\phi_\pm\circ p_\pm=\pr_\pm \circ\theta$, where $\pr_\pm:G/Z\to  G/P$ are the maps described in Theorem~\ref{T:boundary}.
Here equivariance means a.e. equalities: $\theta(Tx)=F(x).\theta(x)$,
$\phi_\pm(Tx)=F(x).\phi_\pm(x)$.
\end{cor}

By the Zariski density of $\rho$, the space $G/P$ admits no $\Gamma$-invariant finite subset,
thus stationary measures associated with $\phi_-$ will have no atomic parts, 
by Proposition~\ref{p:non-atomic}.
In fact, much more is true.

\begin{defn} \label{def:generic}
Let $V$ be a real algebraic variety. 
A measure on $V$ is said to be {\bf generic} if every proper subvariety of $V$ is a null subset.
For a Lebesgue space $\Omega$, a map $f:\Omega\to \Prob(V)$ is said to be generic if for a.e $z\in \Omega$,
$f(z)$ is a generic measure.
\end{defn}

\begin{theorem} \label{t:generic}
In the setting of Corollary~\ref{C:boundary}, the past oriented stationary map $\psi\in \Stat_-(X_+,\Prob(G/P))$,
associated with $\phi_-$ as in Example~\ref{ex:sm}, is generic.
\end{theorem}

\begin{remark} \label{rem:generic}
Only the following weak form of Theorem~\ref{t:generic} will be used later in this paper:
for a.e $y\in X_+$, for every $g\in G$, 
the variety $gL\subset G/P$ is $\psi(y)$-null.
Here $L\subset G/P$ is the subvariety defined in Equation~\eqref{eq:defL}
before Lemma~\ref{L:UV}.
In fact, this statement follows already from Proposition~\ref{p:non-atomic},
applied to the stationary maps $\pi_\alpha \circ \psi:X_+\to \Prob(G/Q_\alpha)$, $\alpha\in \Pi$.
See Remark~\ref{rem:miror} in the proof of Theorem~\ref{T:main}.
\end{remark}

In view of Remark~\ref{rem:generic}, we will allow ourselves to be a bit sketchy 
in the following proof of Theorem~\ref{t:generic}.

\begin{proof}
The proof is a variation of the proofs of Proposition~\ref{p:non-atomic}
and Lemma~\ref{l:atomic}.
We assume $\psi$ is not generic and argue to show a contradiction.
As in the proof of Proposition~\ref{p:non-atomic},
we find maps corresponding to $\psi$,
\[ 
    \begin{split}
        \phi&\in \ext(\Map_w(X_-,\Prob(G/P))),\\
        f&\in \ext(\Map_\Gamma(E_-,\Prob(G/P))).
    \end{split} 
\]
By Theorem~\ref{T:boundary}, $f$ takes values in Dirac measures.
We deduce that this is also the case for $\phi$ and for 
$f\circ\pi_+$.
In particular, it follows that $f\circ\pi_+$ is extremal in 
the convex compact set $\Map_\Gamma(E,\Prob(G/P))$,
thus $\psi$ is extremal in $\Stat(X_-,\Prob(G/P))$.

We can $G$-equivariantly embed $G/P$ in a projective space, 
and get a well defined $G$-invariant notion of degree on subvarieties of $G/P$.
We find minimal $d$ and $k$ such that 
the subset of $y\in X_+$ for which there exists a subvariety of $G/P$
of dimension $d$ and degree $k$
which is $\psi(y)$-non-null is $m_+$-non-null.
We denote by $H(d,k)$ the Hilbert scheme parametrizing such sub-schemes of $G/P$.
This is a $G$-projective variety.
We regard its elements as actual subsets of $G/P$, considering their real points.
We emphasise that the assumption that $\psi$ is not generic gives $d<\dim(G/P)$.

As in \S\ref{subsection:measures}, we let $Q$ be the compact convex space 
of positive measures of total mass in $[0,1]$ on $G/P$
and we define the map
\[ 
    \ev_{d,k}:Q\times H(d,k)\overto{} [0,1], \quad (\nu,V)\mapsto \nu(V). 
\]
By an obvious variant of the proof of Lemma~\ref{l:ev},
this function is upper semi-continuous
and as in the proof of Lemma~\ref{l:maxev},
we get that the map
\[ \maxev_{d,k}:Q\to [0,1], \quad \maxev_{d,k}(\nu)=\max\setdef{\nu(V)}{V\in H(d,k)} \]
is convex and upper semi-continuous.
We conclude by Lemma~\ref{lem:convex} that $\maxev_{d,k}\circ \psi$ is essentially constant
and we denote its essential value by $t$.

We consider the space $Q_t$ consisting of measures on $G/P$ for which every 
subvariety of dimension less then $d$ 
or degree less then $k$ is null and 
the measures of subvarieties of dimension $d$ and degree $k$ is bounded by $t$.
By a variant of the proof of Lemma~\ref{l:Qalpha}, $Q_t$ is a convex $G_\delta$-subset of $Q$.
An obvious variant of Lemma~\ref{l:maxpart} enable us to
argue as in the proof of Lemma~\ref{l:atomic}.
We decompose every measure $\nu\in Q_t$ as a sum $\nu=\nu_0+\nu_1$
such that $\nu_0$ is supported on subvarieties in $H(d,k)$ 
and the measure of such subvarieties is either 0 or $t$
while the $\nu_1$-measure of such subvarieties is less than $t$.
Accordingly,
we can write $\psi=s\psi_0+(1-s)\psi_1$
where $\psi_i$ are stationary, $s\in (0,1]$ and 
for a.e $y\in X_+$, 
$\psi_0(y)$ is supported on subvarieties in $H(d,k)$ 
and the measure of such subvarieties is either $0$ or $t$.
By the extremality of $\psi$, we get $\psi=\psi_0$,
thus for a.e $y\in X_+$, 
$\psi(y)$ is supported on subvarieties in $H(d,k)$ 
and the measure of such subvarieties is either $0$ or $t$.
However, we note that, unlike in the proof of Proposition~\ref{p:non-atomic},
we cannot conclude that the map $\psi$ is deterministic.

Note that for a.e $y\in X_+$, 
$\psi(y)$ is supported on exactly $n=t^{-1}$ elements of $H(d,k)$.
Assigning their union to $y$ determines a map $\psi':X_+\to H(d,kn)$
which is easily seen to be deterministic, that is $w$-invariant.
As before, 
we regard elements of $H(d,kn)$ as actual subsets of $G/P$, considering their real points.
Recalling that the values of the map $\phi$ are Dirac measures, we regard them as points of $G/P$.
Then the stationary equation
	\[
		\psi(y)=\int_X \phi\circ p_-(x)\dd \mu_{y}(x)\qquad (y\in X_+).
	\]
guarantees that for a.e $x\in X$ we have $\phi\circ p_-(x)\in \psi'\circ p_+(x)$.
As $\psi'$ is $w$-invariant, it corresponds by Lemma~\ref{L:equi-inv} to a
$\Gamma$-map $f'_+:E_+\to H_n(d,kn)$
and, denoting $f_-=f\circ \pi_-$, we deduce that 
for a.e $e\in E$, $f_-(e)\in f'_+(e)$.
We now use the Apafi condition and get by a Fubini argument that 
for a.e $e_+\in E_+$, for a.e $e_-\in E_-$, 
$f_-(e_-)\in f'_+(e_+)$.

We now fix a generic $e_+\in E_+$ and set $V=f'_+(e_+)$.
We choose a countable dense subgroup $\Gamma_0<\Gamma$
and find a $\Gamma_0$-invariant full measure subset of $A\subset E_-$ such that $f_-(A)\subset V$. 
We let $V_0$ be the Zariski closure of $f_-(A)$ and observe that it is a $G$-invariant subvariety of $V$,
as $\rho(\Gamma_0)$ is Zariski dense in $G$.
By $G$-invariance we get $V_0=G/P$,
while we have $\dim(V_0)\leq \dim(V)=d<\dim(G/P)$.
We thus get the desired contradiction
and conclude that $\psi$ is indeed generic.
\end{proof}

\begin{remark}
the argument given in the proof of Theorem~\ref{t:generic}
could be also used to give an alternative proof of Proposition~\ref{p:non-atomic}.
\end{remark}

\section{Proofs of the main results}\label{sec:proofs}\hfill{}\\
The goal of this section is to prove the main Theorem~\ref{T:main}.
The key statement is the simplicity of the Lyapunov spectrum.
The logic of the argument is as follows:
\begin{itemize}
	\item 
	We use the past oriented stationary measure $\psi\in \Stat_-(X_+,\Prob(G/P))$ 
	and its weak* convergence to $\delta_{\phi_-(x)}\in \Prob(G/P)$ 
	under appropriate ergodic products of $F_{-n}^{-1}$ 
        as in Theorem~\ref{T:structure-of-stat}. 
	\item
	We use the fact that $\delta_{\phi_-(x)}$ is a \emph{Dirac measure}
        (Theorem~\ref{T:boundary}) 
	while $\psi_{x}$ are \emph{generic measures} (Theorem~\ref{t:generic}). 
	This suggests that the ergodic products $F_{-n}^{-1}(x)$ are \emph{contracting}.
	\item
	To make this contraction property explicit, we use the  map 
        \[\theta:X\overto{}  G/Z\] 
	to conjugate $F$ to an essentially diagonal cocycle $D:X\overto{}  Z$ 
	while keeping the property of contraction of certain \emph{proper measures} 
        (not stationary any longer) towards \emph{Dirac measures}.
	\item 
	Finally, this qualitative, but persistent, contraction property is translated 
        into a quantitative one using the well known Lemma~\ref{L:Kesten}. 
	To establish the required integrability condition we need to pass to an 
        induced system $X'\subset X$ and relate the Lyapunov spectrum there to the original one. 
\end{itemize}
%

\subsection{Some preliminary Lemmas} 
\label{sub:some_preliminary_lemmas}
\begin{lemma}\label{L:Kesten}\hfill{}\\
	Let $(X,m,T)$ be an ergodic p.m.p transformation, $h\in L^1(X,m)$ and $h_n:X\to\bbR$, $n\in\bbZ$ 
	the corresponding $\bbZ$-cocycle:
	\[
		h_n(x)=h(x)+h(Tx)+\dots+h(T^{n-1}x)\qquad (n\in\bbN)
	\]
	and $h_0(x)=0$, $h_{n}(x)=-h_{-n}(T^nx)$ for $n<0$. Then the following conditions are equivalent:
	\begin{enumerate}
		\item For $m$-a.e $x\in X$ one has $\ h_n(x)\to+\infty\ $ as $n\to\infty$,
		\item For $m$-a.e $x\in X$ one has $\ -h_{-n}(x)\to+\infty\ $ as $n\to\infty$,
		\item $\lambda:=\int h\dd m$ is positive: $\lambda>0$ and therefore for $m$-a.e $x\in X$
		\[
			h_n(x)=\lambda\cdot n+o_x(n).
		\]
	\end{enumerate}
\end{lemma}
The proof of this lemma can be found for example in \cite{BQ:book}*{Lemma 3.18}. 

\bigskip


Recall the notion of an \textbf{induced system} in the sense of Kakutani.
Let $\textbf{X}=(X,m,T)$ be an ergodic p.m.p system and 
$X'\subset X$ a subset with $m(X')>0$.
Denote by $m'$ the normalized restriction: 
$m'(E)=m(E\cap X')/m(X')$, and by $n:X\to  \bbN$
the \textbf{first return time map}
\[
	n(x):=\min \setdef{ n \in\bbN }{ T^nx\in X'}.
\]
Define a transformation $T'$ of $(X',m')$ by $T'x=T^{n(x)}x$. 
It preserves the measure $m'$
and is ergodic; the system $\textbf{X}'=(X',m',T')$ is called the \textbf{induced} 
from $\textbf{X}=(X,m,T)$.

Let $n_k:X'\to \bbN$, $k\in\bbN$, be the cocycle generated by 
the return time function $n:X'\to \bbN$
\begin{equation} \label{eq:return}
    n_k(x)=n(x)+n(T'x)+\dots+n(T'^{k-1}x)\qquad (k\in\bbN,\ x\in X').
\end{equation}
	
Let $G$ be a simple real Lie group and $F:X\to G$ a measurable map and $F':X'\to G$ the map defined by
\[
	F'(x):=F_{n(x)}(x)=F(T^{n(x)-1}x)\cdots F(x).
\]

\begin{lemma}\label{L:induced-Lya}\hfill{}\\
	Let $X'\subset X$ be positive measure subset.
	If $F:X\to G$ is integrable, then the induced map $F':X'\to G$ 
	is also integrable, and the Lyapunov spectra $\Lambda_F,\Lambda_{F'}\in\mathfrak{a}^+$ 
	are positively proportional, more precisely:
	\[
		 \Lambda_{F'}=\frac{1}{m(X')}\cdot\Lambda_{F}.
	\]
\end{lemma}
\begin{proof}
	Denote $f(x)=|F(x)|$. It is a non-negative function in $L^1(X,m)$.
	By subadditivity, $|F_k(x))|\le |f_k(x)|$, 
        where $f_k(x)=\sum_{j=0}^{k-1}f(T^j x)$.  
	Denote $X'_k:=\setdef{x\in X'}{n(x)=k}$. Then 
	$X=\bigsqcup_{k=1}^\infty\bigsqcup_{j=0}^{k-1} T^j X'_k$ and therefore
	\[
		\int_{X'}|F'|\dd m'=\sum_{k=1}^\infty\int_{X'_k} |F_k|\dd m'
		\le \sum_{k=1}^\infty\int_{X'_k} f_k\dd m'=\int_X f\dd m<\infty.
	\]
	Thus $F'$ is integrable, and $\Lambda_{F'}$ is well defined. 
	Since 
	\[
		\int_{X'} n(x)\dd m'(x)=\frac{1}{m(X')},
	\] 
	by Birkhoff's ergodic theorem for a.e $x\in X'$ 
	\[
		F'_k(x)=F_{n_k(x)}\qquad\textrm{where}\qquad
		n_k(x)=m(X')^{-1}\cdot k+o_x(k).
	\]
	It follows that $\Lambda_{F'}=m(X')^{-1}\cdot\Lambda_{F}$.
\end{proof}
\medskip

\begin{lemma}\label{L:conjugation}\hfill{}\\
	Let $(X,m,T)$ be an ergodic system, $F:X\to G$ an integrable map into a simple real Lie group $G$,
	and $s:X\to G$ another integrable map. Define the conjugate map $D:X\to G$ by 
	\[
		D(x)=s(Tx) F(x) s(x)^{-1}.
	\] 
	Then $D$ is also integrable, and has the same Lyapunov spectrum: $\Lambda_{D}=\Lambda_F$.
\end{lemma}
\begin{proof}
	Integrability is straightforward since
	\[
		\int_X |D|\dd m\le \int_X |s(Tx)| + |F|+|s(x)^{-1}|\dd m <+\infty
	\]
	and $|s(x)^{-1}|\le C\cdot |s(x)|$ for some fixed constant $C=C(G)$. 
	The Lyapunov spectra $\Lambda_{F}$ and $\Lambda_D$ describe the asymptotic behavior 
	of $F_n(x)$ and 
	\[
		D_n(x)=s(T^nx) F_n(x) s(x)^{-1}
	\] 
	which are at bounded distance from each other,
	at least along a subsequence $n_i(x)$ where $T^{n_i(x)}x$ visits a set where $|s(x)|<M$ for some $M$.
	Thus $\Lambda_{D}=\Lambda_F$ follows.
\end{proof}
\medskip


%

\subsection{Proof of Theorem \ref{T:main}.(1) -- simplicity of the spectrum.} 
\label{sub:proof_of_theorem_ref_t_main}\hfill{}\\
Let $(X,\calX,m,T,w,\Gamma)$ be an Apafic Greg,
let $G$ be a connected semisimple real Lie group with finite center
and let $\rho:\Gamma\to G$ be an integrable Zariski dense representation.
We argue to show that the Lyapunov spectrum $\Lambda_F\in\mathfrak{a}^+$
of $F=\rho\circ w:X\to G$ is simple, i.e. $\Lambda_F\in\mathfrak{a}^{++}$.

Let $P<G$ be a minimal parabolic
and let $Z<P$ be the centralizer of a maximal split torus $A<P$.
By Corollary~\ref{C:boundary} there exist
measurable $F$-equivariant maps
\[
    \theta:X\ \overto{} \ G/Z\quad \mbox{and}\quad  
    \phi_-:\ X\overto{} X_-\overto{} \ G/P
\]
such that $\pi_-\circ \phi_-=\pr\circ\theta$, where $\pr:G/Z\to  G/P$ is the standard quotient map.
Composing with the map  $\delta:G/P\to \Prob(G/P)$,
we consider $\phi_-$ as a map in $\Map_w(X_-,\Prob(G/P))$
and we let $\psi\in \Stat_-(X_+,\Prob(G/P))$ be the corresponding 
past oriented stationary map given by Theorem~\ref{T:equi-inv}.
Then, by Theorem~\ref{t:generic}, 
$\psi$ is generic as defined in Definition~\ref{def:generic}.
In particular, it satisfies the following property regarding the variety 
$L\subset G/P$ defined in Equation~\eqref{eq:defL}
before Lemma~\ref{L:UV}.
\begin{equation} \label{eq:proper}
\mbox{For a.e $y\in X_+$, for every $g\in G$, 
the variety $gL\subset G/P$ is $\psi(y)$-null.}
\end{equation}

\begin{remark} \label{rem:miror}
As already mentioned in Remark~\ref{rem:generic},
\eqref{eq:proper} follows already from Proposition~\ref{p:non-atomic},
applied to the stationary maps 
\[\pi_\alpha \circ \psi:X_+\overto{} \Prob(G/Q_\alpha),\qquad(\alpha\in \Pi).\]
\end{remark}

Denote $\phi=\phi_-\circ p_-:X\overto{}X_-\overto{} G/P$ 
and let $F_n:X\to  G$, $n\in\mathbb{Z}$, be the cocycle generated by
$F=\rho\circ w$ as in \eqref{e:f-cocycle}.
Note that by Theorem~\ref{T:structure-of-stat}, 
for $m$-a.e $x\in X$ there is weak* convergence to the Dirac measure at $\phi(x)\in G/P$:
\begin{equation}\label{e:Fn-to-phi}
    \delta_{\phi(x)}=\lim_{n\to\infty}\ F_{-n}(x)^{-1}_*\psi\circ p_+(T^{-n}x)
\end{equation}
Using the $F$-equivariant map $\theta$,
we get by the Cocycle Reduction Lemma~\ref{l:reduction} 
that there exists a measurable map $s:X\to G$ such that the image of the cocycle 
\[
    D_n(x):=s(T^nx)\cdot F_n(x)\cdot s(x)^{-1}
\]
is essentially in $Z$.
In particular, we set 
\[ 
    D:X\overto{} Z,\qquad D(x)=D_1(x)=s(Tx)\cdot F(x)\cdot s(x)^{-1}. 
\]
Composing with the continuous homomorphism $q:Z\to \mathfrak{a}$ defined in \eqref{eq:q},
we also set 
\[ 
    h=q\circ D:X \overto{} \mathfrak{a}, \qquad h_n=q\circ D_n:X\overto{} \mathfrak{a}. 
\]
Consider the map
\[ 
    \eta:X\overto{} \Prob(G/P), \qquad \eta(x)=s(x)_*\psi\circ p_+(x). 
\]
Note that this map, unlike $\psi$, need not be stationary and need not 
factor via $X_+$. However, it is still generic; 
in particular, it still satisfies the analogue of \eqref{eq:proper}
and has the following property
\begin{equation} \label{eq:pproper}
\mbox{For a.e $x\in X$, 
the variety $L\subset G/P$ is $\eta(x)$-null.}
\end{equation}
Moreover, a version of \eqref{e:Fn-to-phi} still holds:
for $m$-a.e $x\in X$ we have the weak* convergence to a Dirac measure
\begin{equation}\label{e:Dcontracts}
    \begin{split}
			D_{-n}(x)^{-1}_*\eta(T^{-n}x)&=s(x)F_{-n}(x)^{-1}s(T^{-n}x)^{-1}_*\eta(T^{-n}x)\\
			&=s(x)_*\left(F_{-n}(x)^{-1}_*\psi\circ p_+(T^{-n}x)\right)\\
                &\overto{} \ \delta_{s(x).\phi(x)}.
    \end{split}
\end{equation}
This can be viewed as a \textit{contraction of the generic measure} $\eta(T^{-n}x)$
towards the Dirac measure $\delta_{s(x).\phi(x)}$ by the maps $D_{-n}(x)^{-1}$ that lie in $Z$.
We shall translate this  qualitative phenomenon into a quantitative one. 
However, We need some uniform control on the measures $\eta(x)\in \Prob(G/P)$ 
and the elements $s(x)\in G$.
We will achieve this by passing to an induced system as described below. 

Representing the open subset $G/P\setminus L\subset G/P$ as a countable union 
of ascending compact subsets $C_i$ and $G$ as the ascending union of $\{|g|\le i\}$, 
we get from \eqref{eq:pproper} the existence of a compact subset 
$C\subset G/P\setminus L$ and $i_0$ such that the subset
\begin{equation} \label{eq:X'}
    X':=\setdef{x\in X}{ \eta(x)(C)>\half,\ |s(x)|\le i_0 } 
\end{equation}
has positive $m$-measure.
Let $n:X'\overto{} \bbN$ be the return time function
and $(X',m',T')$ be the induced system, 
as in \S\ref{sub:some_preliminary_lemmas} preceding Lemma~\ref{L:induced-Lya}.
We consider the maps 
\[
    \begin{split}
			&F':X'\overto{}  G,\qquad F'(x):=F_{n(x)}(x),\\
			&D':X'\overto{}  Z,\qquad D'(x):=D_{n(x)}(x),\\
			&h':X'\overto{}  \mathfrak{a},\qquad\ \ h'(x):=h_{n(x)}(x)=q\circ D'(x).
    \end{split}
\]
Note that $D'$ is conjugate to $F'$ by $s:X'\to G$, indeed for $x\in X'$
\[
    D'(x)=D_{n(x)}(x)=s(T^{n(x)}x) F_{n(x)}(x)s(x)^{-1}=s(T'x)F'(x)s(x)^{-1}.
\]
Since $|s(-)|$ is bounded on $X'$ by construction, it is integrable on $(X',m')$.
Hence $D'$ is integrable on $(X',m')$ and $\Lambda_{D'}=\Lambda_{F'}$ by
Lemma~\ref{L:conjugation}.
Since Lemma~\ref{L:induced-Lya} gives $\Lambda_{F'}=m(X')^{-1}\cdot \Lambda_F$, 
we deduce positive proportionality
\[
    \Lambda_{F}=m(X')\cdot\Lambda_{D'}.
\]
Thus it suffices to show that $\Lambda_{D'}\in\mathfrak{a}^{++}$.
Note that Equation (\ref{e:Dcontracts}) implies that for $m'$-a.e $x\in X'$ we have the following 
weak* convergence along the subsequence defined by the return times to $X'$:
\begin{equation} \label{eq:D'conv}
    \lim_{n\to\infty} D'_{-n}(x)^{-1}_*\eta(T'^{-n}x)
			\ \overto{} \ \delta_{eP}.
\end{equation}
The integrability of $D'$ implies integrability of $h':X'\to\mathfrak{a}$ and  
denote by $\Lambda'\in\mathfrak{a}$ the barycenter of $h'$.
That is 
\begin{equation} \label{eq:a++}
		\alpha(\Lambda')=\int_{X'} \alpha(h'(x))\dd m'(x)
  \qquad(\alpha\in \Pi).
\end{equation}
\begin{claim}
    $\Lambda'\in \mathfrak{a}^{++}$.
\end{claim}
The claim is equivalent to the statement that $\alpha(\Lambda')>0$ for every $\alpha\in \Pi$.
We apply Lemma~\ref{L:UV}
to the subset $C\subset G/P$ chosen in the definition of $X'$ in Equation~\eqref{eq:X'}
and find a sequence of open neighborhoods $eP \in V_k \subset G/P$ 
such that 
for every $z\in Z$,
	\begin{equation} \label{eq:3.4}
		(z.C)\cap V_k\ne \emptyset
		\qquad\Longrightarrow\qquad
		\alpha\circ q(z)>k.
	\end{equation}
For every $k$, we let $f_k:G/P\to [0,1]$
be a continuous function with 
\[
    f_k(eP)=1,\qquad \supp(f_k)\subset V_k. 
\]
Fix $k\in\bbN$.
Applying \eqref{eq:D'conv} to $f_k$ we get, for $m'$-a.e $x\in X'$,
\[
    \lim_{n\to\infty} \int_{G/P} f_k(D'_{-n}(x)^{-1}.\xi)\dd\eta(T'^{-n}x)(\xi)=f_k(eP)=1.
\] 
Thus, for $m'$-a.e $x\in X'$ there exists $n(k,x)$ such that 
for every $n\ge n(k,x)$ we have
\[
	\int_{G/P} f_k(D'_{-n}(x)^{-1}.\xi)\dd\eta(T'^{-n}x)(\xi)>\frac{1}{2}.
\]
However, by the definition of $X'$ in Equation~\eqref{eq:X'}, for a.e $x\in X'$, 
the $\eta(x)$-measure of $C$ is greater than $1/2$, and
therefore, since $f_k\leq 1$, we get for every $n$,
\[
    \int_{G/P\setminus C} f_k(D'_{-n}(x)^{-1}.\xi)\dd\eta(T'^{-n}x)(\xi) 
    \leq \eta(x)(G/P\setminus C) <\frac{1}{2}.
\]
We conclude for $n\ge n(k,x)$ we must have
\[
    \int_{C} f_k(D'_{-n}(x)^{-1}.\xi)\dd\eta(T'^{-n}x)(\xi)  > 0.
\]
As $f_k$ is supported in $V_k$ it follows that $D'_{-n}(x)^{-1}C\cap V_k\ne \emptyset$
for all $n\ge n(k,x)$.
Thus, by Equation~\eqref{eq:3.4}, we get
\[ 
    -\alpha\circ h'_{-n}(x)=\alpha\circ q(D'_{-n}(x)^{-1})>k. 
\]
As $k\in\bbN$ was arbitrary, we conclude that for $m'$-a.e $x\in X'$, 
\[
    \lim_{n\to\infty} -\alpha\circ h'_{-n}(x)=\infty.
\]
Thus, by Lemma \ref{L:Kesten} we get that the expression in \eqref{eq:a++} is strictly 
positive for every $\alpha\in\Pi$.
Therefore  $\Lambda'\in\mathfrak{a}^{++}$, as claimed.

\begin{claim}
    $\Lambda'=\Lambda_{D'}$.
\end{claim}
By Birkhoff's Ergodic Theorem, for $m'$-a.e $x\in X'$ ,
\[
    \lim_{n\to\infty} \frac{1}{n}q(D'_n(x))=	\lim_{n\to\infty}\frac{1}{n}h'_n(x)	=
\int_{X'} h'(x)\dd m'(x)= \Lambda' \in \mathfrak{a}^{++}.
\]
Since $\mathfrak{a}^{++}$ is open in $\mathfrak{a}$, we get that for $n$ large enough 
(depending on $x$), $q(D'_n(x))\in\mathfrak{a}^{++}$.
Thus by Remark~\ref{rem:kappa=q}, $q(D'_n(x))=\kappa(D'_n(x))$. 
We conclude that  
\[
    \Lambda_{D'}=\lim_{n\to\infty} \frac{1}{n}\kappa(D'_n(x))
    =\lim_{n\to\infty} \frac{1}{n}q(D'_n(x))=\Lambda'.
\]
Finally, combining our previous remarks we deduce
\[
    \Lambda_F=m(X')\cdot \Lambda_{F'}=m(X')\cdot \Lambda_{D'}=m(X')\cdot\Lambda' \in \mathfrak{a}^{++}.
\]
This completes the proof of the simplicity of the Lyapunov spectrum $\Lambda_F$.

\bigskip

\subsection{Proof of Theorem~\ref{T:main}.(2) -- continuity.}
\label{sub:proof_of_theorem_main_continuity}\hfill{}\\
Let us first establish a formula for $\Lambda_F$.
\begin{prop} \label{P:Iwasawa}
Let $(X,\calX,m,T,\Gamma,w)$ be an Apafic Greg.
Let $G$ be a connected semisimple real Lie group with finite center
and $\sigma:G\times G/P \to \mathfrak{a}$ be the corresponding Iwasawa cocycle, 
as in \S\ref{sub:cartan}, Equation~\eqref{e:Iwasawa}.
Let $\rho:\Gamma\to G$ be an integrable Zariski dense representation,
set $F=\rho\circ w:X\to G$, let $\phi_+:X_+ \to G/P$ be the $F$-map given in Corollary~\ref{C:boundary}
and set $\phi=\phi_+\circ p_+:X\to G/P$.
Then we have
\[ \Lambda_F=\int_{X} \sigma(F(x),\phi(x))\dd m(x). \]
\end{prop}
\begin{proof}
We use freely the setting and discussion of the previous section, \S\ref{sub:proof_of_theorem_ref_t_main}.
Note that by the Cocycle Reduction Lemma~\ref{l:reduction}
we have 
for a.e $x\in X$, $s(x)\theta(x)=eZ$.
By Corollary~\ref{C:boundary}, $\phi=\phi_+\circ p_+=\pr_+\circ \theta$, thus we get for a.e $x\in X$, $s(x)\phi(x)=eP$.
Applying the cocycle property of $\sigma$ to the equation $D'(x)=s(T'x)\cdot F'(x)\cdot s(x)^{-1}$ 
we get
\[ 
    \begin{split}
        h'(x)&=q(D'(x))=\sigma(D'(x),eP)= \sigma(D'(x),s(x)\phi(x))\\
        &=\sigma(s(T'x),F'(x)\phi(x))+\sigma(F'(x),\phi(x))+\sigma(s(x)^{-1},s(x)\phi(x))\\
        &=\sigma(s(T'x),\phi(T'x))+\sigma(F'(x),\phi(x))-\sigma(s(x),\phi(x)).
    \end{split}
\]
Therefore, as $T'$ is measure preserving, we get
\[
		\Lambda'=\int_{X'} h'(x)\dd m'(x)=\int_{X'} \sigma(F'(x),\phi(x))\dd m'(x)
\]
and using $m'=m(X')^{-1}\cdot m|_{X'}$ and applying the cocycle equation to $F'(x)=F_{n(x)}(x)$ we easily deduce
\[
		\Lambda'= \frac{1}{m(X')} \cdot \int_{X} \sigma(F(x),\phi(x))\dd m(x).
\]
Thus, by $\Lambda_F=m(X')\cdot \Lambda'$, we indeed get
\[ \Lambda_F=\int_{X} \sigma(F(x),\phi(x))\dd m(x). \]
\end{proof}

\medskip

Let us now prove part (2) of Theorem~\ref{T:main}.
Let $\Sigma\subset \Hom(\Gamma,G)$ be a uniformly integrable collection of representations
and $\rho\in \Sigma$ be a Zariski dense one.  
As $\Hom(\Gamma,G)$ is metrizable we argue to show that for a sequence of representations $\rho_i\in \Sigma$
which converges to $\rho$, 
we have $\|\Lambda_F-\Lambda_{F_i}\|_\mathfrak{a}\to 0$,
where $F_i=\rho_i\circ w$, $F=\rho\circ w$ and $\|-\|_\mathfrak{a}$ is a fixed norm on $\mathfrak{a}$.
Using Proposition~\ref{p:Zopen} 
we assume as we may that all the representations $\rho_i$ are Zariski dense.

We let 
$(\phi_+)_i,\phi_+:X_+ \overto{} G/P$ be, coresspondingly, the $F_i$-map and $F$-map given by Corollary~\ref{C:boundary}
and view these maps as elements of $Q=\Map(X_+,\Prob(G/P))$ using the Dirac map $\delta:G/P\to \Prob(G/P)$.
We endow $Q$ with the topology of converges in measure which is the same as the weak*-topolgy associated with its identification 
with a subset of $L^1(X_+,C(G/P))^*$.
We claim that the sequence $(\phi_+)_i$ converges to $\phi_+$ in the compact convex space $Q$.
By compactness and a sub-sub-sequence argument, it is enough to assume that the sequence $(\phi_+)_i$ converges
to some map $\phi'_+$ and show that $\phi'_+=\phi_+$.
But since $\phi'_+$ is $F$-equivariant, it must equal $\phi_+$ by 
the uniqueness statement in Corollary~\ref{C:boundary}.

We set $\phi_i=(\phi_+)_i\circ p_+$ and $\phi=\phi_+\circ p_+$ and conclude that the sequence $\phi_i$
converges to $\phi$ in the weak*-topology of $L^1(X,C(G/P))^*$.
Note that the map $x \mapsto \sigma(F(x),\cdot)$ is in $L^1(X,C(G/P,\mathfrak{a}))$ by 
Lemma~\ref{l:sigmabound} and the integrability of $F$.
We thus have, for every simple root $\alpha\in \Pi$, that
$x \mapsto \alpha(\sigma(F(x),\cdot))$ is in $L^1(X,C(G/P))$, 
and we deduce the convergence
\[ 
    \int_{X} \alpha(\sigma(F(x),\phi_i(x)))\dd m(x) 
    \to \int_{X} \alpha(\sigma(F(x),\phi(x)))\dd m(x). 
\]
We conclude that
\begin{equation} \label{eq:half1}
\left\|\int_{X} \sigma(F(x),\phi_i(x))\dd m(x) - \int_{X} \sigma(F(x),\phi(x))\dd m(x)\right\|_\mathfrak{a}\ \overto{}\ 0.
\end{equation}

Using Lemma~\ref{l:sigmabound} again, we get that the maps
$x \mapsto \sigma(F_i(x),\cdot)$ are uniformly integrable in $L^1(X,C(G/P,\mathfrak{a}))$.
By the converges $\rho_i\to \rho$ we have the convergence in measure of this sequence of maps 
to the function $x \mapsto \sigma(F_i(x),\cdot)$,
and we deduce the corresponding convergence in $L^1(X,C(G/P,\mathfrak{a}))$,
\[ 
    \int_{X} \sup_{\xi\in G/P}  \left\|\sigma(F_i(x),\xi) - \sigma(F(x),\xi)\right\|_\mathfrak{a} 
    \dd m(x)\ \overto{}\ 0, 
\]
and in particular
\[ 
    \int_{X}  \left\|\sigma(F_i(x),\phi_i(x)) - \sigma(F(x),\phi_i(x))\right\|_\mathfrak{a}
    \dd m(x)\ \overto{}\ 0.
\]
Therefore
\begin{equation} \label{eq:half2}
\left\|\int_{X}  \sigma(F_i(x),\phi_i(x)) - \sigma(F(x),\phi_i(x)) \dd m(x)\right\|_\mathfrak{a}\ \overto{}\ 0.
\end{equation}
Using Proposition~\ref{P:Iwasawa} we get by \eqref{eq:half1} and \eqref{eq:half2},
\[ 
    \begin{split}
        \lim_{i\to\infty} &\|\Lambda_{F_i}-\Lambda_{F}\|_\mathfrak{a} \\
        &=\lim _{i\to\infty}\left\| \int_{X} \left(\sigma(F_i(x),\phi_i(x)) -
        \sigma(F(x),\phi(x))\right)\dd m(x)\right\|_\mathfrak{a}=0.
    \end{split}
\]
This completes the proof.

\subsection{Proof of Theorem~\ref{T:main}.(3) -- positivity of the drift} 
	\label{sub:proof_of_theorem_positivity_of_the_drift}\hfill{}\\
        Note that for $G=\SL_d(\bbR)$ for 
        $g\in G$ one has $\log\|g\|=\|\kappa(g)\|_{\mathfrak{sl}_d}$.
        For a general semi-simple $G$ and a choice of length function $|-|$ on it,
        there are constants $0<c_1\le c_2<\infty$ so that
        $c_1 |g|\le \|\kappa(g)\|_{\mathfrak{a}}\le c_2|g|$.
        Given an integrable $F:X\to G$ over an ergodic system $(X,m,T)$,
        Kingman's theorem ensures that there is an a.e. (and $L^1$) 
        convergence to a constant
        \[
            \lim_{n\to\infty}\frac{1}{n}|F_n(x)|=\lambda.
        \]
        Positivity of the drift, i.e. $\Lambda_F\ne 0$, is equivalent to $\lambda>0$.
             
	Our goal is to show that if $\rho:\Gamma\to  G$ is a homomorphism
	whose image $\rho(\Gamma)$ is not contained in a solvable by compact extension
	in $G$, then the map $F=\rho\circ w:X\to  G$ 
	has positive drift.
	
	Let $H$ denote the Zariski closure of $\rho(\Gamma)$ in $G$. 
        Then $H$ contains a connected normal subgroup of finite index $H_0\normal H$ 
        that is a semi-direct product
	$H_0=L\ltimes S$ of a semisimple Lie group $L$ and a solvable group $S$.
	Note that $L$ cannot be compact, for then $\rho(\Gamma)$ 
        would be contained in a compact extension
	$H$ of a solvable group $S$, contradicting our assumption. 
	The non-compact connected semisimple Lie group $L$ admits an epimorphism
        onto a non-compact simple Lie group $G_0$ with trivial center, so we have
        an epimorphism
	\[
		\pi:H_0\overto{} L\overto{}  G_0.
	\]
        We denote by $|-|$ a length function on $G$, and by $|-|_0$ 
        a choice of a length function on $G_0$.
        Then there exists some $c_0>0$ so that for $h\in H_0$ we have 
        \begin{equation}\label{e:subquo}
            |h|\ge c_0\cdot|\pi(h)|_0 
        \end{equation}
        In the situation where $H=H_0$, the representation
        $\pi\circ \rho:\Gamma\overto{} H_0\overto{} G_0$
        is Zariski dense, and simplicity of the spectrum 
        (see \S\ref{sub:proof_of_theorem_ref_t_main} above) 
        for the map $F_0:X\overto{w} \Gamma\overto{\rho} H_0\overto{\pi} G_0$ 
        implies positivity of its drift and therefore 
        implies positivity of the drift for 
        $F=\rho\circ w$ using the inequality (\ref{e:subquo}).

        To deal with the case where $H$ is not connected, and therefore 
        $H_0$ is a proper subgroup of finite index in $H$,
        we will use a cocycle trick.
        Let $\Lambda=\rho^{-1}(H_0)$. This is a normal subgroup of finite 
        index in $\Gamma$. 
        Consider the cocycle 
        \[
            c:\Gamma\times \Gamma/\Lambda\overto{} \Lambda,
            \qquad 
            c(\gamma', \gamma\Lambda)
            =s(\gamma'\gamma\Lambda)^{-1} \gamma' s(\gamma\Lambda),
        \]
        where $s:\Gamma/\Lambda\to\Gamma$ is a cross-section,
        i.e. a choice of representatives in each coset.
        Then the cocycle 
        \[
		\rho_0:\Gamma\times \Gamma/\Lambda\overto{c} \Lambda\overto{\rho} H_0\to  L\overto{\pi} G_0
	\]
	is Zariski dense. Thus [to be justified in the updated version with cocycles]
        implies that the $G_0$-valued sequence :
        \[
            \rho_0(w_n(x),\gamma\Lambda)=\pi\circ \rho\left(s(w_n(x)\gamma\Lambda)^{-1}w_n(x)s(\gamma\Lambda)\right)
        \]
        has simple spectrum. In particular, it has positive drift:
        \[
            \lim_{n\to\infty} \frac{1}{n} \left|\rho_0(w_n(x),\gamma\Lambda)\right|_0=\lambda_0>0.
        \]
        For $m$-a.e. $x\in X$ and $\gamma\Lambda\in\Gamma/\Lambda$ we have the same limit
        \[
            \lambda=\lim_{n\to\infty} \frac{1}{n} \left|\rho\left(s(w_n(x)\gamma\Lambda)^{-1}w_n(x)s(\gamma\Lambda)\right)\right|
        \]
        It follows from (\ref{e:subquo}) that $\lambda\ge c_0 \lambda_0>0$.
        In view of the "triangle inequality" $|abc|\ge |b|-|a|-|c|$,
        and denoting $C=\max_{\gamma\Lambda} |s(\gamma\Lambda)|$ (recall $\Gamma/\Lambda$ 
        is finite), we have
        \[
            |F_n(x)|=|\rho(w_n(x)|\ge \left|\rho\left(s(w_n(x)\gamma\Lambda)^{-1}w_n(x)s(\gamma\Lambda)\right)\right|-2C.
        \]
        Therefore 
        \[
            \lim_{n\to\infty} \frac{1}{n} \left|F_n(x)\right|=\lambda\ge c_0\lambda_0>0.
        \]
        This proves positivity of the drift for the original system.
	\bigskip

        The proof of Theorem~\ref{T:main} is completed.
	
\bigskip

\section{Real Gregs}
\label{sec:realgregs}

In this paper we study thoroughly the notion of Greg, which is built up over a $\mathbb{Z}$-action. 
In this section we discuss the analogous notion of a \textit{real Greg}, which is built up over an $\mathbb{R}$-action, and we develop a machinery of passing between these two notions.
At the moment, we hold the temptation to develop the theory of real Gregs further in an independent manner.
Rather, we will associate with any given real Greg some standard (integer) Gregs.
We will do it in two ways. One is simply by restricting the $\mathbb{R}$-action to 
a $\mathbb{Z}$-action.
For the other way we will use the theory of \textit{cross sections}.

A real Greg is a tuple 
$(Y,\calY,\mu,\phi,c,\Gamma)$,
where $(Y,\calY,\mu)$ is a probability measure space, 
\[ \phi:\mathbb{R}\times Y \to Y, \quad (t,y) \mapsto \phi^t(y) \]
is a measure preserving $\mathbb{R}$-action, $\Gamma$ is a locally compact second countable group and 
\[ c:\mathbb{R}\times Y \to \Gamma, \quad (t,y) \mapsto c_t(y) \]
is a measurable cocycle,
that is 
$c_{t+s}(y)=c_s(\phi^t(y))c_t(y)$ for all $t,s\in\bbR$ 
and $\mu$-a.e. $y\in Y$.

Next, we consider a representation $\rho:\Gamma \to G$ into a connected semi-simple real Lie group G with finite center.
We form the measurable $G$-valued cocycle 
\[ \Phi=\rho\circ c:\mathbb{R}\times Y \to G, \quad (t,y)\mapsto \Phi_t(y). \]
Assuming ergodicity of the system $(Y,\calY,\mu,\phi)$ and the integrability condition
\begin{equation}\label{e:L1-for-flow}
    \int_Y\sup_{t\in[0,1]}|\Phi_t(y)|\dd\mu(y)<\infty,
\end{equation}
the flow version of Oseledets theorem implies that the limit
\begin{equation} \label{eq:realLyap}
    \Lambda_Y=\lim_{t\to\infty}\frac{1}{t}\cdot \kappa(\Phi_t(y))
\end{equation}
exists and is $\mu$-a.e. constant element of $\mathfrak{a}^+$.

Assume given a real Greg $(Y,\calY,\mu,\phi,c,\Gamma)$.
For any fixed $t\in \mathbb{R}^*$, we consider the transformation $\phi^t:Y\to Y$
and the map $c_t:Y\to \Gamma$, thus forming an integer Greg $(Y,\calY,\mu,\phi^t,c_t,\Gamma)$.
We regard it as the \textit{$t$-discretization Greg}.
If $(Y,\calY,\mu,\phi)$ is weakly mixing then the $t$-discretization is ergodic.
Note that the condition \eqref{e:L1-for-flow} implies the integrability of $|\Phi_t|$ for every $t\in \mathbb{R}$.
We thus may consider the associated Lyapunov spectrum of $(Y,\calY,\mu,\phi^t,c_t,\Gamma)$, which we will denote by $\Lambda_t$.

\begin{lemma} \label{lem:tdiscretization}
    Assume the real Greg $(Y,\calY,\mu,\phi,c,\Gamma)$ is weakly mixing and 
    let $\rho:\Gamma \to G$ be a representation into a connected semi-simple real Lie group G with finite center which satisfies the integrability condition \eqref{e:L1-for-flow}.
    Then $\Lambda_t=t\cdot \Lambda$.
\end{lemma}

\begin{proof}
$\Lambda_t=\lim_{n\to\infty}\frac{1}{n}\cdot \kappa(\Phi_{nt}(y))
    =t\cdot \lim_{n\to\infty}\frac{1}{nt}\cdot \kappa(\Phi_{nt}(y))=t\cdot \Lambda$.
\end{proof}

We recall that given a pmp $\mathbb{Z}$-system $(X,\calX,m,T)$ and an $L^1(X)$ function $r:X \to (0,\infty)$ satisfying $\sum_{n=1}^\infty r(T^n(x))=\infty$ and $\sum_{n=1}^\infty r(T^{-n}(x))=\infty$ one constructs the \textit{flow under the graph} of $r$ as follows.
We associate with $r$ an $\mathbb{R}$-valued cocycle on $X$ by setting 
\begin{equation} \label{eq:r_n}
	r_n(x):=\left\{\begin{array}{lll}
	r(T^{n-1}z)+\cdots +r(Tx)+r(x) & \textrm{if} & n>0,\\
	0 & \textrm{if} & n=0,\\
	-r(T^{-n}x)-\cdots -r(T^{-1}x) & \textrm{if} & n<0,
	\end{array}\right.
\end{equation}
which is an additive version of \eqref{e:f-cocycle}.
We define 
\[
    Y=X\times\bbR/(x,s)\sim (T^nx,s+r_n(x))
\]
and equip $Y$ with the normalized probability measure $\mu$ coming from $m\times m_\bbR$
on $X\times \bbR$, i.e. $\dd\mu(x,s)=(\int_X r\dd m)^{-1}\cdot \dd m(x)\dd s$.
Let $\phi^\bbR$ be the flow $\phi^t$ coming from $(x,s)\mapsto (x,s+t)$ on $X\times \bbR$.
One easily verifies that $(X,T)$ is ergodic iff $(Y,\phi^\bbR)$ is ergodic.

By the Ambrose-Kakutani theorem, \cite{Ambrose-Kakutani, Ambrose},
every pmp ergodic $\mathbb{R}$-action $(Y,\calY,\mu,\phi)$ 
is isomorphic to a flow under the graph for some pmp $\mathbb{Z}$-system $(X,\calX,m,T)$ and some $L^1(X)$ function $r:X \to (0,\infty)$ satisfying $\sum_{n=1}^\infty r(T^n(x))=\infty$ and $\sum_{n=1}^\infty r(T^{-n}(x))=\infty$.
Identifying $X$ with $X\times \{0\}$ we regard $X$ as \textit{Borel cross section} of $Y$ and the function $r$ is called \textit{the first return time} to $X$.
Given a locally compact second countable group $\Gamma$ and a measurable cocycle $c:\mathbb{R}\times Y \to \Gamma$ we define 
\begin{equation} \label{eq:wc}
w:X\to \Gamma, \quad w(x)=c_{r(x)}(x)    
\end{equation}
(for this to be well defined we assume as we may that $c$ is a strict cocycle, see \cite[Theorem B.9]{zimmer-book}). 
We thus obtain a Greg, $(X,\calX,m,T,w,\Gamma)$.
We regard it as the (integer) Greg associated with the real Greg for the given cross section $X\subset Y$.

 \begin{lemma}\label{L:Lya4flow}
    Assume the real Greg $(Y,\calY,\mu,\phi,c,\Gamma)$ is weakly mixing and 
    let $\rho:\Gamma \to G$ be a representation into a connected semi-simple real Lie group G with finite center which satisfies the integrability condition \eqref{e:L1-for-flow}.
Let $X\subset Y$ be a Borel cross section of $Y$, let $r$ be the associated first return time function and let $(X,\calX,m,T,w,\Gamma)$ be the associated (integer) Greg as discussed above.
Then $F=\rho\circ w$ is integrable, so the associated Lyapunov spectrum $\Lambda_F$ is defined, and we have 
    \[
        \Lambda_F=\left(\int_X r\dd m\right)\cdot \Lambda_{Y}.
    \]
\end{lemma}
\begin{proof}
    Let us first make a simplifying assumption that $r(x)$ is bounded away from $0$
    and that $\Phi_t(x)$ is bounded for small $t$.
    More precisely, assume there is $\epsilon_0>0$ and $C$ so that
    $r(x)\ge \epsilon_0$ and $|\Phi_t(x)|\le C$ for $t\in[0,\epsilon_0]$ and $x\in X$.
Fix some $\epsilon\in (0,\epsilon_0)$ and consider the $\epsilon$-discretization integer Greg $(Y,\calY,\mu,\phi^\epsilon,c_\epsilon,\Gamma)$.
By Lemma~\ref{lem:tdiscretization}, we have 
    Then $\Lambda_\epsilon=\epsilon\cdot \Lambda_Y$.
In the sequel we denote $S=\phi^{\epsilon}$.

    As above, we identify $Y$ with the flow under the graph of $r$ associated with $X$.
    Consider the thickening 
    \[
        Y'=X_{[0,\epsilon)}=\setdef{\phi^t(x)}{x\in X,\ t\in[0,\epsilon)} \simeq X\times [0,\epsilon)
    \]
    of $X$. Setting $R=\int_X r\dd m$, we have $\mu(Y')=\epsilon/R$.
    Let 
    \[ n:Y\to \mathbb{N}, \quad n(y)=\min\{n\in \mathbb{N}\mid S^n(y)\in Y'\}\]
    be the first return time to $Y'$ and let $S'$ be the transformation on $Y'$ induced by $S$, that is $S'(y)=S^{n(y)}(y)$.
    Denoting 
    \[ F':Y' \to G, \quad F'(y)=\Phi_{n(y)\epsilon}(y) \]
    and setting
    \[ F'_n(y)=F'((S')^{n-1}(y))\cdots F'(S'(y))F'(y), \]
    we get by Lemma~\ref{L:induced-Lya} that $F'$ is integrable and 
    \[ \Lambda'= \frac{R}{\epsilon} \cdot \Lambda_\epsilon = \frac{R}{\epsilon} \cdot \epsilon \Lambda_Y = \left(\int_X r\dd m\right) \cdot \Lambda_Y, \]
    where $\Lambda'$ is the Lyapunov spectrum of $F'$.

    Fix $y\in Y'$. Then $y=\phi^{s_0}(x)$ for some $x\in X$ and $s_0 \in [0,\epsilon)$.
    Note that $r(x)-s_0=\min\{t\in (0,\infty) \mid \phi^t(y)\in Y'\}$,
    thus $n(y)\epsilon=\lceil r(x)-s_0 \rceil$ and we have $n(y)\epsilon-(r(x)-s_0)<\epsilon$.
    We denote $s_1=n(y)\epsilon-(r(x)-s_0)$, thus $n(y)\epsilon+s_0=s_1+r(x)$ and $s_1\in [0,\epsilon)$.
    We get
    \[
        \begin{split}
        S'(y)& = S^{n(y)}(y)=\phi^{n(y)\epsilon}(y)=\phi^{n(y)\epsilon+s_0}(x) \\ & =\phi^{s_1+r(x)}(x)=\phi^{s_1}\phi^{r(x)}(x)=\phi^{s_1}(T(x)).
        \end{split}
    \]
    Inductively, we find for every $m\in\mathbb{N}$, $s_m\in [0,\epsilon)$
    such that $n(y)\epsilon+s_{m}=s_{m+1}+r(T^m(x))$ and
    conclude that $(S')^m(y)=\phi^{s_m}(T^m(x))$, thus 
    \[ 
    \begin{split}
    & F'((S')^{m}(y))\Phi_{s_m}(T^m(x))=F'(\phi^{s_m}(T^m(x)))\Phi_{s_m}(T^m(x)) \\ 
    & =\Phi_{n(y)\epsilon}(\phi^{s_m}(T^m(x)))\Phi_{s_m}(T^m(x))
    =\Phi_{n(y)\epsilon+s_m}(T^m(x)) \\
    & =\Phi_{s_{m+1}+r(T^m(x))}(T^m(x))=
    \Phi_{s_{m+1}}(\phi^{r(T^m(x))}T^m(x))
    \Phi_{r(T^m(x))}(T^m(x)) \\
    & =\Phi_{s_{m+1}}(T^{m+1}(x))
    F(T^m(x)).
    \end{split}
    \] 
    For the last equation we used the equation $F(x)=\rho\circ w(x)=\rho(c_{r(x)}(x))=\Phi_{r(x)}(x)$.
We get 
\[ F(T^m(x))=\Phi_{s_{m+1}}(T^{m+1}(x))^{-1}F'((S')^{m}(y))\Phi_{s_m}(T^m(x)),
\]
hence
\[ 
\begin{split}
F_n(x) & =F(T^{n-1}(x))\cdots F(T(x))F(x) \\
& = \Phi_{s_{n}}(T^{n}(x))^{-1}F'((S')^{n-1}(y))\cdots F'(S'(y))F'(y) \Phi_{s_0}(x) \\
& = \Phi_{s_{n}}(T^{n}(x))^{-1}F'_n(y)\Phi_{s_0}(x).
\end{split}
\]
By \eqref{e:L1-for-flow} and Proposition~\ref{P:Cartan-subadd}, we conclude that there is a constant $C$ such that 
    \[ \|\kappa(F_n(x))-\kappa(F'_n(y))\|_{\mathfrak{a}}<C.\]
    Specializing to $n=0$, we get that $F$ is integrable,
    thus 
    \[ \Lambda_F=\lim_{n\to \infty} \frac{1}{n}\kappa(F_n(x)) \]
    is defined and it is an a.e constant,
    and by picking $y$ generic, we conclude that
\[ \Lambda_F=\lim_{n\to \infty} \frac{1}{n}\kappa(F_n(x))=\lim_{n\to \infty} \frac{1}{n}\kappa(F'_n(y))=\Lambda'=  \left(\int_X r\dd m\right) \cdot 
\Lambda_Y,\]
as
    \[
        \left\|\frac{1}{n}\kappa\left(F'_n(y\right)-\frac{1}{n}\kappa\left(F_n(x)\right)\right\|_{\mathfrak{a}}\le \frac{C}{n}.
    \]
This finishes the proof under the assumption that $r(x)$ is bounded away from 0.

    We will now treat the general case by using an induced system.    
    Fix a positive $m$-measure subset $X'\subset X$
    over which $r$ is bounded away from $0$ and $\sup_{t\in [0,1]}|\Phi_t(x)|$
    is bounded. 
    We denote by $n':X' \to \mathbb{N}\cup \{\infty\}$
    the first return time for this subsystem of $(X,T)$
    and define the system $(X',m',T')$ and the cocycle $n'_k(x)$ as in \eqref{eq:return} and the preceding discussion.
    We observe that the induced system $(X',m',T')$ 
    constitutes another cross-section for the flow $(Y,\mu,\phi)$
    and we denote by $r':X'\to (0,\infty]$ the first return time for this cross-section and by $r'_k$ the analogue of \eqref{eq:r_n}. 
    Then we have the obvious relations $r'(x)=r_{n'(x)}(x)$ and $r'_k(x)=r_{n'_k(x)}(x)$.
by Kac lemma we have for $m'$-a.e. $x\in X'$,
    \[
        \int_{X'} r'\dd m'=\lim_{k\to\infty} \frac{r'_k(x)}{k}
        =\lim_{k\to\infty} \frac{n_k(x)}{k}\cdot \frac{r_{n_k(x)}(x)}{n_k(x)}
        =\frac{1}{m(X')}\cdot \int_{X} r\dd m.
    \]

We denote by $(X',\calX',m',T',w',\Gamma)$ the associated (integer) Greg as discussed above.
By the already proven part of the lemma, we have that $\rho\circ w'$ is integrable, so the associated Lyapunov spectrum $\Lambda_{\rho\circ w'}$ is defined, and we have 
    \[
        \Lambda_{\rho\circ w'}=\left(\int_X r\dd m\right)\cdot \Lambda_{Y}.
    \]
Using Lemma~\ref{L:induced-Lya} once more, we get that $F$ is integrable and 
    \[
        \Lambda_{F}=m(X')\cdot \Lambda_{\rho\circ w'}=m(X')\cdot \left(\int_{X'} r'\dd m'\right)\cdot \Lambda_{Y}
        =\left(\int_{X} r\dd m\right)\cdot \Lambda_{Y}.
    \]
    This completes the proof of the Lemma.
\end{proof}

Next, we draw the following conclusion from Theorem~\ref{T:main} for real Gregs, 
assuming the cocycle $c$ is uniformly bounded on bounded time intervals.

 \begin{cor}\label{cor:Lya4flowApafic}
    Assume the real Greg $(Y,\calY,\mu,\phi,c,\Gamma)$ is weakly mixing
    and assume that the image in $\Gamma$ under $c$ of the set $[0,1]\times Y$ is precompact.
Assume that there exists a Borel cross section $X\subset Y$ such that the associated (integer) Greg $(X,\calX,m,T,w,\Gamma)$ discussed above is Apafic.
Consider the subset $\Hom_Z(\Gamma,G)\subset \Hom(\Gamma,G)$ consisting of all Zariski dense representations.
Then for every $\rho\in \Hom_Z(\Gamma,G)$, the $\bbR$-cocycle $\rho\circ c$ is integrable and has a simple Lyapunov spectrum 
    $\Lambda_{\rho\circ c}\in\mathfrak{a}^{++}$,
    and the associated map 
    \[ \Hom_Z(\Gamma,G) \to \mathfrak{a}^{++}, \quad \rho \mapsto \Lambda_{\rho\circ c}\]
    is continuous.
\end{cor}

\begin{proof}
Given Lemma~\ref{L:Lya4flow}, this is a straight forward corollary of Theorem~\ref{T:main}, once we observe that the cocycles $\rho\circ w$ are uniformly integrable, when $\rho$ varies on a compact subset $K\subset \Hom(\Gamma,G)$.
This is indeed the case, as by \eqref{eq:wc}, $|\rho\circ w|\leq h$,
where $h=M\cdot r$, for $M=\max_{\rho\in K}|\rho\circ c([0,1]\times Y)|$ ,
so $\rho\circ c$ is contained in the uniformly integrable set $\Sigma_h$ described in \eqref{eq:hunif}.	
\end{proof}

\section{Negatively curved manifolds -- Proof of Theorem~\ref{T:ncm-flow}}
\label{sec:geo-example}

In this section we prove Theorem~\ref{T:ncm-flow}
by deducing it from the already proven Corollary~\ref{cor:Lya4flowApafic} of Theorem~\ref{T:main} and from Theorem~\ref{T:nc-ApaficGreg} below, which is the main result of this section.
Theorem~\ref{T:nc-ApaficGreg} says essentially that the canonical cohomology class associated with a Gibbs measure on a negatively curved manifold could be encoded into an Apafic Greg.
In order to formulate it properly, we will introduce some terminology.

Let $(M,g)$ be a closed negatively curved Riemannian manifold.
Denote by $\Gamma=\pi_1(M)$ the fundamental group,
acting isometrically via deck transformations on the 
the universal cover $(\wt{M},\wt{g})$.
This action is free and cocompact, with $M=\wt{M}/\Gamma$ being the quotient.
The unit tangent bundle $T^1\wt{M}$ is equipped
with the commuting actions of $\Gamma$ and of the 
geodesic flow $\wt{\phi}^\bbR$; the unit tangent bundle $T^1M$
with its geodesic flow $\phi^\bbR$
is the quotient $T^1\wt{M}/\Gamma$ by $\Gamma$.

The action of $\Gamma$ and $\mathbb{R}$ on $T^1\wt{M}$ could be encoded into a cocycle as follows.
For a choice of a fundamental domain $\mathscr{F}$ for the $\Gamma$ action on $T^1\wt{M}$ we define the Borel cocycle
\begin{equation} \label{eq:can}
c=c_\mathscr{F}: \mathbb{R} \times T^1M \to \Gamma
\end{equation}
by setting $c_t(u)=g$ for the unique element $g\in \Gamma$ for which $g\wt{\phi}^t(\wh{u})\in \mathscr{F}$, where $\wh{u}\in \mathscr{F}$ is lift of $u\in T^1M$.
Different choices of fundamental domains give cohomological cocycles
and the corresponding cohomological class is called \textit{canonical}.

Given a fundamental domain $\mathscr{F}\subset T^1\wt{M}$ and a $\phi$-invariant probability measure $\mu$ on $T^1M$ we get the real Greg,
$(T^1M,\calB,\mu,\phi,c_\mathscr{F},\Gamma)$,
where $\calB$ is the Borel $\sigma$-algenbra on $T^1M$.

\begin{theorem}\label{T:nc-ApaficGreg}
    Let $(M,g)$ be a negatively curved compact Riemannian manifold,
    let $\mathscr{F}$ be a fundamental domain for the action of $\pi_1(M)$ on $T^1\wt{M}$ which has a non-empty interior and let $\mu$ be a Gibbs measure on $T^1M$.
    Then the real Greg $(T^1M,\calB,\mu,\phi,c_\mathscr{F},\Gamma)$ admits a Borel cross section $X\subset T^1M$ such that the associate (integer) Greg $(X,\calX,m,T,w,\Gamma)$, see \eqref{eq:wc}, is Apafic.
\end{theorem}

\begin{remark} \label{rem:gibbs}
    In fact, we will see that the Gibbs assumption on the invariant measure on $T^1M$ could be replaced by a \textit{local product property} assumption. This will be discussed in the following subsection, see in particular Definition~\ref{def:localprod} below.    
\end{remark}

\begin{proof}[Proof of Theorem~\ref{T:ncm-flow}]
We consider the real Greg $(T^1M,\calB,\mu,\phi,c_\mathscr{F},\Gamma)$.
Using the fact that $\mathscr{F}$ is bounded and with a non-trivial interior, 
the $1$-neighborhood of $\mathscr{F}$
    in $T^1\wt{M}$ is covered by finitely many translations of $\mathscr{F}$.
It follows that the image in $\Gamma$ under $c$ of the set $[0,1]\times T^1M$ is finite. 
By
Theorem~\ref{T:nc-ApaficGreg}
the real Greg $(T^1M,\calB,\mu,\phi,c_\mathscr{F},\Gamma)$ admits a Borel cross section $X\subset T^1M$ such that the associate (integer) Greg $(X,\calX,m,T,w,\Gamma)$ is Apafic.
The proof follows by Corollary~\ref{cor:Lya4flowApafic},
as Gibbs measures are weakly mixing (in fact, mixing, cf. \cite{CHERNOV}).
\end{proof}

\subsection{Setting up}\hfill{}\\
\label{subsec:NCsetup}

We keep the setup discussed above and provide more.

\medskip

We endow the spaces $T^1\wt{M}$ and $T^1M$ with the \emph{Sasaki} Riemannian structures are derived from the Riemannian structures on $\wt{M}$ and $M$ correspondingly.
We thus regard the spaces $M$, $\wt{M}$, $T^1M$ and $T^1\wt{M}$ as metric spaces with corresponding distance functions $d_M$, $d_{\wt{M}}$, $d_{T^1M}$ and $d_{T^1\wt{M}}$, all of which are denoted simply by $d(\cdot,\cdot)$ when there is no chance of confusion. 

We denote by $\partial\wt{M}$ the \textit{ideal boundary} 
of the universal cover $\wt{M}$ of $M$, which is an example 
of a proper ${\rm CAT}(-\kappa)$ space. 
We identify it with the Gromov boundary of the hyperbolic group $\Gamma$, $\partial\wt{M}\simeq \partial \Gamma$. 
We denote by $\partial^2\wt{M}=\setdef{(\xi,\eta)\in\partial\wt{M}^2}{\xi\ne\eta}$ the space distinct pairs of ideal points.

For $u\in T^1\wt{M}$ we denote by $L_u:\bbR\to\wt{M}$ the geodesic satisfying $\frac{\dd}{\dd t}L_u(0)=u$
and we let $u_\pm\in\partial\wt{M}$ be the ideal end-points of this geodesic, 
namely $u_\pm=\lim_{t\to\pm\infty} L_u(t)$.
We obtain maps
\[ p_\pm:T^1\wt{M}\overto{}\partial\wt{M}, \quad u\mapsto u_\pm \]
and
\[ p_{\bowtie}:T^1\wt{M}\overto{}\partial^2\wt{M}, \quad u\mapsto (u_-,u_+). \]
Noticing that the fibers of $p_{\bowtie}$ are exactly the geodesics in $\wt{M}$,
we identify $\partial^2\wt{M}$ with the set of oriented, unparametrized geodesics in $\wt{M}$.

The fibers of $p_+$ are called the \emph{stable submanifolds} of $T^1\wt{M}$.
Elements in the same fiber are said to be \emph{stably equivalent}.
Elements $u,v\in T^1\wt{M}$ such that 
$\lim_{t\to\infty} d(\wt{\phi}^t(u),\wt{\phi}^t(v))=0$ are said to be \textit{strongly stable equivalent}, and this forms a finer equivalence relation. The corresponding equivalence classes are  the \textit{strongly stable submanifolds} of $T^1\wt{M}$.
For every pair of stably equivalent $u,v\in T^1\wt{M}$ there exists a unique $t\in \mathbb{R}$ such that $u$ is strongly stably equivalent to $\wt{\phi}^t(v)$.
The contraction of the strongly stable submanifolds is uniformly exponentially fast in the following sense.
There exists $c_0>0$ such that for every strongly equivalent $u,v\in T^1\wt{M}$ there exists $C_{u,v}>0$ such that for every $t\geq 0$,
    \begin{equation} \label{eq:Wss}
        d(\wt{\phi}^t(u),\wt{\phi}^t(v))<C_{u,v} e^{-c_0 t}.
    \end{equation}




We also recall (cf. Kaimanovich \cite{Kaim-geo}) 
that there is a natural correspondence between 
finite $\phi^\bbR$-invariant measures on $T^1M$, 
Radon measures on $T^1\wt{M}$ that are invariant under $\wt{\phi}^\bbR\times\Gamma$,
and Radon $\Gamma$-invariant measures on $\partial^2\wt{M}$.
Given a $\phi^\bbR$-invariant measure $\mu$ on $T^1M$, we denote by
$\wt{\mu}$ its $\Gamma\times\wt{\phi}^\bbR$-invariant
lift to $T^1\wt{M}$, and by $\ol{\mu}$ the $\Gamma$-invariant Radon
measure on $\partial^2\wt{M}$ such that $\wt{\mu}$ can be identified
with $\ol{\mu}\times {\rm Leb}_\bbR$ on 
$\partial^2\wt{M}\times\bbR\cong T^1\wt{M}$ (aka Hopf coordinates).
Since we are mostly interested in measure classes,
it is convenient to replace infinite measures 
by equivalent probability measures.
Let $\wt{\mu}^1$ be a probability measure on $T^1\wt{M}$
in the measure class of $\wt{\mu}$, coming from $\mu$ on $T^1M$.
Let 
\[
    \ol{\mu}^1\in\Prob(\partial^2\wt{M}),
    \qquad 
    \ol{\mu}^1_-\in\Prob(\partial \wt{M}),
    \qquad
    \ol{\mu}_+^1\in \Prob(\partial \wt{M})
\]
be the push-forwards of the probability measure $\wt{\mu}^1$ on $T^1\wt{M}$
under the maps $p_{\bowtie}:u\mapsto (u_-,u_+)$,
$p_-:u\mapsto u_-$, $p_+:u\mapsto u_+$, respectively.

\begin{defn} \label{def:localprod}
We say that $\mu$ has the \textit{local product property} if
there is equivalence of measure classes 
\[
    \ol{\mu}^1\sim \ol{\mu}^1_-\times\ol{\mu}^1_+.
\]    
\end{defn}

Observe that this condition is independent of the specific 
choice of $\wt{\mu}^1$ in the measure class of $\wt{\mu}$,
and therefore it is a condition on the $\phi^\bbR$-invariant 
measure $\mu$ on $T^1M$.

\begin{lemma} \label{lem:Gibbs}
    If $\mu$ is a Gibbs measure then it satisfies the local product property.
\end{lemma}

\begin{proof}
    Geodesic flow on a negatively curved manifolds, and more generally Anosov flows,
    admit Markov partitions that relate the flow to the flow under a H\'older
    continuous function over a topologically mixing subshift of finite type
    (see Bowen \cite{Bowen-book}, \cite{Bowen-Anosov-flows}, and more recent survey Chernov \cite{CHERNOV}).
    Gibbs measures for the flow  correspond to Gibbs measures on the subshift of finite type,
    and the local product structure for the latter is well known.
\end{proof}


\subsection{A preliminary lemma} \label{sec:prelem}

We will need the following general but technical lemma.
\begin{lemma}\label{L:good_radius}
    Let $\nu$ be a finite Borel measure on a metric space $(S,d)$.
    Given $0<c<1$ and $p\in S$, there exists a dense set of numbers $\alpha>0$ for which there exists $\epsilon_0>0$ (depending on $\alpha$, but not on $c$ and $p$) such that
    \[
        \nu\setdef{x\in S}{\alpha-\epsilon<d(x,p)<\alpha+\epsilon}<\epsilon^c
    \]
    for all $0<\epsilon<\epsilon_0$.
\end{lemma}
\begin{proof}
    We first consider $S=\bbR$.
    We fix a finite Borel measure $\nu$ on $\bbR$, $c<1$
and a closed interval $I\subset \bbR$ and claim that there is a
    point $\alpha\in I$ and $\epsilon_0>0$ so that for all $0<\epsilon\le \epsilon_0$,
    \begin{equation}
         \label{eq:repsi}
        \nu((\alpha-\epsilon,\alpha+\epsilon))<\epsilon^c.
     \end{equation}
     Here we assume that $I$ is not a singleton. We also assume as we may that it is bounded.
    Choose $n_1$ such that $c(1+1/n_1)<1$
    and then choose $N$ so that $N^{1-c(1+1/n_1)}>3$.
    Denoting by $|I|$ the length of the interval $I$,
    we now choose $n_2\ge n_1$ so that
    \[  \left(\frac{N^{1-c(1+1/n_1)}}{3}\right)^{n_2} > \frac{\nu(I)}{|I|^{c}}\]
    and set $\epsilon_0=|I|/N^{n_2}$.
    We argue to find an $\alpha\in I$ satisfying the claim.

    We will consider a nested sequence of intervals 
    \[
        I=J_0=I_0\supset J_1\supset I_1\supset J_2\supset I_2\supset 
        \dots 
    \]
    satisfying
    \[|I_n|=\frac{|I|}{N^n}, \quad |J_n|=\frac{3|I|}{N^n}, \]
    where $J_0=I_0=I$ and the sequence is constructed inductively as follows. 
    Fix $n$ and assume that $J_n$ and $I_n$ were already constructed.
    Divide the interval $I_n$ into $N$ equal 
    length subintervals $I_{n,1}$, $I_{n,2}$,..., $I_{n,N}$ with indices 
    corresponding to the order. Consider blocks of consecutive triples 
    $J_{n,i}=I_{n,i-1}\cup I_{n,i}\cup I_{n,i+1}$ and 
    consider an index $2\leq i_0 \leq N-1$ such that $\nu(J_{n,i_0})$ is minimal among the $\nu(J_{n,i})$. Now set $J_{n+1}=J_{n,i_0}$ and $I_{n+1}=I_{n,i_0}$.
        Finally, we let $\alpha\in I$ be the intersection point: $\{\alpha\}=\bigcap I_n$.

    We now fix $\epsilon\le \epsilon_0$. All things are in place for proving the claim, we are just left to show \eqref{eq:repsi}.
    Fix $n$.
    Running over $2\leq i \leq N-1$, the union of the $N-2$ intervals $J_{n,i}$ covers $I_n$ with multiplicity 
    at most $3$, thus the average measure of the intervals $J_{n,i}$ is bounded from above by $\frac{3}{N-2}\cdot \nu(I_n)$ and in particular, 
    \[
        \nu(J_{n+1})=\nu(J_{n,i_0})\le \frac{3}{N-2}\cdot \nu(I_n)\le \frac{3}{N}\cdot \nu(J_n). 
    \]
By induction we get that 
 \[ \nu(J_{n}) \le \left(\frac{3}{N}\right)^n\cdot \nu(I). \]
 We let $n_3$ to be the unique integer satisfying 
$|I|/N^{n_3+1}<\epsilon\le |I|/N^{n_3}$.
    Note that
 for $n\le n_2$, $|I|/N^n\geq |I|/N^{n_2} =\epsilon_0 \ge \epsilon$, thus $n_3\ge n_2$ and we get 
    \[ \left(\frac{N}{3}\right)^{n_3} \cdot \frac{1}{(N^{n_3+1})^{c}}
    =\left(\frac{N^{1-c(1+1/n_3)}}{3}\right)^{n_3} \ge \left(\frac{N^{1-c(1+1/n_1)}}{3}\right)^{n_2} > \frac{\nu(I)}{|I|^{c}}. \]
By $\epsilon\le |I|/N^{n_3}=|I_{n_3}|$ we get that $(\alpha-\epsilon,\alpha+\epsilon)\subset J_{n_3}$, therefore
\begin{align*}
        \nu\left((\alpha-\epsilon,\alpha+\epsilon)\right)\le \nu(J_{n_3})
        \le \left(\frac{3}{N}\right)^{n_3}\cdot \nu(I)
        \le \frac{|I|^c}{(N^{n_3+1})^c}
        <\epsilon^c,
\end{align*}
showing \eqref{eq:repsi}, thus proving the claim.

    Now let $\nu$ be a finite Borel measure on a metric space $(S,d)$,
    point $p\in S$, and $0<\alpha_1<\alpha_2$ two radii.
    Choose a closed subinterval $I\subset (\alpha_1,\alpha_2)$ and apply
    the claim to the pushforward of $\nu$
    to $[0,\infty)$ by the map $x\mapsto d(p,x)$ to deduce Lemma.
\end{proof}

\subsection{A good choice of a Borel cross section}
\label{sec:goodchoice}

In this subsection we 
fix a negatively curved compact Riemannian manifold $(M,g)$ and we use the setup introduced in the beginning of \S\ref{sec:geo-example}
and in \S\ref{subsec:NCsetup}.
Our goal is to 
choose a well behaved Borel crosse section $X\subset T^1M$ 
and discuss the properties of its associated Greg,
towards the proof of Theorem~\ref{T:nc-ApaficGreg}.
We fix a fundamental domain with non-empty interior $\mathscr{F}$ for the $\Gamma$ action on $T^1\wt{M}$.
This choice determines a section $T^1M\to T^1\wt{M}$ which we denote by $u\mapsto \wh{u}\in \mathscr{F}$.
We define $c=c_\mathscr{F}$ as in \eqref{eq:can}.
We also fix a fully supported $\phi$-invariant ergodic probability measure $\mu$ on $T^1M$.
We thus get a real Greg,
$(T^1M,\calB,\mu,\phi,c,\Gamma)$.


We fix a point $p\in T^1M$ such that $\wh{p}$ is an inner point of $\mathscr{F}$.
We identify the tangent spaces $T_pT^1M\simeq T_pT^1\wt{M}$ and denote them by $V$.
We endow $V$ with the Sasaki inner product and with the corresponding metric.
Recall that \emph{geodesic spray} is the vector field on $T^1M$ corresponding to the geodesic flow. 
We let $v\in V$ be its value at $p$ and denote $U=v^\perp<V$.
We fix an open ball $B$ around 0 in $V$, small enough so that 
the exponential map $\exp:B\to T^1\wt{M}$ is defined and it is a diffeomorphism on its image, which is an open subset contained in $\mathscr{F}$.
We note that the differential at $0\in V$ of the map $\Phi$,
\begin{equation} \label{eq:mapPhi}
V =U \oplus \mathbb{R}v \supset (B\cap U) \times \mathbb{R}v \overset{\Phi}{\longrightarrow} T^1M, \quad (u,tv) \mapsto \phi^t\exp u,   
\end{equation}
is the identity map $V\to V$, so $\Phi$ is a local diffeomorphism at $0$.
Shrinking $B$, we assume as we may that $\Phi$ is a diffeomorphism on $(B\cap U) \times (-\delta,\delta)v$ for some $\delta>0$.
We denote by $X'$ and $\wh{X}'$ the images under $\exp$ of $B\cap U$ in $T^1M$ and $\mathscr{F}\subset T^1\wt{M}$ correspondingly
(the prime decorating $X$ hints that this is a first approximation of the cross-section we're after, but it is yet to be modified).
We also denote the corresponding images of $(B\cap U) \times (-\delta,\delta)v$ by
\[ X'_{(-\delta,\delta)}= \{{\phi}^t(x) \mid x\in X',~|t|< \delta \}, \quad\wh{X}'_{(-\delta,\delta)}= \{\wt{\phi}^t(x) \mid \wh{x}\in \wh{X}',~|t|< \delta \}. \]
Shrinking further $B$ and $\delta$, we assume as we may that $T^1\wt{M}\to T^1M$ maps $\wh{X}'_{(-\delta,\delta)}$ isometrically onto $X'_{(-\delta,\delta)}$ and $\Phi$ is a $K$-bi-Lipschitz isomorphism of these spaces with $(B\cap U) \times (-\delta,\delta)v\subset V$  for some $K\geq 1$.

Since $\mu$ is fully supported, $X'_{(-\delta,\delta)}$ has a positive measure,
and since $\mu$ is ergodic, it meets a.e $\mathbb{R}$-orbit.
It follows that $X'$ meets a.e $\mathbb{R}$-orbit in a discrete sets of time which  is unbounded from above and from below.
Therefore $X'$ gives rise to a Borel cross section
for the geodesic flow, 
giving an associated Greg $(X',\calX',m',T',w',\Gamma)$.

By Lemma~\ref{L:good_radius}, used for $c=1/2$,
we find $\alpha>0$ and $\epsilon_0>0$ such that 
\begin{equation} \label{eq:alpha}
\mbox{the ball }B(p,\alpha)\subset T^1M \mbox{ is contained in }X'_{(-\delta,\delta)}
\end{equation}
and 
\begin{equation} \label{eq:boundary'}
\mbox{for all } 0<\epsilon<\epsilon_0, \quad
        m'(\setdef{x\in X'}{d(x,p)>\alpha-\epsilon})<\sqrt{\epsilon}.    
\end{equation}
We define 
\[ X= X' \cap B(p,\alpha)=\setdef{x\in X'}{d(x,p)< \alpha},
\quad \wh{X}=\{\wh{x} \mid x\in X\}. \]

As before, $X$ meets a.e $\mathbb{R}$-orbit in a discrete and unbounded from above and from  below sets of times,
thus it gives rise to a Borel cross section
for the geodesic flow and an associated Greg
$(X,\calX,m,T,w,\Gamma)$.
This is finally the Greg we are after, but the proof of Theorem~\ref{T:nc-ApaficGreg} requires some more preparation.

\begin{lemma} \label{lem:Lip}
    For $v\in T^1\wt{M}$, if there exists $\wh{x}\in \wh{X}$ such that \[ (K+1)d(v,\wh{x})+d(\wh{x},\wh{p})<\alpha \] then there exists $t\in \mathbb{R}$ such that $|t|\leq Kd(v,\wh{x})$ and $\wt{\phi}^t(v)\in \wh{X}$.
\end{lemma}

\begin{proof}
Clearly, 
$d(v,\wh{p}) \leq d(v,\wh{x})+d(\wh{x},\wh{p})\leq (K+1)d(v,\wh{x})+d(\wh{x},\wh{p})<\alpha$, thus $v\in B(\wh{p},\alpha)\subset \wh{X}'_{(-\delta,\delta)}$ and we get $v=\wt{\phi}^{-t}(\exp(u))=\wt{\Phi}(u,-t)$ for some $u\in U\cap B$ and $|t|<\delta$.
Here $\wt{\Phi}$ is the lift of the map $\Phi$ defined in \eqref{eq:mapPhi}.
We also have $\wh{x}=\exp(w)=\wt{\Phi}(w,0)$ for some $w\in B\cap U$.
    Recall that $\wt{\Phi}$ is $K$-bi-Lipschitz isomorphic to $(B\cap U) \times (-\delta,\delta)v$ and $\mathbb{R}v\perp U$, so
    \[ |t|\leq \sqrt{t^2+\|u-w\|^2}=d_V((u,t)(w,0))=d_V(\wt{\Phi}^{-1}(v),\wt{\Phi}^{-1}(\wh{x}))\leq Kd(v,\wh{x}).\]
    Finally, 
    \[ d(\wt{\phi}^t(v),\wh{p})\leq d(\wt{\phi}^t(v),v)+d(v,\wh{x})+d(\wh{x},\wh{p})= t+d(v,\wh{x})+d(\wh{x},\wh{p}) \]
    \[\leq (K+1)d(v,\wh{x})+d(\wh{x},\wh{p})<\alpha, \]
    thus $\wt{\phi}^t(v)\in B(\wh{p},\alpha)\cap \wh{X}'=\wh{X}$.
\end{proof}

Since the measure $m$ on $X$ coincides with the normalized restriction of $m'$ to $X$, we set $C=m'(X)^{-1}$ and rewrite \eqref{eq:boundary'} as
\begin{equation} \label{eq:boundary'2}
\mbox{for all } 0<\epsilon<\epsilon_0, \quad
        m(\setdef{x\in X}{d(x,p)}>\alpha-\epsilon)<C\sqrt{\epsilon}.
\end{equation}

\begin{lemma} \label{lem:BC}
    For $m$-a.e $x\in X$ and for every $c>0$, there exists $N=N_c(x)$ such that for every $n\geq N$, $d(T^n(x),p)\leq \alpha-e^{-cn}$.
\end{lemma}

\begin{proof}
    It is enough to consider $c$ rational, hence we can fix $c>0$ and prove the existence of $N_c(x)$ for this fixed $c$ and for a.e $x$. Set
    \[
        E_n=\setdef{x\in X}{d(x,p) > \alpha-e^{-cn}}.
    \]
    For every $n$ large enough so that $e^{-cn}<\epsilon_0$,
    by \eqref{eq:boundary'2}, we have 
    $m(T^{-n} E_n)=m(E_n)<Ce^{-cn/2}$,
    which is a summable series.
    Hence by the Borel--Cantelli Lemma
    \[
        \setdef{x\in X}{T^n x\in E_n\ \textrm{infinitely\ often}}=
        \bigcap_{k=1}^\infty\bigcup_{n=k}^\infty T^{-n}E_n
    \]
    has $m$-measure zero.
    Therefore, $m$-a.e $x$ is in its complement, which proves the lemma. 
    \end{proof}

We let $r:X\to (0,\infty)$ be the first return time to $X$, as discussed in \S\ref{sec:realgregs},
and we let $r_n$ be defined as in \eqref{eq:r_n}.
These are defined for a.e $x\in X$.
By construction, we have that for a.e $x\in X$, $r(x)\geq 2\delta$.
Accordingly, for a.e $x\in X$, for all $n$,
\begin{equation} \label{eq:r>eps}
  r_{n+1}(x)-r_n(x)\geq 2\delta.
\end{equation}
In particular, $r_n(x)\to \pm \infty$ as $n\to \pm\infty$.

By \eqref{eq:wc} we get $c_{r_n(x)}(x)=w_n(x)$, thus by the defining property of the cocycle $c$, \eqref{eq:can}, we get that for $m$-a.e $x\in X$ and for all $n\in\bbZ$ we have
\begin{equation}\label{e:wn-phirn}
        \wh{T^n(x)}=w_n(x)\wt{\phi}^{r_n(x)}(\wh{x}).
\end{equation}


\begin{lemma} \label{lem:katlasst}
    For $m$-a.e $x,y\in X$, if there exists $g$ such that $\wh{x}_+=g\wh{y}_+$ then there exists $k\in \mathbb{Z}$ such that for every $n$ large enough, $w_{n+k}(y)=w_n(x)g$.
\end{lemma}

\begin{proof}
    We assume both $x$ and $y$ satisfy Lemma~\ref{lem:BC} and $\wh{x}_+=g\wh{y}_+$.
    Since $\wh{x}$ and $g\wh{y}$ are stably equivalent, there is $s\in \mathbb{R}$ such that 
    \begin{equation} \label{eq:sstab}
        \mbox{$\wh{x}$ and  $\wt{\phi}^s(g\wh{y})$ are strongly stably equivalent.}
    \end{equation}
We denote
    \[ x_n=\wh{T^n x}=w_n(x)\wt{\phi}^{r_n(x)}(\wh{x}), \quad  u_n=w_n(x)\wt{\phi}^{r_n(x)}(u)=w_n(x)g\wt{\phi}^{r_n(x)+s}(\wh{y}). \]
    
    By \eqref{eq:Wss} there exist constants $c_0$ and $C$ such that for every $t\geq 0$, $d(\wt{\phi}^t(\wh{x}),\wt{\phi}^t(u))<Ce^{-c_0 t}$.
By \eqref{eq:r>eps} we have $2c_0n\delta\leq c_0r_n(x)$.
Fixing $0<c<2c_0\delta$, we thus have for every $n$ large enough, 
\[ (K+1)C < e^{(2c_0\delta-c)n} \leq e^{c_0r_n(x)-cn}, \]
hence 
    \[
        d(x_n,u_n)=d(\wt{\phi}^{r_n(x)}(\wh{x}),\wt{\phi}^{r_n(x)}(u))<Ce^{-c_0r_n(x)}<(K+1)^{-1}e^{-cn}.
    \]
    By Lemma~\ref{lem:BC}, for every $n$ large enough we also have $d(x_n,\wh{p})\leq \alpha-e^{-cn}$, so
    \[ (K+1)d(u_n,x_n)+d(x_n,\wh{p})<e^{-cn}+(\alpha-e^{-cn})=\alpha. \]
We deduce by Lemma~\ref{lem:Lip}
    that there exist $t_n\in \mathbb{R}$ such that $|t_n|\leq Kd(u_n,x_n)$ and $\wt{\phi}^{t_n}(u_n)\in \wh{X}$.
    Note that 
    \[ |t_n|\leq Kd(u_n,x_n) \leq K(K+1)^{-1}e^{-cn}<e^{-cn}, \]
Thus, for every $n$ large enough there exists $t_n\in \mathbb{R}$ such that  $|t_n|<\delta$ and 
\[ w_n(x)g\wt{\phi}^{r_n(x)+t_n+s}(\wh{y})=\wt{\phi}^{t_n}(w_n(x)g\wt{\phi}^{r_n(x)+s}(\wh{y}))=\wt{\phi}^{t_n}(u_n)\in \wh{X}. \]
It follows that there exists a $\mathbb{Z}$-valued function $f$, defined for every $n$ large enough, such that 
\begin{equation} \label{eq:funf}
r_{f(n)}(y)=r_n(x)+t_n+s\quad \mbox{and} \quad w_{f(n)}(y)=w_n(x)g.
\end{equation}
By \eqref{eq:r_n} we get
\[ r_{f(n+1)}(y)-r_{f(n)}(y)=(r_{n+1}(x)+t_{n+1}+s)-(r_n(x)+t_n+s) \]
\[ =(r_{n+1}(x)-r_{n}(x))+(t_{n+1}-t_{n})> 2\delta-2\delta=0.\]
Since $r_n(y)$ is monotonically increasing in $n$, we conclude that also the function $f$ is.

We write $\wh{y}_+=g^{-1}\wh{x}_+$
and invert in all of the above the roles of $x$ and $y$, where $g^{-1}$ takes the role of $g$.
Since the strongly stable equivalence relation is invariant under the actions of $\Gamma$ and $\wt{\phi}$, applying $g^{-1}$ and $\wt{\phi}^{-s}$ to \eqref{eq:sstab}, we find that $\wt{\phi}^{-s}(g^{-1}\wh{x)}$ and  $\wh{y}$ are strongly stably equivalent, thus $-s$ takes the role of $s$.
We get for every $n$ large enough $t'_n\in \mathbb{R}$ satisfying $|t'_n|<\delta$ and a $\mathbb{Z}$-valued strictly increasing function $f'$
so that the analogue of \eqref{eq:funf} holds, that is
\begin{equation} \label{eq:funf'}
r_{f'(n)}(x)=r_n(y)+t'_n-s \quad \mbox{and} \quad w_{f'(n)}(x)=w_n(y)g^{-1}.
\end{equation}
Substituting \eqref{eq:funf} in \eqref{eq:funf'}, we get for every $n$ large enough 
\[ r_{f'\circ f(n)}(x)=r_{f(n)}(y)+t'_{f(n)}-s=r_n(x)+t_n+t'_{f(n)}\]
Since $|t_n+t'_{f(n)}|<2\delta$, we conclude by \eqref{eq:r>eps} that $f'\circ f(n)=n$.
Similarly, we get that $f\circ f'(n)=n$.
Since the functions $f$ and $f'$ are strictly increasing, we deduce that there exists $k\in \bbZ$ such that $f(n)=n+k$ and $f'(n)=n-k$.
We obtain from \eqref{eq:funf} that $w_{n+k}(y)=w_n(x)g$.
\end{proof}

As a compact $\Gamma$-space the geometric boundary $\partial\wt{M}$  
can be identified with the Gromov boundary $\partial\Gamma$ of $\Gamma$.
Under this identification, a sequence $(g_n)_{n\in\bbN}$ in $\Gamma$ convergences to a point $\xi\in\partial\wt{M}$
if $g_nu\to\xi$ for $u\in\wt{M}$. This convergence is uniform on compact subsets of $\wt{M}$.

\begin{lemma}\label{L:wn-lim}
For $m$-a.e. $x\in X$ one has
    \[
        \lim_{n\to \infty} w_n(x)^{-1} =\wh{x}_+,
        \qquad 
        \lim_{n\to-\infty} w_n(x)^{-1} =\wh{x}_-,
    \]
\end{lemma}

\begin{proof}
We will compute the first limit, the second one is done similarly.
By \eqref{e:wn-phirn}, we have $w_n(x)\wt{\phi}^{r_n(x)}(\wh{x})\in \wh{X}\subset B(\wh{p},\alpha)$,
thus 
\[ d_{T^1\wt{M}}(w_n(x)^{-1}\wh{p},\wt{\phi}^{r_n(x)}(\wh{x}))=d_{T^1\wt{M}}(\wh{p},w_n(x)\wt{\phi}^{r_n(x)}(\wh{x})) < \alpha\]
Recall that the natural map $\pi:T^1\wt{M}\to \wt{M}$ is contracting with respect to the Sasaki metric, therefore
\[ d_{\wt{M}}(w_n(x)^{-1}\pi(\wh{p}),\pi(\wt{\phi}^{r_n(x)}(\wh{x}))) < \alpha.\]
Since $r_n(x)\to \infty$ as $n\to \infty$, we have 
\[ \lim_{n\to \infty} \pi(\wt{\phi}^{r_n(x)}(\wh{x}))=\lim_{t\to \infty} \pi(\wt{\phi}^{t}(\wh{x}))= \wh{x}_+\]
and we conclude that $w_n(x)^{-1}\pi(\wh{p})\to \wh{x}_+$, thus indeed, $w_n(x)^{-1}\to \wh{x}_+$.
\end{proof}

\subsection{Proof of Theorem~\ref{T:nc-ApaficGreg}}

    In this subsection we let $(M,g)$ be a negatively curved compact Riemannian manifold,
    we let $\mathscr{F}$ be a fundamental domain for the action of $\pi_1(M)$ on $T^1\wt{M}$ which has a non-empty interior and we let $\mu$ be a Gibbs measure on $T^1M$.
We will let $(X,\calX,m,T,w,\Gamma)$ be the Greg constructed in \S\ref{sec:goodchoice} and we argue to show that it is Apafic, thus proving Theorem~\ref{T:nc-ApaficGreg}.

We note that by Lemma~\ref{lem:Gibbs}, the measure $\mu$ satisfies the local product property, see Definition~\ref{def:localprod}.
The proof below will only rely on this property, see Remark~\ref{rem:gibbs}.

Recall that in \S\ref{subse:RW} we fixed a fully supported probability measure $m^1_\Gamma$ on $\Gamma$ and constructed the Lebesgue isomorphism
$j:(X\times \Gamma,m\times m^1_\Gamma) \overto{\cong} (\Gamma^\bbZ,\wt{m}^1)$.
By the constructions of \S\ref{subs:idealpastfuture} we get the following diagram of quotients:
\begin{equation}
    \begin{tikzcd}
        &(\Gamma^\bbZ,\wt{m}^1) \ar[dd, "\beta"] 
        \ar[dl, bend right] 
        \ar[dr, bend left]& \\
        (\Gamma^{\bbZ_-},\wt{m}^1_-) \ar[d,"\beta_-"]& &  
        (\Gamma^{\bbZ_+},\wt{m}^1_+) \ar[d,"\beta_+"]\\
        (E_-,\eta_-) & (E,\eta) \ar[l, "\pr_-"'] \ar[r, "\pr_+"] & (E_+,\eta_+).
    \end{tikzcd}
\end{equation}

Somewhat similarly, we have the following geometrically defined, 
$\Gamma$-equivariant and $\phi^\bbR$-invariant, maps:
\begin{equation} \label{eq:geom}
      \begin{tikzcd}
        & T^1\wt{M} \ar[d, "\pr_{\bowtie}"] \ar[dl, bend right,"u\mapsto u_-"'] 
        \ar[dr, bend left, "u\mapsto u_+"]& \\
        \partial\wt{M} & \partial^2\wt{M} \ar[l, "\pr_1"'] \ar[r, "\pr_2"] 
        &\partial\wt{M}
    \end{tikzcd}  
\end{equation}
Here the maps are the projections:
\[
    \pr_{\bowtie}:u\mapsto (u_-,u_+),\qquad
    \pr_1:(u_-,u_+)\mapsto u_-,\qquad
    \pr_2:(u_-,u_+)\mapsto u_+.
\]
We consider the map
\[ X\times \Gamma \to T^1\wt{M}, \quad (x,g) \mapsto g\wh{x} \]
and denote by $\bar{m}_1$ the corresponding push forward of the measure $m\times m^1_\Gamma$ on $X\times \Gamma$. 
We denote by $\ol{m}^1_\pm$, $\ol{m}^1_{\bowtie}$ 
the probability measures on $\partial\wt{M}$, $\partial^2\wt{M}$,
obtained by the push forwards of $\bar{m}^1$ by the maps in the diagram \eqref{eq:geom}.

\begin{prop}
    There are $\Gamma$-equivariant measurable maps $\sigma$, $\tau$, $\tau_-$ and $\tau_+$ which are all measure space isomorphisms and such that the following diagram commutes
    \[
        \begin{tikzcd}
         (\Gamma^\bbZ,\wt{m}^1) \ar[d, dashed, "\sigma", "\simeq"'] \ar[rr, bend left=20, "\beta"]
         & (E_-,\eta_-)\ar[d, dashed, "\tau_-", "\simeq"']  & (E,\eta)
         \ar[d, dashed, "\tau", "\simeq"'] \ar[l, "\pr_-"']\ar[r, "\pr_+"]& (E_+,\eta_+)\ar[d, dashed, "\tau_+", "\simeq"']\\
         (T^1\wt{M},\bar{m}_1) \ar[rr, bend right=20, "\pr_{\bowtie}"] & (\partial\wt{M},\ol{m}^1_-) & 
         (\partial^2\wt{M},\ol{m}^1_{\bowtie}) 
         \ar[r, "\pr_2"] \ar[l, "\pr_1"'] &
         (\partial\wt{M},\ol{m}^1_+).
        \end{tikzcd}
    \]
\end{prop}

\begin{remark} \label{rem:X'}
    In the proof below, the most subtle part is showing that the maps $\tau_+$ and $\tau_-$, which are easily constructed (Lemma~\ref{L:wn-lim}), are actually isomorphisms.
    This is where using the Greg $(X',\calX',m',T',w',\Gamma)$ from \S\ref{sec:goodchoice} is insufficient and the reason for us to use Lemma~\ref{L:good_radius} from \S\ref{sec:prelem} and construct $(X,\calX,m,T,w,\Gamma)$.
    The technicalities are hidden in the proof of Lemma~\ref{lem:katlasst} and the lemmas preceding it.
\end{remark}

\begin{proof}
    We identify both $(\Gamma^\bbZ,\wt{m}^1)$ and $(T^1\wt{M},\bar{m}_1)$ with $(X\times \Gamma,m\times m^1_\Gamma)$.
    This provides the isomorphism $\sigma$.
    The map
    \[ X\times \Gamma\overto{} \partial^2\wt{M}, \quad 
    (x,g) \mapsto (g\wh{x}_-,g\wh{x}_+)\] 
    coincides with the maps $\pr_{\bowtie}$ and $\pr_{\bowtie}\circ \sigma$ under these identifications of.
    Considered as a map from $\Gamma^\bbZ$, by Lemma~\ref{L:wn-lim}, it is given explicitly by 
    \[ (w_n(x)g^{-1})_{n\in \mathbb{Z}} \mapsto g.\left(\lim_{n\to-\infty}w_n(x)^{-1},\lim_{n\to+\infty}w_n(x)^{-1}\right). \]
 Since the limits are shift invariant, the map $\pr_{\bowtie}\circ \sigma$ factors through 
    the space $(E,\eta)$ of $\bbZ$-ergodic components of $(\Gamma^\bbZ,\wt{m}^1)$, giving rise to the map $\tau$.

    We claim that $\tau$ is a measure space isomorphism.
    Indeed, assume 
    \[
        \tau\left((w_n(x)g^{-1})_{n\in\bbZ}\right)
        =\tau\left((w_n(y)h^{-1})_{n\in\bbZ}\right).
    \]
    Then $g\wh{x}_\pm=h\wh{y}_\pm$, and therefore $\wh{y}_\pm=(h^{-1}g)\wh{x}_\pm$.
    Thus $h^{-1}g$ maps the unparametrized oriented geodesic line $(\wh{x}_-,\wh{x}_+)$ to the unparametrized oriented geodesic line $(\wh{y}_-,\wh{y}_+)$.
    It follows that there exists a time shift $s\in\bbR$ such that for every $t$ $(h^{-1}g) \wt{\phi}^{t+s}(\wh{x})=\wt{\phi}^{t}(\wh{y})$.
    In particular, $\wh{y}=(h^{-1}g)\wt{\phi}^s(\wh{x})$.
    This implies that there exists $k\in\bbZ$ so that 
    \[
        s=r_k(x), \qquad h^{-1}g=w_k(x),\qquad y=T^kx.
    \]
    Therefore 
    \[
        w_{n+k}(x)g^{-1}=w_n(T^kx)w_k(x)g^{-1}=w_{n}(y)h^{-1}
        \qquad(n\in\bbZ)
    \]
    which means that such two points are related by a
    $\bbZ$-shift. 
    Hence the measure space map $\tau:(E,\eta)\to(\partial\wt{M},\ol{m}^1)$
    is an isomorphism as claimed.

    In a similar fashion to the construction of $\tau$, Lemma~\ref{L:wn-lim} is used to construct $\tau_+$ and $\tau_-$.
    Indeed, the maps
        $(\Gamma^{\bbZ_\pm},\wt{m}_\pm)\overto{}\partial\wt{M}$:
    \[
        \left(w_n(x)^{-1}g^{-1}\right)_{n\in\bbZ_\pm}\mapsto g.\lim_{n\to\pm\infty} w_n(x)^{-1}
        =g\wh{x}_\pm
    \]
    are shift invariant, thus factor via $(E_\pm,\eta_\pm)$.
    We are left to show that $\tau_\pm$ are measure space isomorphisms.
    We will do this for $\tau_+$, the other case being similar.

    Suppose 
    \[
        \tau_+\left((w_n(x)g^{-1})_{n\in\bbZ_+}\right)
        =\tau_+\left((w_n(y)h^{-1})_{n\in\bbZ_+}\right).
    \]
    That is, assume $g\wh{x}_+=h\wh{y}_+$ for $x,y\in X$, $g,h\in\Gamma$.
Then $\wh{x}_+=g^{-1}h\wh{y}_+$, thus by Lemma~\ref{lem:katlasst} there exists $k\in \mathbb{Z}$ such that for every $n$ big enough, $w_{n+k}(y)=w_n(x)g^{-1}h$.
It follows that the two sequences $(w_n(x)g^{-1})_{n\in\bbZ_+}$
    and $(w_n(y)h^{-1})_{n\in\bbZ_+}$ are equivalent
    by the $\bbZ_+$-action on $\Gamma^{\bbZ_+}$, and so define 
    the same point in $E_+$. This proves that $\tau_+$ is an isomorphism
    and completes the proof.
\end{proof}

\begin{lemma}
    We have equivalence of measures 
    \[
        \ol{m}^1_{\bowtie}\sim \ol{\mu}^1_{\bowtie}
        \quad\textrm{on}\quad \partial^2\wt{M},
        \qquad 
        \ol{m}^1_{\pm}\sim \ol{\mu}^1_{\pm}
        \quad\textrm{on}\quad\partial\wt{M}.
    \]
\end{lemma}
\begin{proof}
    Recall that $\ol{m}^1_{\bowtie}$ and $\ol{\mu}^1_{\bowtie}$
    are obtained from $\wt{m}^1$ and $\wt{\mu}^1$, respectively, 
    via the pushforward by the map $p_{\bowtie}:T^1\wt{M}\to \partial^2\wt{M}$.
    Consider the convolution $\wt{m}^1_{[0,\epsilon_0]}$ of 
    the measure $\wt{m}^1$ on $\wt{X}\subset T^1\wt{M}$
    with the normalized Lebesgue measure on $[0,\epsilon_0]$ via $\wt{\phi}^t$, defined by
    \[
        \int_{T^1\wt{M}} f\dd \wt{m}^1_{[0,\epsilon_0]}
        =\frac{1}{\epsilon_0}\int_0^{\epsilon_0}\int_{\wt{X}}f(\wt{\phi}^t(u))\dd \wt{m}^1(u)\dd t
        \qquad(f\in C_c(T^1\wt{M}).
    \]
    Then $\wt{m}^1_{[0,\epsilon_0]}$ is absolutely continuous with respect to 
    $\wt{\mu}^1$, yet its pushforward via $p_{\bowtie}$ coincides with that of
    $\wt{m}^1$. This shows 
    \[
        \ol{m}^1_{\bowtie}=(p_{\bowtie})_*\wt{m}^1=(p_{\bowtie})_*\wt{m}^1_{[0,\epsilon_0]}
        \preceq (p_{\bowtie})_*\wt{\mu}^1=\ol{\mu}^1_{\bowtie}.
    \]
    Now take some probability distribution $\theta$ on $[0,\infty)$
    with everywhere positive density and define the convolution
    $\wt{m}^1_{\theta}$ in a similar fashion.
    Then we get $\wt{\mu}^1\preceq \wt{m}^1_{\theta}$ and comparing the pushforward
    to $\partial^2\wt{M}$ using $p_{\bowtie}$, 
    we get $\ol{\mu}^1_{\bowtie}\preceq \ol{m}^1_{\bowtie}$.
    Therefore $\ol{m}^1_{\bowtie}\sim \ol{\mu}^1_{\bowtie}$.
    The other cases follow similarly.
\end{proof}
We have now established measure space isomorphisms and equality of measure classes
\[
    \begin{split}
        &(E,\eta)\cong 
    (\partial^2\wt{M},\ol{m}^1_{\bowtie})\sim(\partial^2\wt{M},\ol{\mu}^1_{\bowtie}),\\
        &(E_\pm, \eta_\pm) \cong (\partial\wt{M},\ol{m}^1_\pm)\sim
        (\partial\wt{M},\ol{\mu}^1_\pm).
    \end{split}
\]
Therefore the local product property of $\mu$, 
$\ol{\mu}^1_{\bowtie}\sim \ol{\mu}^1_{-}\times\ol{\mu}^1_{+}$,
becomes the Apafic condition 
$\eta\sim \eta_-\times \eta_+$ on the Greg $(X,\calX,m,T,w,\Gamma)$.
This completes the proof of Theorem~\ref{T:nc-ApaficGreg}.


\bibliographystyle{acm}
\bibliography{bibtexrefs} 

\begin{bibdiv}
\begin{biblist}

\bib{Ambrose}{article}{
      author={Ambrose, Warren},
       title={Representation of ergodic flows},
        date={1941},
        ISSN={0003-486X},
     journal={Ann. of Math. (2)},
      volume={42},
       pages={723\ndash 739},
         url={https://doi.org/10.2307/1969259},
      review={\MR{4730}},
}

\bib{Ambrose-Kakutani}{article}{
      author={Ambrose, Warren},
      author={Kakutani, Shizuo},
       title={Structure and continuity of measurable flows},
        date={1942},
        ISSN={0012-7094,1547-7398},
     journal={Duke Math. J.},
      volume={9},
       pages={25\ndash 42},
         url={http://projecteuclid.org/euclid.dmj/1077493068},
      review={\MR{5800}},
}

\bib{AEV}{article}{
      author={Avila, Artur},
      author={Eskin, Alex},
      author={Viana, Marcelo},
       title={Continuity of the lyapunov exponents of random matrix products},
        date={2023},
      eprint={2305.06009},
         url={https://arxiv.org/abs/2305.06009},
}

\bib{AV}{article}{
      author={Avila, Artur},
      author={Viana, Marcelo},
       title={Extremal lyapunov exponents: an invariance principle and
  applications},
        date={2010},
     journal={Invent. Math.},
      volume={181},
      number={1},
       pages={115\ndash 189},
}

\bib{AV2}{article}{
      author={Avila, Artur},
      author={Viana, Marcelo},
       title={Simplicity of lyapunov spectra: a sufficient criterion},
        date={2007},
     journal={Port. Math. (N.S.)},
      volume={64},
      number={3},
       pages={311\ndash 376},
}

\bib{AV3}{article}{
      author={Avila, Artur},
      author={Viana, Marcelo},
       title={Simplicity of lyapunov spectra: proof of the zorich-kontsevich
  conjecture},
        date={2007},
     journal={Acta Math.},
      volume={198},
      number={1},
       pages={1\ndash 56},
}

\bib{BDL}{article}{
      author={Bader, Uri},
      author={Duchesne, Bruno},
      author={L\'{e}cureux, Jean},
       title={Almost algebraic actions of algebraic groups and applications to
  algebraic representations},
        date={2017},
        ISSN={1661-7207},
     journal={Groups Geom. Dyn.},
      volume={11},
      number={2},
       pages={705\ndash 738},
      review={\MR{3668057}},
}

\bib{BF:products}{article}{
      author={Bader, Uri},
      author={Furman, Alex},
       title={Super-rigidity and non-linearity for lattices in products},
        date={2020},
        ISSN={0010-437X},
     journal={Compos. Math.},
      volume={156},
      number={1},
       pages={158\ndash 178},
      review={\MR{4036451}},
}

\bib{BF:AREA}{article}{
      author={Bader, Uri},
      author={Furman, Alex},
       title={Algebraic representations of ergodic actions and super-rigidity},
        date={2014},
       pages={1\ndash 25},
      eprint={https://arxiv.org/abs/1311.3696},
}

\bib{BF:icm}{inproceedings}{
      author={Bader, Uri},
      author={Furman, Alex},
       title={Boundaries, rigidity of representations, and {L}yapunov
  exponents},
        date={2014},
   booktitle={Proceedings of the {I}nternational {C}ongress of
  {M}athematicians---{S}eoul 2014. {V}ol. {III}},
   publisher={Kyung Moon Sa, Seoul},
       pages={71\ndash 96},
      review={\MR{3729019}},
}

\bib{BF:examples}{article}{
      author={Bader, Uri},
      author={Furman, Alex},
       title={Simplicity of the lyapunov spectrum via boundary theory, ii -
  applications},
        date={2021},
      eprint={https://arxiv.org/abs/???},
}

\bib{BF:Margulis}{incollection}{
      author={Bader, Uri},
      author={Furman, Alex},
       title={An extension of {M}argulis's superrigidity theorem},
        date={[2022] \copyright 2022},
   booktitle={Dynamics, geometry, number theory---the impact of {M}argulis on
  modern mathematics},
   publisher={Univ. Chicago Press, Chicago, IL},
       pages={47\ndash 65},
      review={\MR{4422051}},
}

\bib{BQ:book}{book}{
      author={Benoist, Yves},
      author={Quint, Jean-Fran\c~cois},
       title={Random walks on reductive groups},
      series={Ergebnisse der Mathematik und ihrer Grenzgebiete. 3. Folge. A
  Series of Modern Surveys in Mathematics [Results in Mathematics and Related
  Areas. 3rd Series. A Series of Modern Surveys in Mathematics]},
   publisher={Springer, Cham},
        date={2016},
      volume={62},
        ISBN={978-3-319-47719-0; 978-3-319-47721-3},
      review={\MR{3560700}},
}

\bib{BL}{book}{
      author={Bougerol, Philippe},
      author={Lacroix, Jean},
       title={Products of random matrices with applications to {S}chr\"odinger
  operators},
      series={Progress in Probability and Statistics},
   publisher={Birkh\"auser Boston, Inc., Boston, MA},
        date={1985},
      volume={8},
        ISBN={0-8176-3324-3},
         url={https://doi.org/10.1007/978-1-4684-9172-2},
      review={\MR{886674}},
}

\bib{Bowen-Anosov-flows}{inproceedings}{
      author={Bowen, Rufus},
       title={Symbolic dynamics for hyperbolic flows},
        date={1975},
   booktitle={Proceedings of the {I}nternational {C}ongress of {M}athematicians
  ({V}ancouver, {B}.{C}., 1974), {V}ol. 2},
   publisher={Canad. Math. Congr., Montreal, QC},
       pages={299\ndash 302},
      review={\MR{426059}},
}

\bib{Bowen-book}{book}{
      author={Bowen, Rufus},
      editor={Chazottes, Jean-Ren\'e},
       title={Equilibrium states and the ergodic theory of {A}nosov
  diffeomorphisms},
     edition={revised},
      series={Lecture Notes in Mathematics},
   publisher={Springer-Verlag, Berlin},
        date={2008},
      volume={470},
        ISBN={978-3-540-77605-5},
        note={With a preface by David Ruelle},
      review={\MR{2423393}},
}

\bib{CHERNOV}{incollection}{
      author={Chernov, N.},
       title={Chapter 4 invariant measures for hyperbolic dynamical systems},
        date={2002},
      editor={Hasselblatt, B.},
      editor={Katok, A.},
      series={Handbook of Dynamical Systems},
      volume={1},
   publisher={Elsevier Science},
       pages={321\ndash 407},
  url={https://www.sciencedirect.com/science/article/pii/S1874575X02800066},
}

\bib{YCor}{article}{
      author={de~Cornulier, Yves},
       title={On lengths on semisimple groups},
        date={2009},
        ISSN={1793-5253},
     journal={J. Topol. Anal.},
      volume={1},
      number={2},
       pages={113\ndash 121},
         url={https://doi.org/10.1142/S1793525309000102},
      review={\MR{2541757}},
}

\bib{Filip:MET}{article}{
      author={Filip, Simion},
       title={Notes on the multiplicative ergodic theorem},
        date={2019},
        ISSN={0143-3857},
     journal={Ergodic Theory Dynam. Systems},
      volume={39},
      number={5},
       pages={1153\ndash 1189},
      review={\MR{3928611}},
}

\bib{F:RT}{incollection}{
      author={Furman, Alex},
       title={Random walks on groups and random transformations},
        date={2002},
   booktitle={Handbook of dynamical systems, {V}ol.\ 1{A}},
   publisher={North-Holland, Amsterdam},
       pages={931\ndash 1014},
         url={https://doi.org/10.1016/S1874-575X(02)80014-5},
      review={\MR{1928529}},
}

\bib{Furstenberg-Poisson}{article}{
      author={Furstenberg, Harry},
       title={A poisson formula for semisimple lie groups},
        date={1963},
     journal={Ann. of Math. (2)},
      volume={77},
       pages={335\ndash 386},
}

\bib{Fur-Kif}{article}{
      author={Furstenberg, Hillel},
      author={Kifer, Yuri},
       title={Random matrix products and measures on projective spaces},
        date={1983},
        ISSN={0021-2172},
     journal={Israel J. Math.},
      volume={46},
      number={1-2},
       pages={12\ndash 32},
      review={\MR{0727020}},
}

\bib{GM}{article}{
      author={Gol{\cprime}dshe{\u\i}d, I.~Ya.},
      author={Margulis, G.~A.},
       title={Lyapunov exponents of a product of random matrices},
    language={Russian},
        date={1989},
        ISSN={0036-0279},
     journal={Uspekhi Mat. Nauk},
      volume={44},
      number={5(269)},
       pages={13\ndash 60},
}

\bib{GR}{article}{
      author={Guivarc'h, Y.},
      author={Raugi, A.},
       title={Fronti\`ere de {F}urstenberg, propri\'et\'es de contraction et
  th\'eor\`emes de convergence},
        date={1985},
        ISSN={0044-3719},
     journal={Z. Wahrsch. Verw. Gebiete},
      volume={69},
      number={2},
       pages={187\ndash 242},
         url={https://doi.org/10.1007/BF02450281},
      review={\MR{779457}},
}

\bib{GR2}{incollection}{
      author={Guivarc'h, Y.},
      author={Raugi, A.},
       title={Products of random matrices: convergence theorems},
        date={1986},
   booktitle={Random matrices and their applications ({B}runswick, {M}aine,
  1984)},
      series={Contemp. Math.},
      volume={50},
   publisher={Amer. Math. Soc., Providence, RI},
       pages={31\ndash 54},
         url={https://doi.org/10.1090/conm/050/841080},
      review={\MR{841080}},
}

\bib{Kaimanovich:DE}{article}{
      author={Kaimanovich, V.~A.},
       title={Double ergodicity of the poisson boundary and applications to
  bounded cohomology},
        date={2003},
     journal={Geom. Funct. Anal.},
      volume={13},
      number={4},
       pages={852\ndash 861},
}

\bib{Kaim-geo}{incollection}{
      author={Kaimanovich, Vadim~A.},
       title={Invariant measures of the geodesic flow and measures at infinity
  on negatively curved manifolds},
        date={1990},
      volume={53},
       pages={361\ndash 393},
         url={http://www.numdam.org/item?id=AIHPA_1990__53_4_361_0},
        note={Hyperbolic behaviour of dynamical systems (Paris, 1990)},
      review={\MR{1096098}},
}

\bib{KM}{article}{
      author={Karlsson, Anders},
      author={Margulis, Gregory~A.},
       title={A multiplicative ergodic theorem and nonpositively curved
  spaces},
        date={1999},
     journal={Comm. Math. Phys.},
      volume={208},
      number={1},
       pages={107\ndash 123},
}

\bib{Kostant}{article}{
      author={Kostant, Bertram},
       title={On convexity, the {W}eyl group and the {I}wasawa decomposition},
        date={1973},
        ISSN={0012-9593},
     journal={Ann. Sci. \'{E}cole Norm. Sup. (4)},
      volume={6},
       pages={413\ndash 455 (1974)},
      review={\MR{364552}},
}

\bib{LePage}{article}{
      author={Le~Page, \'{E}mile},
       title={R\'{e}gularit\'{e} du plus grand exposant caract\'{e}ristique des
  produits de matrices al\'{e}atoires ind\'{e}pendantes et applications},
    language={French, with English summary},
        date={1989},
        ISSN={0246-0203},
     journal={Ann. Inst. H. Poincar\'{e} Probab. Statist.},
      volume={25},
      number={2},
       pages={109\ndash 142},
      review={\MR{1001021}},
}

\bib{Ledrappier}{incollection}{
      author={Ledrappier, F.},
       title={Positivity of the exponent for stationary sequences of matrices},
        date={1986},
   booktitle={Lyapunov exponents ({B}remen, 1984)},
      series={Lecture Notes in Math.},
      volume={1186},
   publisher={Springer, Berlin},
       pages={56\ndash 73},
         url={https://doi.org/10.1007/BFb0076833},
      review={\MR{850070}},
}

\bib{Richardson}{article}{
      author={Richardson, R.~W., Jr.},
       title={A rigidity theorem for subalgebras of {L}ie and associative
  algebras},
        date={1967},
        ISSN={0019-2082},
     journal={Illinois J. Math.},
      volume={11},
       pages={92\ndash 110},
         url={http://projecteuclid.org/euclid.ijm/1256054787},
      review={\MR{206170}},
}

\bib{Viana}{book}{
      author={Viana, M.},
       title={Lectures on lyapunov exponents},
      series={Cambridge Studies in Advanced Mathematics},
   publisher={Cambridge University Press, Cambridge},
      volume={145},
        ISBN={978-1-107-08173-4},
      review={\MR{3289050}},
}

\bib{zimmer-book}{book}{
      author={Zimmer, R.~J.},
       title={Ergodic theory and semisimple groups},
      series={Monographs in Mathematics},
   publisher={Birkh\"auser Verlag},
      volume={81},
}

\end{biblist}
\end{bibdiv}

\end{document}